\documentclass[11pt,letterpaper]{amsart}
\usepackage{amsmath, amscd, amssymb}
\usepackage[enableskew,vcentermath]{youngtab}
\usepackage{ytableau}
\usepackage[mathscr]{eucal}
\usepackage{soul} 
\usepackage[all,cmtip]{xy}
\usepackage{tikz-cd}
\usepackage{xfrac}

\usepackage{geometry}
\usepackage{amsfonts}
\usepackage[backref]{hyperref}
\usepackage{graphicx}
\usepackage{cite}

\numberwithin{equation}{section}

\theoremstyle{definition}
\newtheorem{thm}{Theorem}[section]
\newtheorem{theorem}[thm]{Theorem}

\newtheorem{lemma}[thm]{Lemma}
\newtheorem{corollary}[thm]{Corollary}
\newtheorem{proposition}[thm]{Proposition}

\newtheorem*{SYZ}{ Strominger--Yau--Zaslow Mirror Symmetry}
\newtheorem*{TMS}{Topological Mirror Symmetry}

\newtheorem{remark}[thm]{Remark}

\newtheorem{definition}[thm]{Definition}
\newtheorem*{claim}{Claim}
\newtheorem{assumption}[thm]{Assumption}

\newtheorem{example}[thm]{Example}

\newtheorem{condition}{Condition}

\newtheorem{defn-thm}[thm]{Definition-Theorem}

\newtheorem{Notation}[thm]{Notation}

\newenvironment{observe}{\noindent\textcolor{blue}{\textit{Observation}}.}{\hfill \textcolor{blue}{$\blacktriangleleft$}\par}
\newtheorem*{theorem*}{Theorem}
\newtheorem*{proposition*}{Proposition}

\usepackage{todonotes}
\definecolor{pistachio}{rgb}{0.58, 0.77, 0.45}
\definecolor{eggshell}{rgb}{0.94, 0.92, 0.84}

\newcommand{\sA}{{\mathcal A}}

\newcommand{\sE}{{\mathcal E}}
\newcommand{\sF}{{\mathcal F}}

\newcommand{\sH}{{\mathcal H}}

\newcommand{\sK}{{\mathcal K}}
\newcommand{\sL}{{\mathcal L}}

\newcommand{\sN}{{\mathcal N}}
\newcommand{\sO}{{\mathcal O}}
\newcommand{\sP}{{\mathcal P}}
\newcommand{\sQ}{{\mathcal Q}}
\newcommand{\sR}{{\mathcal R}}

\newcommand{\sV}{{\mathcal V}}

\newcommand{\g}{{\mathfrak g}}

\newcommand{\gc}{\mathfrak{c}}
\newcommand{\gp}{\mathfrak{p}}

\newcommand{\gn}{\mathfrak{n}}
\newcommand{\gl}{\mathfrak{l}}

\newcommand{\D}{{\mathbb D}}

\newcommand{\N}{{\mathbb N}}

\newcommand{\Z}{{\mathbb Z}}

\newcommand{\GL}{\operatorname{GL}}
\newcommand{\SL}{\operatorname{SL}}
\newcommand{\Sp}{\operatorname{Sp}}
\newcommand{\PSp}{\operatorname{PSp}}
\newcommand{\SO}{\operatorname{SO}}
\newcommand{\OG}{\operatorname{OG}}
\newcommand{\Spin}{\operatorname{Spin}}

\newcommand{\Prym}{\operatorname{\mathbf{Prym}}}
\newcommand{\LG}{\operatorname{LG}}
\newcommand{\KL}{\operatorname{KL}}

\newcommand{\Ker}{\operatorname{Ker}}

\renewcommand{\Im}{\operatorname{Im}}
\newcommand{\Coker}{\operatorname{Coker}}
\newcommand{\Tr}{\operatorname{Tr}}
\newcommand{\Nm}{\operatorname{Nm}}
\newcommand{\Div}{\operatorname{Div}}
\newcommand{\Pic}{\operatorname{Pic}}
\newcommand{\Br}{\operatorname{Br}}
\newcommand{\Spec}{\operatorname{Spec}}

\newcommand{\Split}{\operatorname{Split}}

\newcommand{\Res}{\operatorname{Res}}
\newcommand{\Jac}{{\operatorname{Jac}}}
\newcommand{\id}{{\operatorname{id}}}

\newcommand{\codim}{{\operatorname{codim}}}

\newcommand{\Tot}{{\operatorname{Tot}}}

\newcommand{\Fil}{\operatorname{Fil}}

\newcommand{\ad}{\operatorname{ad}}
\newcommand{\End}{\operatorname{End}}
\newcommand{\Ad}{\operatorname{Ad}}

\newcommand{\btheorem}{\begin{theorem}}
\newcommand{\etheorem}{\end{theorem}}
\newcommand{\bproposition}{\begin{proposition}}
\newcommand{\eproposition}{\end{proposition}}
\newcommand{\bdefinition}{\begin{definition}}
\newcommand{\edefinition}{\end{definition}}
\newcommand{\bcorollary}{\begin{corollary}}
\newcommand{\ecorollary}{\end{corollary}}
\newcommand{\bproof}{\begin{proof}}
\newcommand{\eproof}{\end{proof}}
\newcommand{\bremark}{\begin{remark}}
\newcommand{\eremark}{\end{remark}}
\newcommand{\eexample}{\end{example}}
\newcommand{\bexample}{\begin{example}}
\newcommand{\elemma}{\end{lemma}}
\newcommand{\blemma}{\begin{lemma}}
\newcommand{\bobserve}{\begin{observe}}
\newcommand{\eobserve}{\end{observe}}

\newcommand{\ord}{\tx{ord}}
\renewcommand{\bar}{\overline}

\renewcommand{\phi}{\varphi}
\newcommand{\ee}{\end{eqnarray*}}
\newcommand{\be}{\begin{eqnarray*}}

\newcommand{\beq}{\begin{equation}}
\newcommand{\eeq}{\end{equation}}
\newcommand{\bd}{\begin{enumerate}}
\newcommand{\ed}{\end{enumerate}}
\newcommand{\bti}{\begin{tikzcd}}
\newcommand{\eti}{\end{tikzcd}}

\newcommand{\rar}[1]{\xrightarrow{#1}}
\renewcommand{\tilde}{\widetilde}

\renewcommand{\sf}[1]{\textsf{#1}}
\renewcommand{\bf}[1]{\mathbf{#1}}
\newcommand{\tx}[1]{\text{#1}}

\newcommand{\PGL}{\operatorname{PGL}}
\newcommand{\rk}{\operatorname{rk}}

\newtheorem{mainthm}{Theorem}

\title[Springer correspondence and mirror symmetries]{Springer correspondence and mirror symmetries for parabolic Hitchin systems}

\author{Bin Wang}
\address{Department of Mathematics, Chinese University of Hong Kong, New Territories, Hong Kong SAR.}
\email{binwang@math.cuhk.edu.hk}

\author{Xueqing Wen}
\address{Chongqing University of Technology, No. 69, Hongguang Avenue, Banan District, Chongqing, 400054, China.}
\email{wenxq@cqut.edu.cn}

\author{Yaoxiong Wen}
\address{School of Mathematics, Korea Institute for Advanced Study, Seoul 02455, Korea.}
\email{y.x.wen.math@gmail.com}

\date{}

\begin{document}

\begin{abstract}
     We establish the Strominger--Yau--Zaslow (SYZ) and topological mirror symmetries for parabolic Hitchin systems of types B and C, providing new insights into their geometric and topological structures. Unlike type A, these cases require a geometric reinterpretation of Springer duality, as the nilpotent orbits for types B and C lie in distinct Lie algebras. Moreover, in contrast to Hitchin's approach in the non-parabolic setting, analyzing the relationship between generic fibers in types B and C necessitates addressing changes in the partitions of Springer dual nilpotent orbits, which presents the central challenge of this work. To address this, we construct new moduli spaces of Higgs bundles associated with nilpotent orbit closures and examine their generic Hitchin fibers. In the Richardson case, we further explore their connection with the generic fibers of parabolic Hitchin systems. Along the way, we uncover deep connections between Springer duality, Kazhdan--Lusztig maps, and the singularities of spectral curves, culminating in a novel geometric interpretation of Lusztig's canonical quotient.
\end{abstract}

\maketitle
\tableofcontents

\section{Introduction}

\subsection{Motivation}
The geometric Langlands program has been a long-standing and rich subject in mathematics. Over the past two decades, this program has undergone significant reinterpretation from a physical perspective through the work of Gukov, Kapustin, and Witten \cite{KW07, GW08, GW10}, sparking considerable interest within the geometric community. In \cite{DP08}, a certain semi-classical limit of Kontsevich’s homological mirror symmetry conjecture \cite{Ko95} led to the equivalence
\begin{align*}
    \mathcal{D}^b( \mathrm{Coh}( \bf{Higgs}_{G}(\Sigma)) \sim  \mathcal{D}^b( \mathrm{Coh}( \bf{Higgs}_{{}^LG}(\Sigma)),
\end{align*}
between the derived categories of coherent sheaves on Hitchin systems over a Riemann surface $\Sigma$ for Langlands dual groups.

A ``topological shadow" of this equivalence, often referred to as \emph{topological mirror symmetry}, was first conjectured by Hausel and Thaddeus \cite[Conjecture 5.1]{HT03} for $G=\SL_n$, and ${}^L G=\PGL_n$. This conjecture has since been proven independently by Groechenig, Wyss, and Ziegler \cite{GWZ20}, as well as by Maulik and Shen \cite{MS21}. More precisely, let $\bf{Higgs}_{\SL_n, L} $ denote the moduli space of stable $\SL_n$-Higgs bundles $(E, \theta)$ over $\Sigma$, where $E$ is a rank $n$ vector bundle with an isomorphism $\mathrm{det}(E) \cong L$, and $\theta \in \mathrm{H}^0(\Sigma, \mathrm{End}(E) \otimes \omega_\Sigma )$ is trace-free. Assuming $d$ is coprime to $n$, the moduli space $ \bf{Higgs}_{\SL_n, L} $ is smooth. Meanwhile, $ \bf{Higgs}_{\PGL_n, L} =  \bf{Higgs}_{\SL_n, L} / \Gamma $, where $\Gamma = \mathrm{Jac}(\Sigma)[n]$ is the subgroup of $n$-th torsion points of Jacobian. The \emph{topological mirror symmetry} concerns the equality of (stringy) Hodge numbers:

\begin{TMS}[\cite{GWZ20, MS21}] \label{conjecture}
Assume $d = \deg L$ and $ d' = \deg L' $ are coprime to $n$, then
\[
h^{p,q}( \bf{Higgs}_{\SL_n, L}) = h^{p,q}_{st}( \bf{Higgs}_{\PGL_n, L'}, \alpha).
\]
Here, $\alpha$ is a natural unitary gerbe on $ \bf{Higgs}_{\PGL_n, L'}$ arising from the existence of the ``universal" family.
\end{TMS}

The proof by Groechenig, Wyss, and Ziegler relies on $p$-adic integration techniques, an approach we adopt in this paper. Alternatively, Maulik and Shen used a sheaf-theoretic method to obtain a more refined identification between the cohomology groups, where the Chen--Ruan cohomology \cite{CR04, Ru03} of $\bf{Higgs}_{\PGL_n, L'}$---a global finite quotient---is taken into account. Furthermore, they demonstrated that this identification preserves perverse filtrations, a key result with implications for the proof of the $P=W$ conjecture (see \cite{CMS1, CMS2, MS22, HMMS, MSY23}).

Despite differences in methodology, both approaches rely on a fundamental geometric property of the Langlands dual Hitchin systems:
\begin{SYZ}[\cite{HT03}]
    The two moduli spaces, along with their Hitchin maps, exhibit Strominger--Yau--Zaslow (SYZ) mirror symmetry, i.e., 
    \[
    \xymatrix{
     \bf{Higgs}_{\SL_n, L} \ar[rd]_{\check{h}} &  & \bf{Higgs}_{\PGL_n, L'} \ar[ld]^{\hat{h}} \\
     & A &
    }.
    \]
    Here, the generic fibers of $\check{h}$ and $\hat{h}$ are torsors over dual abelian varieties.
\end{SYZ}
\begin{remark}
    A more refined condition on the torsor structure, related to specific unitary gerbes, arises in practice but can often be canonically chosen.
\end{remark}

In \cite{GW08}, Gukov and Witten proposed a physical version of the geometric Langlands by introducing the surface operators, which, informally, allow Higgs fields to have simple poles. Motivated by their work, we consider Higgs fields with one fixed simple pole. This leads to the construction of moduli spaces of Higgs bundles (of types B and C) where the residue of the Higgs field lies in a ``good" resolution of a nilpotent orbit closure. We denote these spaces as $\bf{Higgs}_{G,\overline{\bf{O}}}$ and hypothesize that both the \emph{SYZ} and \emph{topological mirror symmetries} extend to these moduli spaces for Langlands dual groups.

A key feature in this generalization is the interaction between adjoint orbits in $\mathfrak{g}$ and ${}^L\mathfrak{g}$. Adjoint orbits, particularly nilpotent ones, are classical objects in representation theory and geometry. Any mirror symmetry phenomenon between $\bf{Higgs}_{G,\overline{\bf{O}}}$ and $\bf{Higgs}_{{}^LG,\overline{\bf{O}'}}$ must build upon a corresponding mirror symmetry between nilpotent orbits $\bf{O} \subset \mathfrak{g}$ and $\bf{O}' \subset {}^L\mathfrak{g}$. This serves as the starting point of our paper.

For type A complex Lie groups ($G=\rm SL_n$ and ${}^LG=\PGL_n$), it is well-known that the nilpotent orbits of $\rm SL_n$ and $\rm PGL_n$ coincide. Additionally, these nilpotent orbits are all Richardson, and their closures admit crepant resolutions (see \cite{CM93} for details). Consequently, both the topological and SYZ mirror symmetries can be formulated analogously, with proofs given in \cite{She18, SWW22, SWW22t, She24}.

For classical groups beyond type A, the most intriguing and nontrivial cases arise in types B and C, where $G=\rm SO_{2n+1}$ and ${}^LG = \rm Sp_{2n}$, as the nilpotent orbits belong to different Lie algebras. This paper focuses on these cases, and we use subscripts ${}_B$ or ${}_C$ to emphasize the type. For instance, $\mathbf{O}_B$ (resp. $\mathbf{O}_C$) denotes a nilpotent orbit of type B (resp. type C).

\subsection{Main results}

In this paper, we focus on the case of a single marked point $x$ for simplicity; the general case follows in a similar manner.

The nilpotent orbit closures, being highly singular, lead to similarly singular moduli spaces when the residue of the Higgs field is constrained to lie in them. However, for any nilpotent orbit $\bf{O}$, the well-known \emph{Jacobson--Morozov resolution} provides a desingularization:
$$
    G \times_{P_{JM}} \mathfrak{n}_{2} \longrightarrow \overline{\bf{O}},
$$
where $\mathfrak{n}_2\subset\mathfrak{g}$ is a subspace associated with an $\mathfrak{sl}_{2}$-triple. Using this resolution, we construct in Section~\ref{Sec:via_JM} new moduli spaces, denoted $\bf{Higgs}_{\overline{\bf{O}}}$, of dimension $(2g-2) \dim G + \dim \overline{\bf{O}}$. It is worth noting that $\bf{Higgs}_{\overline{\bf{O}}_B}$ has two connected components, denoted by $\bf{Higgs}_{\overline{\bf{O}}_B}^{+}$ and $\bf{Higgs}_{\overline{\bf{O}}_B}^{-}$.

To investigate the SYZ mirror symmetry, we begin by studying the Hitchin bases. Our first main result establishes a connection between the Hitchin bases for $\bf{Higgs}_{\overline{\bf{O}}_{B}}$ and $\bf{Higgs}_{\overline{\bf{O}}_{C}}$, yielding a new perspective on Springer duality for special nilpotent orbits:\footnote{See Definition \ref{def:Obar} for a combinatorial description of special orbits and Equation~\eqref{Springer_dual_map} for the Springer duality map.}

\begin{mainthm}[Theorem \ref{Thm:why special}] \label{Thm:intro why special}
    The following are equivalent:
    \begin{enumerate}
        \item The nilpotent orbits $\bf{O}_B$ and $\bf{O}_C$ are special and Springer dual.
        \item The Hitchin bases of $ \bf{Higgs}_{\overline{\bf{O}}_{B}}$ and $ \bf{Higgs}_{\overline{\bf{O}}_{C}}$, denoted as $ \bf{H}_{\overline{\bf{O}}_{B}}$ and $ \bf{H}_{\overline{\bf{O}}_{C}}$, are canonically isomorphic.
    \end{enumerate}
\end{mainthm}

For Springer dual special orbits $\overline{\bf{O}}_B$, $\overline{\bf{O}}_C$, let $\bf{H}$ denote the Hitchin base. The relationship between the moduli spaces can be summarized as follows:
\begin{equation*}
    \begin{tikzcd}
    \bf{Higgs}_{\overline{\bf{O}}_B} \ar[rd, swap, "h_{\overline{\bf{O}}_B}"]  & &  \bf{Higgs}_{ \overline{\bf{O}}_C} \ar[ld, "h_{\overline{\bf{O}}_C}"] \\
     &  \bf{H}  &
    \end{tikzcd}.
\end{equation*}
By Corollary \ref{half dimension}, we have the dimension of the Hitchin base given by
\[
\dim \bf{H}=\dfrac{1}{2}\dim \bf{Higgs}_{\overline{\bf{O}}_B}=\dfrac{1}{2}\dim \bf{Higgs}_{\overline{\bf{O}}_C}.
\]
At first glance, this suggests a natural mirror symmetry relationship (via TMS and SYZ) between $\bf{Higgs}_{\overline{\bf{O}}_B}$ and $\bf{Higgs}_{\overline{\bf{O}}_C}$. However, this assertion is \emph{WRONG} in general, as we will demonstrate in Corollary~\ref{intro_not dual}. The failure stems from a subtle connection to Lusztig's canonical quotient. To remedy this issue, it is necessary to consider certain generically finite ``covers". We will provide a detailed explanation in Proposition~\ref{prop:intro dual abelian}. In brief, these covers are constructed via the usual moduli spaces of parabolic Higgs bundles using the following generically finite map.

We say an orbit $\bf{O}_{R} \subset \g$ \emph{Richardson} if there exists a parabolic subgroup $P < G$ such that the image of the moment map (also called \emph{generalized Springer map})
\begin{align}\label{Ap:pol_R}
    T^*(G/P) \longrightarrow \overline{\bf{O}}_R,
\end{align}
is the closure of $\bf{O}_R$. The parabolic subgroup $P$ is then called a \emph{polarization} of the Richardson orbit $\bf{O}_R$. Note that all Richardson orbits are special. In type A, this map is always crepant, but in general, it is only generically finite.

Let $\bf{O}_{B, R}$ and $\bf{O}_{C, R}$ denote Springer dual Richardson orbits, with $P_{B}$ and $P_{C} $ denoting dual polarizations (see Definition~\ref{def:dual_pol}). We construct moduli spaces of parabolic Higgs bundles $\bf{Higgs}_{ P_{C}}$ and $\bf{Higgs}_{ P_{B}}$, where $\bf{Higgs}_{P_B}$ has two connected components: $ \bf{Higgs}_{P_B}^{+} 
\sqcup \bf{Higgs}_{ P_B}^{-}$. 

Let $\bf{H}_{P_{B}}$ and $\bf{H}_{P_{C}}$ denote the Hitchin bases of these spaces, as studied in \cite{BK18}. Combining this with Theorem \ref{Thm:intro why special}, we find that:
$$
\bf{H}_{P_{B}}=\bf{H}_{\overline{\bf{O}}_{B,R}}=\bf{H}_{\overline{\bf{O}}_{C,R}}=\bf{H}_{P_{C}}.
$$ 
We continue to denote the Hitchin base by $\bf{H}$. On an open subvariety $\bf{H}^{\text{KL}}$ of $\bf{H}$ (see Definition \ref{Def:H^KL}), we establish the following result:

\begin{mainthm}[Theorem \ref{thm:syz mirror}]\label{Thm:intro SYZ}
    The SYZ mirror symmetry holds for
    \[
    \begin{tikzcd}
        (\bf{Higgs}_{P_B},\alpha_B)\ar[rd, "h_{P_B}"']&&(\bf{Higgs}_{P_C},\alpha_C)\ar[ld, "h_{P_C}"]\\
        &\mathbf{H}
    \end{tikzcd}
    \]
    where $\alpha_B$ is trivial, as $\SO_{2n+1}$ is adjoint. More precisely:
    \begin{enumerate}
        \item $\bf{Higgs}_{P_C}|_{\bf{H}^{\text{KL}}}$ and $\bf{Higgs}_{P_B}^{+}|_{\bf{H}^{\text{KL}}}$ are trivial torsors over families (over $\bf{H}^{\text{KL}}$) of dual abelian varieties.
        \item $ \Split'(\bf{Higgs}_{P_C}|_{\bf{H}^{\text{KL}}},\alpha_C|_{\bf{H}^{\text{KL}}})\cong \bf{Higgs}_{P_B}^{-}|_{\bf{H}^{\text{KL}}} $.\footnote{Here, $\Split'$ refers to the induced torsor of the dual abelian scheme. See \cite[Definition 6.4]{GWZ20}.}
        \item $\Split'(\bf{Higgs}^{\pm}_{P_B}|_{\bf{H}^{\text{KL}}},\alpha_B|_{\bf{H}^{\text{KL}}})\cong \bf{Higgs}_{P_C}|_{\bf{H}^{\text{KL}}}$.\footnote{This again shows the close relation between rational points of parabolic Hitchin fibers and SYZ mirror symmetry.}
    \end{enumerate}  
\end{mainthm}

In contrast to the non-parabolic case, where Hitchin \cite{Hit07} showed a correspondence between the generic fibers of types B and C, establishing a relationship between $h_{P_B}^{-1}(a)$ and $h_{P_C}^{-1}(a)$ for $a \in \bf{H}^{\text{KL}}$ is more intricate due to the parabolic structure. This requires not only addressing the ``non-degeneracy” of bilinear pairings, as Hitchin did, but also identifying the residues of Higgs fields at marked points.

The identification of these residues is subtle and deeply connected to the combinatorial properties of special nilpotent orbits and the singularities of generic spectral curves. Specifically, it involves the residually nilpotent local Higgs bundles of types B and C, discussed in Section~\ref{Sec:local_level}. Furthermore, we reveal a surprising geometric interpretation of Lusztig's canonical quotient group, detailed in Theorem~\ref{Thm:intro local level}.

Finally, adhering to the philosophy of ``abstract dual Hitchin systems" proposed in \cite[\S 6]{GWZ20}, or rather a modified version ``weak abstract dual Hitchin systems" by Shen\cite{She18},  we prove the following \emph{topological mirror symmetry}:

\begin{mainthm}[Theorem~\ref{Thm:TMS}]\label{Thm:intro TMS}
    Under the Condition~\ref{cond:coprime}, the following topological mirror symmetry holds for Langlands dual parabolic Hitchin systems (with $\alpha_B$ omitted as it is trivial):
	\begin{equation*}
		E^{\alpha_C}(\bf{Higgs}_{P_C};u,v)=E(\bf{Higgs}_{P_C};u,v)= E_{\text{st}}(\bf{Higgs}_{P_B}^{-};u,v)
        =E_{\text{st}}(\bf{Higgs}_{P_B}^{+};u,v).
	\end{equation*}
    Here, $E_{\text{st}}$ is the stringy E-polynomial, a generating series of stringy Hodge numbers. Moreover, $\bf{Higgs}_{P_B}^{\pm}$ are treated as quotients of corresponding moduli spaces for (twisted) parabolic $\Spin_{2n+1}$-Higgs bundles. 
\end{mainthm}

We emphasize that we consider only the moduli of $\omega_{\Sigma}(x)$-valued strongly parabolic Higgs bundles, which naturally carry the symplectic forms $\omega_B$ and $\omega_C$ on the corresponding moduli spaces. These symplectic forms are used to construct gauge forms for the purposes of $p$-adic integration.

However, due to the presence of parabolic structures, our situation does not satisfy the “codimension 2” condition required in the definition of abstract dual Hitchin systems; see condition (c) of Definition 6.8 in \cite{GWZ20}. To compare the $p$-adic integrals, we must therefore find a method for comparing the corresponding gauge forms.

A similar issue arises in Shen’s work \cite{She24}, which studies the moduli of parabolic $\operatorname{SL}_n$ and $\operatorname{PGL}_n$ Higgs bundles. Since the moduli space of $\operatorname{PGL}_n$ Higgs bundles is a global quotient of that of $\operatorname{SL}_n$ Higgs bundles (potentially in different connected components). Shen demonstrates that the gauge forms he constructs are equivalent, enabling a meaningful comparison.

Our setting is different, as there is no morphism between $\mathbf{Higgs}_{P_B}$ and $\mathbf{Higgs}_{P_C}$. Nevertheless, we show that there exists a morphism between their open subvarieties:
\[
    \bf{Higgs}_{P_B}|_{\bf{H}^{\text{KL}}}\longrightarrow \bf{Higgs}_{P_C}|_{\bf{H}^{\text{KL}}}.
\]

Moreover, there are symplectic forms on these subvarieties that are compatible with this morphism and coincide with the restrictions of the natural symplectic forms $\omega_B$ and $\omega_C$, up to a constant scalar. This allows us to compare $\omega_B$ and $\omega_C$, and hence the associated gauge forms on both sides. See Section \ref{Sec:TMS} for further details.

\subsection{Idea of proof}

\subsubsection{Generic fibers of Hitchin maps}

The key to proving both the SYZ and topological mirror symmetries lies in understanding the generic fibers of the Hitchin maps, specifically their geometry and torsor structures.

\begin{proposition}[Proposition \ref{Prop:fiber_JM_moduli} and \ref{L_BC}]\label{prop:intro L_BC} \
Let $\bf{H}^{\text{KL}} \subset \bf{H}_{\overline{\bf{O}}_C}$ be as defined in Definition \ref{Def:H^KL}. For any $a\in \bf{H}^{\text{KL}}$, the following hold:
    \begin{enumerate}
        \item For any nilpotent orbit $\overline{\bf{O}}_C$ (not necessarily special), the fiber $h_{\overline{\bf{O}}_C}^{-1}(a)$ is naturally a torsor of abelian variety $\Prym_a:=\Prym(\overline{\Sigma}_a,\overline{\Sigma}_a/\sigma )$, where $\overline{\Sigma}_a$ is the normalization of the spectral curve $\Sigma_a$ and $\sigma$ is the involution. 
        \item If $\overline{\bf{O}}_B$ and $\overline{\bf{O}}_C$ are special and Springer dual, then for $a \in \bf{H}^{\text{KL}}$, there exists a finite map 
        $$
            L_{BC}: h_{\overline{\bf{O}}_B}^{-1}(a) \longrightarrow h_{\overline{\bf{O}}_C}^{-1}(a)
        $$ 
        of degree 
        \begin{equation}\label{eq.intro_l_BC}
            2^{2n(2g-2)+\beta(\bf{d}_B)-c(\bf{d}_B)-1}.\footnote{Here  $\beta(\bf{d}_B)$ and $c(\bf{d}_B)$  are calculated from partitions and Springer dual map (see Definition \ref{def:e(d_B),c(d_B)}). The analysis involves residually nilpotent local Higgs bundles.}
        \end{equation}
        Furthermore, $h_{\overline{\bf{O}}_B}^{-1}(a)$ has two connected components: $h_{\overline{\bf{O}}_B}^{-1}(a)^+$ and $h_{\overline{\bf{O}}_B}^{-1}(a)^-$. Each is a torsor over $\Prym_{\overline{\bf{O}}_B, a}$, a finite cover of $\Prym_a$ (defined in Proposition \ref{type B JM and different components}). A canonical point exists on $h_{\overline{\bf{O}}_B}^{-1}(a)^+$.
    \end{enumerate}
\end{proposition}

To define the map $L_{BC}$, we analyze the relationship between residually nilpotent local Higgs bundles, explained in Theorem~\ref{Thm:intro local level}. However, the degree of $L_{BC}$ reveals the following obstructions:

\begin{corollary}[Corollary \ref{Cor.Not_dual}] \label{intro_not dual}
    Let $\overline{\bf{O}}_B$ and $\overline{\bf{O}}_C$ be special and Springer dual. If $c(\bf{d}_B)\neq 0$, then the connected components $h_{\overline{\bf{O}}_B}^{-1}(a)^{\pm}$ and $h_{\overline{\bf{O}}_C}^{-1}(a)$ are torsors of abelian varieties which are \emph{NOT} dual to each other. In particular, \emph{SYZ} mirror symmetry fails in this case.
\end{corollary}

The failure is caused by the term $c(\bf{d}_B)$ in \eqref{eq.intro_l_BC}. By Lemma \ref{lem:Lusztig's_quotient}, $2^{c(\bf{d}_B)}$ equals the order of Lusztig's canonical quotient, denoted by $\bar{A}(\bf{O}_B)$, which is a quotient of the component group ${A}(\bf{O}_B)$. This connection is not coincidental, as we will elaborate.

Let us now consider Springer dual Richardson orbits $\bf{O}_{B, R}$ and $\bf{O}_{C, R}$, and their dual polarizations $P_{B}$ and $P_C$. The four associated moduli spaces fit into the following commutative diagram:
\begin{equation}\label{dia:moduli spaces}
    \begin{tikzcd}
    \bf{Higgs}_{ P_B} \ar[rd, "h_{P_B}"]     &   & \bf{Higgs}_{ P_C} \ar[ld, "h_{P_C}", swap] \\
    \bf{Higgs}_{\overline{\bf{O}}_{B, R}} \ar[r, "h_{\overline{\bf{O}}_{B, R}}", swap]     &   \bf{H}   &  \bf{Higgs}_{\overline{\bf{O}}_{C, R}} \ar[l, "h_{\overline{\bf{O}}_{C, R}}"] 
    \end{tikzcd}.
\end{equation} 

As noted earlier, there is no direct relationship between the generic Hitchin fibers of $h_{P_B}$ and $h_{P_C}$. To find such a relationship, we first construct a canonical map between the generic fibers of $h_{\overline{\bf{O}}_{B/C, R}}$ and $h_{P_{B/C}}$, and then use the finite map $L_{BC}$. Denote the connected components of $h_{P_B}^{-1}(a)$ by $h_{P_B}^{-1}(a)^{\pm}$. The relations between the fibers can be summarized as follows:

\begin{proposition}[Corollary \ref{Cor.factor through} and  Theorem \ref{Thm:dual abelian}] \label{prop:intro dual abelian}
          For $a\in\bf{H^{\text{KL}}}$, the Hitchin fibers satisfy the following relations:
     \begin{equation*}  
         \begin{tikzcd}
              h_{P_B}^{-1}(a)^{\pm} \ar[d, "\nu_{P_B}"'] & h_{P_C}^{-1}(a)\ar[d,"\nu_{P_C}"] & \Prym_{P_B,a}\ar[d]& \Prym_{P_C,a}\ar[d]\\
              h_{\overline{\bf{O}}_{B, R}}^{-1}(a)^{\pm} \ar[ur]\ar[r,"L_{BC}"]& h_{\overline{\bf{O}}_{C, R}}^{-1}(a) & \Prym_{\overline{\bf{O}}_{B, R},a}\ar[ur]\ar[r]& \Prym_a
         \end{tikzcd}.  
     \end{equation*}
    The right-hand side describes the diagram of abelian varieties, while the left-hand side describes torsors over these varieties. Furthermore,
     \begin{enumerate}
         \item $L_{BC}$ factors through $h^{-1}_{P_C}(a)$.
         \item $\deg \nu_{P_B} \cdot \deg \nu_{P_C} = \# \bar{A}(\bf{O}_{B, R})$.
         \item $\Prym_{P_B,a}$ and $\Prym_{P_C,a}$ are dual abelian varieties.
     \end{enumerate} 
\end{proposition}

\begin{remark} 
    For a Richardson orbit $\bf{O}_{B/C, R}$, different choices of polarizations $P_{B}$ and $P_{C}$ are possible. The degrees $\deg \nu_{P_{B/C}}$ may vary depending on the choice of polarization, but $\Prym_{P_B, a}$ and $\Prym_{P_C,a}$ remain dual to each other.    
\end{remark}

Following the strategies of \cite{DP12} and \cite{GWZ20}, we use the trivial $\mu_2$-gerbe $\alpha_B$ on $\bf{Higgs}_{P_B}^{\pm}$, as $\SO_{2n+1}$ is an adjoint group. On the other hand, we define $\alpha_C$ as the lifting gerbe of the universal (parabolic) $\PSp_{2n}$-Higgs bundle on $\bf{Higgs}_{ P_C}$. The existence of rational points (even over function fields of the Hitchin base) on $h^{-1}_{\overline{\bf O}_{C,R}}(a)$ shows that it is a trivial torsor. 

With this framework, and following the philosophy of ``abstract dual Hitchin systems'' from \cite[\S 6]{GWZ20}, we prove Theorem~\ref{Thm:intro SYZ} and Theorem~\ref{Thm:intro TMS} (see Section~\ref{Sec:SYZ} and~\ref{Sec:TMS} for more details).

\begin{remark}\label{rmk:rational pt KL singularities}
    An essential feature of parabolic Hitchin systems (or more general systems with special nilpotent orbits at marked points) is the construction of rational points on generic fibers over fields that are not necessarily algebraically closed (see Proposition Proposition \ref{prop:generic fiber is trivial torsor}). These rational points are critical for building the SYZ and topological mirror symmetries. Their existence is closely tied to the singularities of generic spectral curves or, more specifically, their normalizations. The normalizations are obtained through the local decomposition of characteristic polynomials, a process described by Spaltenstein \cite{Spa88} in the context of Kazhdan--Lusztig maps. This connection is why the good open subset of the Hitchin base is denoted by $\bf{H}^{\text{KL}}$.
\end{remark}

\subsubsection{Relation to Lusztig's canonical quotient}

Lusztig's canonical quotient plays a pivotal role in the classification of unipotent representations of finite groups of Lie type. Interestingly, it also plays a crucial role in Proposition~\ref{prop:intro dual abelian}. Specifically, the seesaw property in Proposition~\ref{prop:intro dual abelian} arises from the analysis of residually nilpotent local Higgs bundles, defined as follows.

In the following, let $\Bbbk$ denote an algebraically closed field with $\text{char}(\Bbbk)\ne 2$ unless otherwise specified. Let $\mathcal{O}=\Bbbk[\![t]\!]$ be the ring of formal power series, and $\sK=\Bbbk(\!(t)\!)$ be its fractional field. The corresponding (positive) loop groups and algebras are denoted as $LG, L^{+}G, L\mathfrak{g},L^{+}\mathfrak{g}$, where
\[
    LG(\Bbbk)=G(\sK),\quad L^{+}G(\Bbbk)=G(\sO),\quad L\mathfrak{g}(\Bbbk)=\mathfrak{g}(\sK),\quad L^{+}\mathfrak{g}(\Bbbk)=\mathfrak{g}(\sO).
\]
Additionally, let $L\mathfrak{g}^{\heartsuit}$ denote the set of topologically nilpotent and generically regular semisimple elements. 

Let $\widehat{P} \subset L^{+}G$ be a (parabolic type) parahoric subgroup, defined as the preimage of a parabolic subgroup $P < G$ under the reduction map $L^{+}G \rightarrow G$. Define the affine Grassmannian of $G$ as $\bf{Gr}:=LG/L^{+}G$, and the affine flag variety of $G$ as $\bf{Spal}_P = LG/\widehat{P}$. 

Recall from \cite[\S 0]{KL88} and \cite[\S 4.2]{SXY23} the following definitions:
\begin{definition}[Affine Springer/Spaltenstein fiber] \label{def:affine_Spaltenstein_fiber}
Let $\theta \in L\mathfrak{g}^{\heartsuit}$. The associated affine Springer fiber is defined as
\[
   \bf{Gr}_{\theta} := \{g L^{+}G \in \bf{Gr} \mid \Ad_{g^{-1}} \theta \in L^{+}\mathfrak{g} \}.
\]
The affine Spaltenstein fiber is
\[
    \textbf{Spal}_{\theta, P}:=\left\{ g\widehat{P}\in LG/\widehat{P} \mid \Ad_{g^{-1}}\theta\in\widehat{\gn}\right\},
\]
where $\widehat{\gn}$ is the topologically nilpotent radical of the parahoric subalgebra $\widehat{\gp}$. 
\end{definition}

Let $\theta \in L\mathfrak{g}^\heartsuit$. The affine Springer fiber admits a restriction map
\[
   \rm{ev}_{\theta} : \bf{Gr}_{\theta} \rightarrow [\sN/G],
\]
which sends $gL^{+}G$ to $\Ad_{g^{-1}}\theta$ $\mod t$. Here, $\sN$ is the nilpotent cone in $\g$. Following Yun's notations \cite{Yun21}, we define: 

\begin{definition}[residually nilpotent local Higgs bundles] \label{def:local Higgs bundle} 
For any nilpotent orbit $\bf{O} \subset \sN$, let  $\bf{Gr}_{\theta, \bf{O}}$ be the preimage of $\bf{O}/G$ under $\rm{ev}_{\theta}$. Elements of $\bf{Gr}_{\theta, \bf{O}}$ are called \emph{residually nilpotent local $G$-Higgs bundles} associated to $\theta$ and the nilpotent orbit $\bf{O}$. 
\end{definition}

Affine Springer fibers are intimately connected to the Hitchin fibers of Higgs bundles via the celebrated product formula of Ng{\^o}~\cite{Ngo10}. Analogously, affine Spaltenstein fibers are related to Hitchin fibers of parabolic Higgs bundles. 

Under the conditions of \eqref{dia:moduli spaces}, we analyze the local structure of Hitchin fibers, leading to Proposition~\ref{prop:intro dual abelian}. To distinguish between types B and C, we denote the fibers as $\bf{Gr}_{\theta_B, \bf{O}_{B}}$ and $\bf{Gr}_{\theta_C, \bf{O}_C}$.

\begin{mainthm}[Proposition \ref{description of W}, \ref{prop:mirror_position}, and Theorem \ref{Thm.gp_action}]\label{Thm:intro local level}
Let $\theta_B$ and $\theta_C$ be related as in \eqref{theta_B/C}. Let $\widehat{\mathbf{Gr}}_{\theta_C, \bf{O}_{C, R},}$ be a cover of affine Spaltenstein fiber, with $\deg \nu_{P_C}^\vee = \deg \nu_{P_C} $, constructed in \eqref{tilde_LH}. Then, 
    \begin{equation*}
        \begin{tikzcd}
        \widehat{\mathbf{Gr}}_{\theta_C, \bf{O}_{C, R}} \ar[d, "\nu_{P_C}^\vee"'] & \\
        \textbf{Spal}_{\theta_B, P_B} \ar[d, "\nu_{P_B}", swap] & \textbf{Spal}_{\theta_C, P_C} \ar[d, "\nu_{P_C}"]\\
        {\mathbf{Gr}}_{\theta_B, \bf{O}_{B, R}} \ar[r, "l_{BC}"] \ar[ru] & {\mathbf{Gr}}_{\theta_C, \bf{O}_{C, R}}
    \end{tikzcd}.
    \end{equation*}
Here, $l_{BC}$ is of degree $2^{\beta(\bf{d}_B) - c(\bf{d}_B)}$. The centralizer $Z_{\Sp_{2n}(\sK)}(\theta_{C})$ acts transitively on each set in the following diagram, with fibers of the maps having the following torsor structures: the fiber of $\nu_{P_C}$ is a $\sA(\theta_C)/\sA(P_C)$ torsor, the fiber of $l_{BC}$ is a $\sA(\theta_C)/\sA(W)$ torsor, the fiber of $\nu_{P_B}$ is a $\sA(W)/\sA(P_B)$ torsor, and the fiber of $\nu_{P_C}^\vee$ is a $\sA(P_B)$ torsor. See Section \ref{s.gp_action} for the notations of these groups. Moreover, 
\[
    (\sA(W)/\sA(P_B)) \times (\sA(\theta_C)/\sA(P_C)) \cong \bar{A}(\bf{O}_{C,R}),
\]
where $\bar{A}(\bf{O}_{C, R})$ is the Lusztig's canonical quotient.
\end{mainthm}

\begin{remark}\ 
    \begin{itemize}
        \item The construction of $l_{BC}$ applies to all special Springer dual orbits $\bf{O}_{B/C}$, not just Richardson orbits (see Theorem \ref{Thm.type B decomposition}).

        \item The above diagram can be regarded as the local counterpart of Proposition~\ref{prop:intro dual abelian}. Both results offer new geometric interpretations of Lusztig's canonical quotient.
        \end{itemize}
\end{remark}

\subsection{Plan of the paper}

In Section~\ref{Sec:local_level}, we study residually nilpotent local Higgs bundles and establish a connection between types B and C (Theorem~\ref{Thm.type B decomposition}).

In Section~\ref{Sec:seesaw_for_group_action}, we explore affine Spaltenstein fibers and uncover a geometric interpretation of Lusztig's canonical quotient (Theorem~\ref{Thm.gp_action}).

In Section~\ref{S.JM_moduli_space}, using Jacobson–Morozov resolutions, we construct moduli spaces of Higgs bundles associated to nilpotent orbit closures and relate Springer duality to spectral curve singularities (Theorem~\ref{Thm:why special}).

In Section~\ref{Sec:Par_BNR}, using the parabolic BNR correspondence, we analyze generic Hitchin fibers and find seesaw relations between parabolic Higgs bundles and newly constructed moduli spaces in the Richardson case. We prove that the generic Hitchin fibers are torsors over dual abelian varieties (Theorem~\ref{Thm:dual abelian}).

In Section~\ref{Sec:SYZ}, carefully analyzing the various torsor structures finally leads to the Strominger--Yau--Zaslow mirror symmetry.

In Section~\ref{Sec:TMS}, using the framework of ``abstract dual Hitchin systems" \cite{GWZ20}, we establish topological mirror symmetry.

\subsection*{Acknowledgements}

We would like to express our gratitude to Yongbin Ruan for suggesting this project and for his valuable feedback on the initial draft. Special thanks go to Hiraku Nakajima for suggesting the consideration of group actions, which shaped Section~\ref{s.gp_action}. We also thank Weiqiang He for insightful discussions on nilpotent orbits. Discussions with Xiaoyu Su on moduli spaces were also invaluable. 

We are grateful to Tam\'as Hausel, Mirko Mauri, Michael McBreen, Junliang Shen, and Qizheng Yin for their helpful suggestions. Part of this manuscript was written during the second and third authors' visit to the Institute for Advanced Study in Mathematics at Zhejiang University, and we sincerely appreciate the institute’s inspiring environment and support.

We also thank the anonymous referee for his comments on Section~7 and for pointing out a gap in our previous proof.

B. Wang conducted part of this work at the Steklov Institute of Mathematics (Moscow), supported by the Ministry of Science and Higher Education of the Russian Federation (agreement no. 075-15-2019-1614 ). He is currently supported by funding from the Department of Mathematics at the Chinese University of Hong Kong, the General Research Fund (Project code: 14307022), and the Early Career Scheme (Project code: 24307121) of Michael McBreen. X. Wen acknowledges support from the Chongqing Natural Science Foundation Innovation and Development Joint Fund (CSTB2023NSCQ-LZX0031) and the Chongqing University of Technology Research Startup Funding Project (2023ZDZ013). Y. Wen is supported by a KIAS Individual Grant (MG083902) at the Korea Institute for Advanced Study.

\section{Residually nilpotent local Higgs bundle} \label{Sec:local_level}

\subsection{$\theta_A$-direct summand and modification}\label{Sec:theta_direct_summand}

In this subsection, we address the $\GL_n$ case. The cases for types B and C will be discussed in the following subsection. Let $\bf{O}_A \subset \mathfrak{gl}_n $ be a nilpotent orbit, and let $(E,\theta) \in \mathbf{Gr}_{\theta,\bf{O}_A}$ represent a residually nilpotent local Higgs bundle. Consider the characteristic polynomial
$$
f(\lambda):=\chi_\theta(\lambda)=\lambda^n+a_1\lambda^{n-1}+\cdots + a_{n-1}\lambda+ a_n,
$$ 
with $a_i\in \mathcal{O}$. Then $(E,\theta)$ can be regarded as an $\sO_{f}$-module, where $\sO_f=\sO[\lambda]/f(\lambda)$. Assume that $\chi_{\theta}(\lambda)$ admits a factorization into irreducible factors in $\mathcal{O}[\lambda]$
\begin{align*}
    \chi_{\theta}(\lambda)=f_1(\lambda)\cdot \cdots \cdot f_k(\lambda).
\end{align*}
Let $e_i=\deg f_i(\lambda)$, and define $\bf{deg} = [e_1, \ldots, e_k]$, which forms a partition of $n$. 

Since $(E, \theta)$ is residually nilpotent, all coefficients $a_i$ are divisible by $t$. Let $\sP(n)$ denote the set of partitions of $n$. If the partition of $\bf{O}_A \subset \mathfrak{gl}_n$ is given by $\bf{d}=[d_1, \cdots, d_r] \in \sP(n)$, the $t$-order of $a_i$ is determined as follows:

\begin{proposition}[Theorem 4~\cite{SWW22}] \label{order of coe}
The coefficients $a_i$ in $\chi(\theta)$ satisfy the following inequality
\[
\text{ord}_t a_i\geq \min\{\ell \mid d_{\ell}\geq i\}.
\]
In particular, $\bf{d}\geq \bf{deg}$ as partitions of $n$.
\end{proposition}

\begin{remark}\label{KLmap}
    If $\theta$ is generic, then $\chi_\theta(\lambda)$ coincides with the Kazhdan--Lusztig map. Recall that Kazhdan--Lusztig map is an injective map (proven by Spaltenstein~\cite{Spa90} and Yun~\cite{Yun21}) from the set of nilpotent orbits to conjugacy classes of the Weyl group
    \[
        \KL: \sN \longrightarrow \text{Conj}(\sf{W}).
    \]   
    where $\sN$ is the set of nilpotent orbits. The Kazhdan--Lusztig map is defined as follows: for an element $e$ in the nilpotent orbit $\mathbf{O} \subset \g$, choose a generic lifting $\tilde{e} \in e+tL^{+}\mathfrak{g}$ that is regular semisimple in the loop group $\LG$. The centralizer $Z_{\LG}(\theta)$ is then a maximal torus in $\LG$. Kazhdan--Lusztig \cite{KL88} showed that the conjugacy classes of $Z_{\LG}(\theta)$ is independent of the choice of $\tilde{e}$, yielding a well-defined map:
    \begin{align*}
       \sN & \rightarrow \text{rational conjugacy classes of maximal torus of}\; \LG \\
       e   & \mapsto Z_{\LG}(\tilde{e}).
    \end{align*} 
    This map corresponds canonically to conjugacy classes of the Weyl group $\sf{W}$. In the next subsection, we will employ Spaltenstein's interpretation of Kazhdan--Lusztig maps for types B and C.
\end{remark}

From now on, let $K_i = \Ker f_i(\theta)$ for each irreducible factor $f_i(\lambda)$. Clearly, we have an injection $\oplus_{i=1}^kK_i\rightarrow E$. The main goal of this subsection is to determine the quotient of this injection under certain mild conditions.

\begin{definition}
    For a submodule $i_F: F\hookrightarrow E$, we say that $F$ is a \emph{$\theta$ direct summand} of $E$ if
    \begin{itemize}
        \item $F$ is $\theta$ invariant.
        \item There is an $\mathcal{O}$ morphism $s_F: E \rightarrow F $ such that $s_F$ is compatible with $\theta$, and $s_F\circ i_F = \id_F$.
    \end{itemize}  
\end{definition}

For a $\theta$-invariant saturated submodule $F$ of $E$ (i.e., $E/F$ is also torsion-free), denote by $\bar{\theta}$ the induced morphism on $E/F$, and let $\chi_{F}$ be the characteristic polynomial of $\theta|_{F}$. Choose an $\mathcal{O}$-linear basis of $F$ and extend it to a basis of $E$. Then, in this basis, $\theta$ has the matrix form: 
$$
\Theta= \begin{pmatrix}
            \Theta_F & \Theta_{12} \\
        0        & \bar{\Theta} 
\end{pmatrix}.
$$
\begin{lemma}\label{splitting lemma}
  
We put 
$$
\chi_F(\Theta)= \begin{pmatrix}
                             0 & M \\
                             0 & \chi_F(\bar{\Theta})
                            \end{pmatrix}.
$$
Suppose $\chi_F(\bar{\Theta})$ is invertible (over $\sK$). Then $F$ is a $\theta$ direct summand of $E$ if and only if $M\cdot \chi_F(\bar{\Theta})^{-1} $ is integral, i.e., its entries belong to $\mathcal{O}$.
\end{lemma}

\begin{proof}
    Consider the following commutative diagram:
    \[
    \begin{tikzcd}
        E\ar[rd,"\pi"]\ar[rr,"\chi_{F}(\theta)"]&& E\ar[rd,"\pi"]\\
        & E/F\ar[ur,"\phi"]\ar[rr,"\chi_{F}(\bar{\theta})"]&& E/F .
    \end{tikzcd}
    \]
    If $M\cdot \chi_F(\bar{\Theta})^{-1}$ is integral, then $\phi\circ \chi_{F}(\bar{\theta})^{-1} $ is well defined, making $F$ a $\theta$ direct summand of $E$.

    Conversely, if there exists a section $s:E/F\rightarrow E$ compatible with $\theta$, then $s$ ensures the following diagram
    \[
    \begin{tikzcd}
        E\ar[rd, shift left=0.50ex, "\pi"]\ar[rr,"\chi_{F}(\theta)"]&& E\ar[rd,shift left=0.50ex, "\pi"]\\
        & E/F\ar[ul, shift right=-0.55ex, "s"]\ar[ur, "\phi"]\ar[rr,"\chi_{F}(\bar{\theta})"]&& E/F\ar[ul, shift right=-0.55ex, "s"].
    \end{tikzcd}
    \]
    Since $\phi\circ \pi=\chi_{F}(\theta)$, then we have
    \[
    \phi=\chi_{F}(\theta)\circ s=s\circ \chi_{F}(\bar{\theta}),
    \]
    i.e., the whole diagram, including $s$, is commutative. In particular, we have $\phi\circ \chi_{F}(\bar{\theta})^{-1}=s$ is defined.
\end{proof}

Inductively, we can construct a basis of $E$ such that the matrix of $\theta$ is a block upper triangular matrix, with diagonal blocks corresponding to $R(f_i)$, the companion matrix of $f_i(\lambda)$.

\begin{lemma}\label{resultant}
   If $f_i(\lambda)$ and $f_j(\lambda)$ are Eisenstein polynomials with no common roots. Then 
   $$\text{ord}_t\text{det}f_i(R(f_j))\geq \text{min}\{e_i,\ e_j\}.$$  
   Equality holds if and only if $e_i=e_j$ and $\ord_{t}(a_{e_j,j}-a_{e_i,i})^{e_i}=e_i$. Moreover, if $e_i\geq e_j$, we have $t|f_i(R(f_j))$.
\end{lemma}
\begin{proof}
    We observe that
    \[
    \text{det}f_i(R(f_j))=\text{res}(f_i,f_j),
    \]
     where $\text{res}(f_i,f_j)$ is the resultant of the polynomials $f_i$ and $f_j$. Since both $f_i$ and $f_j$ are Eisenstein polynomials, the definition of the resultant gives
    \[
    \text{ord}_t\text{det}f_i(R(f_j))\ge\text{min}\{e_i,\ e_j\}.
    \] 
    The equality holds if and only if $e_i=e_j$ and $\ord_{t}(a_{e_j,j}-a_{e_i,i})^{e_i}=e_i=e_j$.    
\end{proof}

Throughout the paper, we adopt the following assumptions:
\begin{assumption}\label{main assumption on char}
\ 
\begin{itemize}
    \item [(1)] All $f_i(\lambda)$ are Eisenstein polynomials, i.e., the non-leading coefficients of $f_i(\lambda)$ have $t$-order $1$;
    \item [(2)] For $f_i$ and $f_j$ such that $i \neq j$ and $e_i\geq e_j$, we require $\text{ord}_t\text{det}f_i(R(f_j))= e_j$.
    \item[(3)] For any $i \geq 1$, $\sum_{j=1}^i(d_j-e_j)\leq 1$. 
\end{itemize}
\end{assumption}

\begin{definition}\label{def:Obar}
    Whenever Assumption \ref{main assumption on char} holds, define
    \[
    \bar{\sO}_f:=\prod \sO[\lambda]/(f_i(\lambda)).
    \]
    Here, $\bar{\sO}_f$ is the normalization of $\sO_f$, and its total fraction field is denoted by $\sK_{f}$.
\end{definition}

Although $\sum_{j=1}^i(d_j-e_j)\leq 1$, the submodules $K_i$ are not always $\theta$-direct summand of $E$. However, we have the following key result:

\begin{theorem}\label{Thm.sub_splitting}
    Let $\bf{d}$ and $\bf{deg}$ be as defined earlier. For $i\geq 1$, if $\sum_{j=1}^{i}(d_j-e_j)=0$ and $e_i=d_i$, then $K_i$ is a $\theta$-direct summand of $E$.
\end{theorem}

We start from the following technical lemmas leading to $\theta$-direct splitting.

\begin{lemma}\label{lem:intersection estimate}
    We have the following inequality
    \[
    \dim_{\Bbbk} K_i(0)\cap \Im(f_i(\theta)(0))\le \sum_{\{j \mid d_j\ge e_i, e_j\ge e_i\}}(d_j-e_j).
    \]
    Equality holds if and only if $\Im f_i(\bar{\theta})$ contains $t(E/K_i)$.
\end{lemma}
\begin{proof}
Consider the following diagram
\[
0\rightarrow K_i(0)\rightarrow E(0)\xrightarrow{f_i(\theta)(0)}E(0).
\]
From this, we find
\[
\dim \Im(f_i(\theta)(0))= \sum_{d_j\ge e_i}(d_j-e_i).
\]
But we know that $\ord_t\det f_i(\bar{\theta})=\sum_{j\ne i,e_j\ge e_i} e_{i}+\sum_{e_j<e_i} e_j$, and $\rk E/K_i=\sum_{j\ne i}e_j$, then from the following diagram
\begin{equation}
    \begin{tikzcd}       E/K_i\ar[d]\ar[r,"f_i(\bar{\theta})"]\ar[r]&E/K_i\ar[r]\ar[d]&\mathcal{Q}\ar[r]\ar[d]&0\\  E/K_i(0)\ar[r,"f_i(\bar{\theta})(0)"]&E/K_i(0)\ar[r]&\mathcal{Q}\otimes_{\sO} \Bbbk \ar[r]&0    
    \end{tikzcd}.
\end{equation}
Since $\dim\sQ\ge \dim \sQ\otimes_{\sO} \Bbbk$,
\[
\dim \Im (f_i(\bar{\theta})(0))\ge \sum_{j\ne i}e_j-\ord_t\det f_i(\bar{\theta})=\sum_{e_j\ge e_i}(e_j-e_i).
\]
Then consider the following diagram
\[
\begin{tikzcd}
    0 \ar[r] & K_i(0) \ar[r] \ar[d, "0"]  &   E(0) \ar[r] \ar[d, "f_i(\theta)(0)"]  & (E/K_i)(0) \ar[r] \ar[d, "f_i(\bar{\theta})(0)"] & 0 \\
    0 \ar[r] & K_i(0) \ar[r]  &   E(0) \ar[r]  & (E/K_i)(0) \ar[r] & 0 
\end{tikzcd}.
\]
Hence
\[
\dim_{\Bbbk} K_i(0)\cap  \Im(f_i(\theta)(0))\le \sum_{d_j\ge e_i,e_j\ge e_i}(d_j-e_j).
\]
The equality holds if and only if $\dim\sQ=\dim \sQ\otimes_{\sO} \Bbbk $, if and only if $t(E/K_i)$ is contained in the image of $f_{i}(\bar{\theta})$.
\end{proof}

\begin{lemma}\label{lem: direct summand at 0}
    If $\dim K_i(0)\cap  \Im(f_i(\theta)(0))=0$, and $e_i=d_j$ for some $j$, then $K_i(0)$ is a $\theta(0)$-direct summand of $E(0)$.
\end{lemma}
Since we always have $\sum_{j=1}^i(d_j-e_j)\leq 1$ for any $i$, and $\bf{d}$ and $\bf{deg}$ are partition of $n$. Then for any $k$, $d_k - e_k = -1$, $0$, or $1$. So the possible choices are $j=i-1, i,i+1$.
\begin{proof}
    Recall that $\theta(0)$ has partition $\bf{d}$. Using the $\mathfrak{sl}_2$-triple, decompose $E(0)$ as 
    \[
    E(0)=\oplus_{d_j} V_{d_j}    
    \]
    where each $V_{d_j}$ is an irreducible $\mathfrak{sl}_2$-representation with highest weight $d_{j}-1$. 
    
    Since $f_i$ is Eisenstein, we can write
    \[
    K_i(0)=\{v, \theta(0)v,\ldots \theta(0)^{d_i-1} v \}.
    \]
    Here, $v\in \ker\theta(0)^{d_i}$. Since $e_i=d_j$, $K_i(0)\cap \Im (\theta^{e_i}(0))=0$, the coordinate of $v$ in the highest weight subspace of $\oplus_{d_{\ell}=e_i} V_{d_\ell}$ is non-zero. Hence 
    \[
    K_i(0)\rightarrow E(0)/\oplus_{d_\ell\ne e_i} V_{d_\ell}
    \]
    is injective. 
    
    Since $E(0)/\oplus_{d_\ell\ne e_i} V_{d_\ell}$ is a direct sum of irreducible representations of the $\mathfrak{sl}_2$-triple with highest weight $d_{\ell}-1$. Then we can see that the image of $K_i(0)$ is also such an irreducible representation, hence a $\theta(0)$-direct summand of $E(0)/\oplus_{d_\ell\ne e_i} V_{d_\ell}$. In particular, $K_i(0)$ is a $\theta(0)$-direct summand of $E(0)$.
\end{proof}

\begin{proof}[Proof of Theorem \ref{Thm.sub_splitting}]
By Lemma \ref{lem:intersection estimate}, the given conditions imply the following:
\begin{align}
    &\dim K_i(0)\cap \Im(f_i(\theta)(0))=0\label{eq:no intersection}\\
    &\dim\sQ=\dim \sQ\otimes_{\sO} \Bbbk \label{eq: can be invertible}
\end{align}
By Lemma \ref{lem: direct summand at 0}, the equality \eqref{eq:no intersection} implies that $K_i(0)$ is a $\theta(0)$-direct summand. 

Now consider the matrix as in Lemma \ref{splitting lemma}. Let $F=K_i$. Since $K_i(0)$ is a $\theta(0)$-direct summand, it means that we can choose $M$ such that $M/t$ is integral. 

Next, \eqref{eq: can be invertible} implies that the image of $f_i(\bar{\theta})$ contains $t(E/K_i)$. Thus, $tf_i(\bar{\theta})^{-1}$ is integral. By Lemma \ref{splitting lemma}, it follows that $K_i$ is a $\theta$-direct summand.
\end{proof}

An important special case:
\begin{corollary}\label{Cor.sub_splitting}
If $\bf{d}=\bf{deg}$, then under Assumption \ref{main assumption on char}, we have a canonical isomorphism 
$$
     E\cong \bigoplus_{i=1}^{k}K_i.
$$
\end{corollary}

In the following, when we refer to characteristic polynomials, we implicitly assume that $G$ is a classical group and do not restate this explicitly.

\begin{definition} \label{Def:Char}
    We denote by $L\mathfrak{c}_{\mathbf{O}}$ the set of all the characteristic polynomials of residually nilpotent local principal $G$-Higgs bundle associated with the nilpotent orbit $\mathbf{O}$. This set is a subset of $\mathcal{O}[\lambda]$. A condition is said to be \emph{generic} in $L\mathfrak{c}_{\mathbf{O}}$ if it holds for polynomials in a Zariski open subset of $L\mathfrak{c}_{\mathbf{O}}$.
\end{definition}

Notice that $L\mathfrak{c}_{\bf{O}_A}$ is determined as in Proposition~\ref{order of coe}. With this understanding, let $\mathbf{Gr}_{\theta, \mathbf{O}}$ be as in Definition~\ref{def:local Higgs bundle}. Using these ideas, we present a new proof of the following result:

\begin{proposition}[Theorem 6~\cite{SWW22}] \label{Thm.type A decomposition}
    Let $G=\GL_n$ and $\bf{O}_A \subset \mathfrak{gl}_n$ with partition $\bf{d}_A$. For $\theta\in \End(\sK^{n})$ whose characteristic polynomial $f$ is generic in $L\mathfrak{c}_{\bf{O}_A}$, write the decomposition of $f$ as $f=\prod_{i=1}^kf_i$. Then, there is a set-theoretic isomorphism 
    \[
    \mathbf{Gr}_{\theta, \mathbf{O}_A}\cong     
       \{ (\sL\hookrightarrow \sK_f) \mid \sL \text{ is a  rank $1$ free module of} \;\bar{\sO}_f
    \}.
    \]
\end{proposition}
\begin{proof}
    By \cite[Proposition 4.3]{Spa88} and Lemma~\ref{resultant}, Assumption~\ref{main assumption on char} is generic in $L\mathfrak{c}_{\bf{O}_A}$. Furthermore, by \cite[\S 4.2,4.3]{SWW22}, we have  $\bf{d}_A=\bf{deg}$. Therefore, by Corollary \ref{Cor.sub_splitting}, every lattice in $  \mathbf{Gr}_{\mathbf{O}_A, \theta}$ can be viewed as rank-1 free module of $\bar{\sO}_f$. To construct the bijection, it suffices to fix an isomorphism $\sK_{f}\cong \sK^{n}$ (as $\sK$ vector spaces)  compatible with the action of $\lambda$ and $\theta$ respectively. Such an isomorphism certainly exists (although it is not unique). 
\end{proof}

\subsection{$\theta_{B/C}$-direct summand and modification}
Nilpotent orbits in $\mathfrak{so}_{2n+1}$ (type B) and $\mathfrak{sp}_{2n}$ (type C) correspond to partitions of $2n+1$ and $2n$, respectively. Specifically:
\begin{itemize}
	\item For $\mathfrak{so}_{2n+1}$, partitions belong to the set $\mathcal{P}_{1}(2n+1)$, where even parts occur with even multiplicity.
	\item For $\mathfrak{sp}_{2n}$, partitions belong to the set $\mathcal{P}_{-1}(2n)$, where odd parts occur with even multiplicity.
\end{itemize}
These partition sets are defined as
\begin{align*}
	\mathcal{P}_{\varepsilon}(N) = \big\{ [d_1, \ldots, d_N] \in \mathcal{P}_{}(N) \ \big{|} \ \sharp \{j \ | \ d_j =i \} \ \text{is even for all $i > 0$ with} \ (-1)^i= \varepsilon      \big\},
\end{align*}
where $\varepsilon=\pm 1$. For more details, see \cite{CM93}.

Among nilpotent orbits, some are distinguished as \emph{special orbits}.

\begin{definition}\label{def:special nilpotent}
    For a partition $\bf{d}=\left[ d_1, \ldots, d_N  \right]$, its dual partition $\bf{d}^t$ is defined by $d^t_i = \sharp \left\{ j \ \big{|} \ d_j \geq i \right\} $ for all $i > 0$. A partition of type B ($\bf{d} \in \mathcal{P}_{1}(2n+1)$) or type C ($\bf{d} \in \mathcal{P}_{-1}(2n)$) is called {\emph{special}} if its dual  partition $\bf{d}^t$ lies in the same set ($\mathcal{P}_{1}(2n+1)$ for type B or $\mathcal{P}_{-1}(2n)$ for type C). The corresponding nilpotent orbits are called \emph{special orbits}. 
    
    We denote the set of special partitions as 
    \[
        \mathcal{P}_{1}^{sp}(2n+1) \; \text{for type B}, \quad \mathcal{P}_{-1}^{sp}(2n) \; \text{for type C}.
    \]
    For further details, see \cite[Proposition 6.3.7]{CM93}.
\end{definition}

By Springer\cite[Theorem 6.10]{Spr}, there is an injective map from the irreducible representations of the Weyl group to irreducible equivariant local systems on nilpotent orbits. If a nilpotent orbit $\bf{O}$ with a trivial representation corresponds to a special representation of the Weyl group (see \cite{Lus79}), then $\bf{O}$ is special.  

Langlands dual groups share the same Weyl group, allowing for a bijection between their special partitions, known as the \emph{Springer duality map}:
\begin{align}
    S: \sP^{sp}_{-1}(2n)\to \sP^{sp}_{1}(2n +1). \label{Springer_dual_map}
\end{align}
We refer to $S(\bf{d})$ as the \emph{Springer dual} of $\bf{d}$, denoting it as ${}^S\bf{d}$. Similarly, we denote the Springer dual of a special orbit $\bf{O}$ as ${}^S\bf{O}$.

For a partition $\mathbf{d} = [d_1, \cdots, d_k]$, with $d_k \geq 1$, define 
\[
    \mathbf{d}^- = [d_1, \cdots, d_k-1], \quad \mathbf{d}^+ = [d_1+1, \cdots, d_k].
\]
The following result is from \cite[Chapter III]{Spa06} (see also \cite[Proposition 4.3]{KP89}).

\begin{proposition} \label{p.Springerdual}
The map $\mathbf{d} \mapsto (\mathbf{d}^+)_B$ provides a bijection $S: \sP^{sp}_{-1}(2n)\to \sP^{sp}_{1}(2n +1)$. Its inverse is given by $\mathbf{f} \mapsto (\mathbf{f}^-)_C$. Here, $(\bf{d})_{B/C}$ denotes the largest partition of type B or C that is smaller than $\bf{d}$ under the partial order, where 
\[
    \bf{d} =\left[ d_1, \ldots, d_N  \right] \geq \mathbf{f}= \left[ f_1, \ldots, f_N  \right] \Longleftrightarrow  \sum_{j=1}^k d_j \geq \sum_{j=1}^k f_j, \ \text{for all} \ 1 \leq k \leq N.
\]
\end{proposition}

\subsubsection{Type C}
For $G=\Sp_{2n}$, we fix a nondegenerate skew-symmetric two-form $g_C:\sK^{2n}\otimes \sK^{2n}\rightarrow \sK$. For $\theta_C: \sK^{2n}\rightarrow \sK^{2n}$ such that $g_C(\theta_C -,-)+g_C(-,\theta_C -)=0$. The set of residually nilpotent $\Sp_{2n}$-Higgs bundle associated with $\theta_C$ and a nilpotent orbit $\bf{O}_C$ (see Definition \ref{def:local Higgs bundle}) is given by

\[
{\mathbf{Gr}}_{\theta_C, \bf{O}_{C}}:=\left\{E_C \subset \sK^{2n}\ \left| \
\begin{aligned}
    & E_C \;\text{is a rank-}2n\; \text{lattice, such that} \\
    & g_C|_{E_C} \; \text{is perfect, with values in}\  \sO; \\
    & E_C \text{ is } \theta_C \text{ invariant}; \\
    & \theta_C(0) \in \bf{O}_C.
\end{aligned}\right.\right\}.
\]

Let $\bf{O}_C$ be a nilpotent orbit of type C with partition $\bf{d}_C=[d_1,\ldots,d_k]$. For generic $\theta_C$, consider  $(E_C, \theta_C) \in \mathbf{Gr}_{\theta_C, \bf{O}_C}$ and its characteristic polynomial:
\[
\chi_{\theta_C}(\lambda)=\det(\lambda-\theta_C) = \prod_if_{C,i},\quad f_{C,i}\in\sO[\lambda].
\]
Let $e_{C, i}=\deg f_{C,i}$.

\begin{proposition}\label{type C generic}
Assumption \ref{main assumption on char} and condition $\bf{d}_C=\bf{deg}_C$ are all generic in $L\mathfrak{c}_{\bf{O}_C}$. Furthermore: 
\begin{itemize}
	\item If $e_{C, i}$ is even, then it is self dual, i.e., $f_{C,i}(\lambda)=f_{C,i}(-\lambda)$.
	\item If $e_{C, i}$ is odd, then there exists a unique $f_{C,j}$ such that $f_{C,j}(\lambda)=-f_{C,i}(-\lambda)$.
\end{itemize}
\end{proposition}

\begin{proof}
    The results are established in \cite[Proposition 5.2]{Spa88}, except for verifying Assumption \ref{main assumption on char} (2) for $f_{C, i}$ and $f_{C, j}$ where $f_{C,j}(\lambda)=-f_{C,i}(-\lambda)$. Clearly, $f_{C, i}(R(f_{C, j}))/t$ is integral. To compute $\ord_t \det f_{C, i}(R(f_{C, j}))$, note that $f_{C, j}(R(f_{C, j}))=0$ and $f_{C,i}-f_{C,j}$ is a polynomial with constant term of $t$-order $1$ (since $e_{C,i}$ is odd). Thus, for general $f_{C,i}$, $(f_{C, i}(R(f_{C, j}))/t)(0)$ is invertible, and the $t$-order of $\det f_{C, i}(R(f_{C, j}))$ equals $e_i$.
\end{proof}

Let $E=\{\pm 1, \pm 2,\ldots,\pm n\}$, and let $\sf{W}_0$ be the permutation group of $E$. For $G=\Sp_{2n}, \SO_{2n+1}$, their Weyl group $\sf{W}$ can be identified as
\[
    \{w \in \sf{W}_0 \mid w(-i)=-i, 1\le i\le n\}.
\]
For $w\in \sf{W}$, we can associate a pair of partitions $(\alpha,\beta)$, where $|\alpha|+|\beta|=n$, as follows: Let $W$ be a $\langle w\rangle$-orbit, then $-W$ is also an orbit. If $W \ne -W$, then $\alpha$ gets one part $\alpha_i=|W|$. If $W=-W$, then $|W|$ is even and $\beta$ gets one part $\beta_i=\frac{|W|}{2}$.  The conjugacy classes of its Weyl group $\sf{W}$ are parametrized by all such pairs of partitions $(\alpha,\beta)$. For type C, the Kazhdan--Lusztig map coincides with $\KL_C(\bf{O}_C)=(\alpha_C, \beta_C)$, where
\begin{enumerate}
    \item[1.] $e_{C,i}$ is odd, $\alpha_C$ gets one part $\bf{\alpha}_{C,i}=e_{C,i}$.
    \item[2.] $e_{C,i}$ is even, then $\beta_C$ gets one part $\beta_{C,i}=\frac{e_{C,i}}{2}$.
\end{enumerate}

Combining Corollary~\ref{Cor.sub_splitting} and Proposition~\ref{type C generic}, similar as Proposition~\ref{Thm.type A decomposition}, we have the following result.

\begin{theorem}\label{Thm.type C decomposition}
    If $\chi(\theta_C)$ is generic in $L\mathfrak{c}_{\bf{O}_C}$, there is an isomorphism: 
    \[
    \mathbf{Gr}_{\theta_C, \mathbf{O}_C} \cong     
       \{ (\sL\hookrightarrow \sK_f, \sigma^*\sL\cong\sL^{\vee}) \mid \sL \text{ is a  rank $1$ free module of} \;\bar{\sO}_f.
       \}
    \]
    where $\sigma$ is the involution $\lambda\mapsto -\lambda$.
\end{theorem}
\begin{remark}
    The right-hand side can be viewed as a local analog of Prym varieties associated with $(\bar{\sO}_f,\sigma)$.
\end{remark}

\subsubsection{Type B}
For $G=\SO_{2n+1}$, let $g_B: \sK^{2n+1} \otimes \sK^{2n+1} \to \sK$ be a fixed nondegenerate symmetric bilinear form. Consider $\theta_B: \sK^{2n+1} \to \sK^{2n+1}$ such that $\theta_B$ satisfies the compatibility condition $g_B(\theta_B(-), -) + g_B(-, \theta_B(-)) = 0$. Similarly, we have $\bf{Gr}_{\bf{O}_{B}, \theta_B}$ as in the type C case.

Let $\bf{d}_B = [d_{B,1}, \ldots, d_{B,r'}]$ denote the partition associated with the nilpotent orbit $\bf{O}_B$. The characteristic polynomial of $\theta_B$ is expressed as
\begin{align*}
    \chi_{\theta_B}(\lambda) = \lambda(\lambda^{2n}+a_{B,2}\lambda^{2n-2}+\cdots +a_{B,2n-2}\lambda^2+a_{B,2n}) 
\end{align*}
or equivalently, as the product of irreducible factors
\begin{align}
     \chi_{\theta_B}(\lambda) = \lambda\prod_{i=1}^{k}f_{B, i}(\lambda) \label{type B char decomposition},
\end{align}
where $f_{B,i}(\lambda) \in \sO[\lambda]$ are irreducible polynomials and $e_{B,i} = \deg f_{B,i}$. The partition $\bf{deg}_B=[e_{B, 1}, \ldots , e_{B, k}, 1]$ encodes the degrees of these factors. The analog of Proposition \ref{type C generic} does not hold in the $\SO_{2n+1}$ case, which causes the main difficulty. However, $\bf{deg}_B$ and the properties of $f_{B,i}$ can be determined as follows, which is due to \cite[Proposition 6.4]{Spa88}.

Motivated by Corollary~\ref{rough decomposition of E}, we classify partitions of type B into four distinct types based on their structure and properties:
\begin{itemize}
       \item B1: Partitions of the form $[a, a]$ where $a \equiv 1$.
       \item B1*: Partitions of the form $[b, b]$ where $b \equiv 0$.
       \item B2: Partitions of the form $[ a_1, b_1^2, \cdots , b_k^2, a_2 ]$, where $a_i \equiv 1$, $b_j \equiv 0$, and $a_1 > b_1 \geq \cdots \geq b_k > a_2$, with $k\geq 0$. 
       \item B3: Partitions of the form $[ a, b_1^2, \cdots , b_k^2 ]$, where $a \equiv 1$, $b_j \equiv 0$, and $a > b_1 \geq \cdots \geq b_k $, with $k\geq 0$.
\end{itemize}

\begin{lemma}\label{structure of special partition}
   A partition $\bf{d}_B = [d_{B,1}, \ldots, d_{B,r'}]$ of type B can be uniquely expressed as a concatenation of blocks of type B1, B1*, B2, and B3. Furthermore, the block of type B3 appears exactly once and is positioned at the end of the partition. 
\end{lemma}

With this notation, we can write $\bf{d}_B$ as $\bf{T}_B = [\bf{T}_1, \bf{T}_2,  \ldots, \bf{T}_s]$ where each $\bf{T}_i$ is either of Type B1, B1*, B2 or B3. Moreover, $\bf{d}_B$ is special if and only if there is no Type B1*. For Richardson orbits, we have
\begin{lemma} \label{lem:structure_of_Richardson_partition}
    Let $\bf{d}_{B} = [d_{B,1}, d_{B, 2}, \ldots] = [\bf{T}_{1},\bf{T}_{2}, \ldots, \bf{T}_{l-1}, \bf{T}_{l}, \ldots, \bf{T}_{k}, \bf{T}_{k+1}]$ be a partition of Richardson orbit $\bf{O}_{B,R}$. Then, there exists $l \geq 1$, such that for $1 \leq i \leq l-1$, $\bf{T}_{i}$ is either of type B1 or B2 of the form $[a_1, a_2]$, and $\bf{T}_{j}$, for $l \leq j \leq k$, is of type B2, and $\bf{T}_{k+1}$ is of type B3.
\end{lemma}
\begin{proof}
    It is known that Richardson orbits are special. From the finer structure of their partitions, see \cite[\S 2.3]{FRW24}, we conclude.
\end{proof}

Define a partition 
\[
{}^S \bf{T}_B = [{}^S\bf{T}_1, {}^S\bf{T}_2,  \ldots, {}^S\bf{T}_s] \in \sP_{-1}(2n)
\]
as follows
\[
   {}^S\bf{T}_i = \left \{ \begin{aligned}
                         &[a, a], &  \text{if } \bf{T}_i \text{ is of type B1;} \\
                         &[b, b], & \text{if } \bf{T}_i \text{ is of type B1*;} \\
                         &[a_1-1, b_1^2, \cdots , b_k^2, a_2+1], &\text{if } \bf{T}_i \text{ is of type B2;} \\
                         &[a-1, b_1^2, \cdots , b_k^2] & \text{if } \bf{T}_i \text{ is of type B3}.
                      \end{aligned}
             \right.
\]

\begin{proposition}\label{type B generic}
Let $\chi_{\theta_B}(\lambda)$ decomposed as in \eqref{type B char decomposition}. Then, the degree partition $\bf{deg}_B$ satisfies $\bf{deg}_B = [{}^S \bf{T}_B ,1]$, where ${}^S \bf{T}_B$ is the dual partition constructed above. More preciously, let $\bf{deg}_B=[e_{B,1}, e_{B,2}, \ldots]$. Then
\begin{enumerate}
    \item if $e_{B,j} \equiv 1$, there exists a unique $j'$ such that $f_{B,j}(-\lambda) = - f_{B,j'}(\lambda)$;
    \item if $e_{B,j} \equiv 0$ and does not appear in Type B1*, then $f_{B,j}(-\lambda)=f_{B,j}(\lambda)$;
    \item if $e_{B,j} \equiv 0$ and appears in Type B1*, then there exists a unique $j'$ such that $f_{B,j}(-\lambda) = f_{B,j'}(\lambda)$.
\end{enumerate}
Moreover, if $\bf{d}_B$ is special, then the Assumption \ref{main assumption on char} is generic in $L\mathfrak{c}_{\bf{O}_B}$. Additionally, the dual partition satisfies ${}^S \bf{d}_B={}^S \bf{T}_B$, as defined in \eqref{Springer_dual_map}.
\end{proposition}
\begin{proof}
    The results follow from \cite[Proposition 6.4]{Spa88}, except for verifying that when $\bf{d}_B$ is special, Assumption \ref{main assumption on char} (2) holds for $f_{B, i}$ and $f_{B, j}$ such that $f_{B,j}(\lambda)=-f_{B, i}(-\lambda)$. This verification proceeds in the same manner as in the proof of Proposition \ref{type C generic}.
\end{proof}

The Kazhdan--Lusztig map for type B, denoted $\KL_B(\bf{O}_B)=(\alpha_B, \beta_B)$, can be described as follows:
\begin{enumerate}
    \item[1.] For each odd $e_{B,j}$, $\alpha_B$ gains one part $\bf{\alpha}_{B,j}=e_{B,j}$.
    \item[2.] For each even $e_{B,i}$ corresponding to case (3) in Proposition~\ref{type B generic}, $\alpha_B$ gains one part $\alpha_{B,i} = e_{B,i}$.
    \item[3.] For each even $e_{B,i}$ corresponding to case (2) in the Proposition~\ref{type B generic}, $\beta_B$ gains one part $\beta_{B,i}=\frac{e_{B,i}}{2}$.
\end{enumerate}

\begin{proposition}\label{p.com_image_of_KL}
    The common images of Kazhdan--Lusztig maps of types B and C are those corresponding to Springer dual special orbits.
\end{proposition}
\begin{proof}
    First notice that if $\bf{O}_B$ is non-special, for generic $\theta_B$, case (3) in Proposition~\ref{type B generic} is non-empty which doesn't appear in $\Im \KL_C$. Thus, there exists no $\bf{O}_C$ such that $\KL_B(\bf{O}_B)=\KL_C(\bf{O}_C)$. For special orbits $\bf{O}_B$, $\bf{deg}_B = [{}^S\bf{d}_B, 1]$ here ${}^S\bf{d}_B$ is the Springer dual partition.
\end{proof}

\subsection{Local symmetries}\label{Sec:local mirror symmetry}

\begin{definition} \label{def:e(d_B),c(d_B)}
    Let $\bf{d}_B$ be a partition of type B with $\KL_B(\bf{O}_B)=(\alpha_B, \beta_B)$. Define:
    \begin{itemize}
    	\item $c(\bf{d}_B)  := \# \left\{  \text{type B2 in} \ \bf{T}_B \right\}$, where $\bf{T}_B$ represents the decomposition of $\bf{d}_B$ into Types B1, B1*, B2, and B3 (cf. Lemma \ref{structure of special partition}).
    	\item $\beta(\bf{d}_B) := \# \beta_B$, which equals $\# \left\{ d_{B,i} \equiv 0 \mid d_{B,i} \in \ \bf{deg}_B   \right\}$ when $\bf{d}_B$ is special.
    \end{itemize} 
\end{definition}

\begin{lemma} \label{lem:Lusztig's_quotient}
Let $\#\bar{A}(\bf{O}_B)$ denote the order of Lusztig's canonical quotient. Then,
    \begin{align*}
        \#\bar{A}(\bf{O}_B) = 2^{c(\bf{d}_B)}.
    \end{align*}
\end{lemma}
\begin{proof}
 
It is known that $\bar{A}(\bf{O}_B) = \Z_2^q$ for some $q$. By \cite[\S 5]{so01}, $q+1$ equals the number of ``corners" of the Young diagram of $\bf{d}_B$ that have both odd length and odd height. By analyzing the definitions of Types B1, B1*, B2, and B3, it follows that $q=c(\bf{d}_B)$.
\end{proof}

To simplify the discussion, we will denote $\theta_{B}|_{\sK},\theta_{C}|_{\sK}$, and similar elements simply as $\theta_B,\theta_C$ etc., when the context is clear. For a fixed $\theta_B$ , if $a_{B,2n}\neq 0$, then $\Ker \theta_B \cong \sK$ and we have an exact sequence: 
\[
    0\longrightarrow \Ker \theta_B \longrightarrow \sK^{2n+1} \longrightarrow \sK^{2n}\longrightarrow 0.
\]
We define a nondegenerate skew-symmetric two-form $g_C$ on the quotient $\sK^{2n}$ via 
\[
    g_C(\bar{u}, \bar{v})=g_B(\theta_B u, v)/t
\] 
and define $\theta_C: \sK^{2n}\rightarrow \sK^{2n}$ by 
\begin{align}\label{theta_B/C}
\theta_C(\bar{u})=\bar{\theta_B(u)}.
\end{align}

With these notations, our main theorem in this subsection is

\begin{theorem}\label{Thm.type B decomposition} 
    Let $\bf{O}_B$ be a special nilpotent orbit in $\mathfrak{so}_{2n+1}$ with partition $\bf{d}_B$. If $\chi_{\theta_B}$ is generic in $L\mathfrak{c}_{\bf{O}_B}$, then there exists a finite morphism 
    $$
    l_{BC}:\mathbf{Gr}_{\theta_B, \bf{O}_B}\rightarrow \mathbf{Gr}_{\theta_C, {}^S\bf{O}_B}
    $$ 
    with degree $2^{\beta(\bf{d}_B)-c(\bf{d}_B)}$.
\end{theorem}

Since the proof of Theorem \ref{Thm.type B decomposition} is lengthy, we outline the key steps below. The full proof follows in  \S\ref{subsubsec:B to C} and \S\ref{subsubsec:C to B}.
\begin{itemize}
    \item[\textbf{Part 1}] (\S\ref{subsubsec:B to C}) We construct $l_{BC}$ step-by-step, starting from a fixed $(E_B\subset \sK^{2n+1},\theta_B\in\End(E_B))\in \mathbf{Gr}_{\bf{O}_B, \theta_B}$.
\begin{itemize}\label{part 1}
\item[Step 1.1.] Let $K_{B, i}=\Ker f_{B, i}(\theta_B)$ for each $i \geq 1$, and let $K_{B, 0}=\Ker \theta_B$. Then
\begin{itemize}
    \item If $e_{B,i}$ is odd, $K_{B, i}$ is a $\theta_B$-direct summand of $E_B$;
    \item If $e_{B,i}$ is even, $K_{B, i}$ is $1$-degenerate (cf. Definition~\ref{Def:l-degeneracy}).
    \item If $\ord_{t}a_{B,2n}$ is even, $K_{B, 0}$ is a $\theta$-direct summand of $E_B$; otherwise it would be $1$-degenerate.
\end{itemize}
This leads to an exact sequence:
$$
0\longrightarrow \bigoplus_{i=0}^k K_{B,i} \longrightarrow E_B \longrightarrow R\longrightarrow 0,
$$  
where $R$ is supporting at $t=0$, and $\dim_{\Bbbk} R$ is determined by $\bf{d}_B$ (see Proposition \ref{coker dimension}).

\item[Step 1.2.] For each $i\geq 1$, consider $\theta_{B,i} = \theta_B |_{K_{B,i}}$. There exists a canonical $\theta_{B,i}$-invariant submodule $K_{C,i} \subset K_{B,i}$, such that the pairing 
\[
g_C(-,-) = g_B(\theta_B-,  -) / t
\] 
defines a skew-symmetric non-degenerate bilinear form on $\oplus_{i=1}^k K_{C,i}$. Define 
\[
l_{BC}\big((E_B\subset \sK^{2n+1})\big)=(\oplus_{i=1}^k K_{C,i}\subset \sK^{2n}).
\]
This defines an element in $\mathbf{Gr}_{\theta_C, {}^S\bf{O}_B}$. See Proposition \ref{part 1.2}.
\end{itemize}
\vspace{0.3cm}

\item[\textbf{Part 2}] (\S\ref{subsubsec:C to B}) We analyze the fiber of $l_{BC}$ for a fixed $(E_C\subset \sK^{2n})\in \mathbf{Gr}_{\theta_C,{}^S\bf{O}_B}$.
\begin{itemize}
\item[Step 2.1.] Using Theorem \ref{Thm.type C decomposition}, decompose $E_C\cong \oplus_{i=1}^k \Ker f_{C, i}(\theta_C)$. For $K_{C, i}=\Ker f_{C, i}(\theta_C)$, there is a canonical submodule $K_{B,i}$ in $t^{-1}K_{C,i}$, containing $K_{C,i}$, such that:
\begin{itemize}
	\item If $\deg f_{C, i}$ is odd, let $j$ satisfy $f_{C, j}(\lambda)=-f_{C,i}(-\lambda)$. Then, the pairing $$tg_C(\theta_C^{-1}-,-)$$ is a non-degenerate symmetric bilinear form on $K_{B,i}\oplus K_{B,j}$.
	\item If $\deg f_{C, i}$ is even, the pairing defines a $1-$degenerate symmetric bilinear form on $K_{B, i}$.
\end{itemize}
This step reverses step 1.2, and it will be done in Proposition \ref{part 2.1}.

\item[Step 2.2.] Define $K_{B, 0}=\mathcal{O}$ and a symmetric pairing on $K_{B,0}$ given by $a_{B, 2n}/t^{\lfloor \#(\bf{d}_C)/2 \rfloor}$. Combine $K_{B,0}$ with $\oplus_{i= 1}^kK_{B,i}$ to form a submodule of $\sK^{2n+1}$.

\item[Step 2.3.] We have an exact sequence 
$$
0\longrightarrow \oplus_{i= 0}^kK_{B,i}\longrightarrow \oplus_{i= 0}^kK_{B,i}^{\vee} \longrightarrow Q\longrightarrow 0
$$ 
given by the pairing on $\oplus_{i=0}^{k}K_{B, i}$ defined in Step 2.1 and Step 2.2. The fiber $l_{BC}^{-1}((E_C\subset \sK^{2n}))$ corresponds to certain $\oplus_{i=1}^k\theta_{B,i}^{\vee}$-invariant submodule of $\oplus_{i=0}^{k}K_{B,i}^{\vee}$, which corresponds to $\iota$-isotropic subspaces in $Q$, (cf. Definition \ref{def:iota_isotropic}).

\item[Step 2.4.] Under generic conditions, we show that there are precisely $2^{\beta(\bf{d}_B)-c(\bf{d}_B)}$ $\iota$-isotropic subspaces in $Q$. This is proved in Proposition~\ref{description of W} and establishes that the degree of $l_{BC}$ is $2^{\beta(\bf{d}_B)-c(\bf{d}_B)}$.
\end{itemize}
\end{itemize}

\subsubsection{From B-side to C-side}\label{subsubsec:B to C}
Firstly, we deal with Step 1.1. The relation between nondegenerate bilinear pairing and $\theta_B$ direct summands would be stated as

\begin{lemma}\label{direct summand being nondeg}
    Suppose $f_B(\lambda)$ is a factor of $\chi_{\theta_B}(\lambda)$ such that $f_B(\lambda)=f_B(-\lambda)$. Then, if $\Ker f_B(\theta_B)$ is a $\theta_B$-direct summand of $E_B$, the restriction of $g_B$ on $\Ker f_B(\theta_B)$ is nondegenerate. Conversely, for a $\theta_B$-invariant submodule $F \subseteq E_B$, if the restriction $g_B|_F$ is nondegenerate, then $F$ is a $\theta_B$-direct summand of $E_B$. 
\end{lemma}
\begin{proof}
If $\Ker f_B(\theta_B)$ is a $\theta_B$-direct summand, then $E_B=\Ker f_B(\theta_B)\oplus F'$, where $F'$ is $\theta_B$-invariant. The irreducible factors of the characteristic polynomials of $\theta_B|_{F'}$ are coprime with $ f_B$. Since $f_B(\lambda)=f_B(-\lambda)$ and $g_B(\theta_B, -)+g_B(-,\theta_B)=0$, it follows that $g_B(\Ker f_B(\theta_B),F')=0$. Hence, the restriction of $g_B$ on $\Ker f_B(\theta_B)$ is nondegenerate.

Conversely, if $g_B|_F$ is non-degenerate, then its orthogonal complement provide its $\theta_B$-invariant complement, implying $F$ is a $\theta_B$-direct summand.
\end{proof}

As mentioned before, $K_{B, i}$ may not be a $\theta_B$-direct summand of $E_B$, then by Lemma \ref{direct summand being nondeg}, the restriction of $g_B$ on $K_{B, i}$ will be degenerate. To measure the degeneracy of the restriction of $g_B$, we give the following definition.

\begin{definition} \label{Def:l-degeneracy}
    Let $\sF$ be a free $\mathcal{O}$-module, and let $g: \sF \otimes \sF\rightarrow \mathcal{O}$ be a symmetric pairing. We say $g$ is \emph{$\ell$-degenerate} if the cokernel of the induced morphism $\sF \rightarrow \sF^{\vee}$ is a torsion module of dimension $\ell$. Thus, $g$ is non-degenerate if and only if it is $0$-degenerate. When the pairing $g$ is fixed, we may also say that $\mathcal{F}$ is $\ell$-degenerate if $g$ is $\ell$-degenerate.
\end{definition}

The following lemma about $\ell$-degenerate is easy. We state it here without proof.

\begin{lemma}\label{l deg lemma}
Let $\sF,g$ be as above. 
\begin{itemize}
    \item If $\{\alpha_i\}$ is an $\mathcal{O}$-linear basis  of $\sF$, then $g$ is $\ell$-degenerate if and only if  $\ord_t\det (g(\alpha_i, \alpha_j))$ is $\ell$.
    
    \item If $g$ is $\ell$-degenerate on $\sF$, and $\sF^{\prime}$ is a submodule of $\sF$ with $\dim_{\Bbbk}\sF/\sF'=\ell '$, then $g|_{\sF^{\prime}}$ is $(\ell+2\ell ')$-degenerate.

    \item If $g$ is $\ell$-degenerate on $\sF$, and $\sF=\sF_1\oplus \sF_2$, where $g(\sF_1, \sF_2)=0$ and the restriction of $g|_{\sF_i}$ is $m_i$-degenerate, $i=1, 2$, then $m_1+m_2=\ell$. 
\end{itemize}
\end{lemma}

In the following, we always assume that $\chi_{\theta_B}(\lambda)$ is generic in $L\mathfrak{c}_{ \bf{O}_B}$ and write the decomposition of $\chi_{\theta_B}(\lambda)$ as in (\ref{type B char decomposition}). For any $(E_B\subset \sK^{2n+1}) \in \mathbf{Gr}_{\theta_B, \bf{O}_B}$, we denote $K_{B, i}=\Ker f_{B, i}(\theta_B)$, for $i \neq 0$, and $K_{B, 0}=\Ker \theta_B$. 

\begin{proposition}\label{coker dimension}
With the above notations:
\begin{itemize}
    \item If $e_{B,i}$ is odd, then $K_{B, i}$ is a $\theta_B$-direct summand of $E_B$. Moreover, consider the unique $j$ with $f_{B,j}(\lambda)=-f_{B,i}(-\lambda)$, then the restriction of $g_B$ on $K_{B,i}\oplus K_{B,j}$ is nondegenerate.
    \item If $e_{B,i}$ is even, then $K_{B, i}$ is $1$-degenerate. 
    \item If  $\ord_t a_{B,2n}$ is even, $K_{B, 0}$ is a $\theta_B$-direct summand of $E_B$; otherwise, it is $1$-degenerate.
\end{itemize}
As a consequence, the following exact sequence holds:
\[
    0\longrightarrow \bigoplus_{i=0}^k K_{B,i} \longrightarrow E_B \longrightarrow R \longrightarrow 0,
\]
where $R$ is a torsion module support at $t=0$, and $\text{dim}_{\Bbbk}R=\lceil \beta(\bf{d}_B)/2 \rceil$.
\end{proposition}

\begin{proof}
    First, note that $\bf{deg}_B = [ {}^S\bf{d}_B, 1 ] $. For $e_{B,i}$ odd ($i\geq 1$), Lemma~\ref{structure of special partition}, Proposition~\ref{type B generic} and Theorem~\ref{Thm.sub_splitting} imply that  $K_{B,i}$ is a $\theta_B$-direct summand of $E_B$. We consider the unique $j$ such that $f_{B,j}(\lambda)=-f_{B,i}(-\lambda)$. Then $K_{B,j}$ is also a $\theta_B$-direct summand of $E_B$, and $K_{B,i}\oplus K_{B, j} = \Ker (f_{B,i}\cdot f_{B,j})(\theta_B)$. By Lemma~\ref{direct summand being nondeg}, the restriction of $g_B$ on $K_{B, i}\oplus K_{B,j}$ is nondegenerate.

    Since we only consider special orbits in type B, by ruling out all the $K_{B, i}\oplus K_{B,j}$ as in the above case, we can assume that $\operatorname{KL}(\bf{O}_B)$ is elliptic. Then by discussion in Section 9.1 of \cite{Yun21}, we see that if $e_{B, i}$ is even, then $g_B|_{K_{B,i}}$ can only be $1$-degenerate.

    To determine the restriction of $g$ on $K_0$, it suffices to compute $K_{B, 0}$ directly. We choose a basis of $E$ such that $g$ is given by the identity matrix. Then $\theta_B$ is represented by a skew-symmetric matrix $\Theta$ under the same basis. Now $K_{B, 0}= \Ker \theta$ is generated by the vector $v$ with $i$-th entry being the Pfaffian of the $i$-th leading principle minor of $\Theta$, but dividing the common $t$ order. Thus $g(v,v)=vv^{\top}$ has $t$ order $0$ or $1$ depending on $\ord_t a_{2n}$.
    
     The dimension of $R$ is determined by the restriction of $g$ on each $K_{B, i}$ and Lemma \ref{l deg lemma}.
\end{proof}

Now, we will consider Step 1.2 in Part 1 (\ref{part 1}).

\begin{proposition}\label{part 1.2}
    For each $K_{B, i}$ ($i\geq 1$), there exists a canonical submodule $tK_{B, i}\subset K_{C, i}\subset K_{B, i}$ such that for any $u, v \in \oplus_{i= 1}^kK_{C, i}$, we have $t \mid g_B(\theta_B u, v)$, and the pairing $g_C$ defined by $g_C(u, v)=g_B(\theta_Bu, v)/t$ is a nondegenerate skew-symmetric bilinear form on $\oplus_{i=1}^kK_{C, i}$.
\end{proposition}

\begin{proof}
    For each $i\geq 1$, let $k_i=\lfloor \dfrac{e_{B,i}-1}{2}\rfloor$ and define $K_{C,i}$ as the kernel of
    $$
    K_{B, i}\longrightarrow K_{B, i}(0)\longrightarrow K_{B, i}(0)/\Im \theta_{B,i}(0)^{k_i}.
    $$
    We aim to verify that $K_{C,i}$ satisfies the properties.
    
    Firstly, for those $K_{B, i}$ such that $e_{B, i}=2k_i+1$ being odd, we consider the $K_{B, j}$ such the restriction of $g_B$ on $K_{B, i}\oplus K_{B, j}$ is nondegenerate, as in the proof of Proposition \ref{coker dimension}. Now we choose a basis $\{\alpha_l \mid \alpha_l=\theta_B^{l-1}\alpha_1, 1\leq l \leq 2k_i+1\}$ for $K_{B, i}$, and $\{\tilde{\alpha}_l \mid \tilde{\alpha}_l=\theta_B^{l-1}\tilde{\alpha}_1, 1\leq l \leq 2k_i+1\}$ for $K_{B, j}$. Thus $K_{C, i}\oplus K_{C, j}$ is generated by 
    $$
\{t\alpha_1, \cdots, t\alpha_{k_i}, \alpha_{k_i+1}, \cdots , \alpha_{2k_i+1}, t\tilde{\alpha}_1, \cdots, t\tilde{\alpha}_{k_i}, \tilde{\alpha}_{k_i+1}, \cdots, \tilde{\alpha}_{2k_i+1}\}.
$$
For $\alpha_{n_1}$ and $\tilde{\alpha}_{n_2}$ with $n_1, n_2 \geq k_i+1$, we have $$g_B(\theta_B\alpha_{n_1}, \tilde{\alpha}_{n_2})=(-1)^{n_2-1}g_B(\theta_B^{n_1+n_2-1}\alpha_1, \tilde{\alpha}_1).$$ Notice that $n_1+n_2-1\geq 2k_i+1$ hence $t\mid g_B(\theta_B u, v)$ for any $u,v \in K_{C, i}\oplus K_{C, j}$. 

We see that the restriction of $g_B$ on $K_{C, i}\oplus K_{C, j}$ is $4k_i$-degenerate. Since $\ord_t \det \theta_B|_{K_{C, i}\oplus K_{C, j}}=2$, $g_C(-, -)=g_B(\theta_B-, -)/t$ defines a non-degenerate skew-symmetric bilinear form on $K_{C, i}\oplus K_{C, j}$.

Now let us consider $K_{B, i}$ such that $e_{B, i}=2k_i+2$ being even. Now we pick a basis $\{\alpha_l \mid \alpha_l=\theta_B^{l-1}\alpha_1, 1\leq l \leq 2k_i+2\}$ for $K_{B,i}$ and hence $K_{C,i}$ is generated by 
$$
\{t\alpha_1, \cdots, t\alpha_{k_i}, \alpha_{k_i+1}, \cdots, \alpha_{2k_i+2}\}.
$$
For $\alpha_{n_1}$ and $\alpha_{n_2}$ with $n_1, n_2 \geq k_i$, we have 
$$
g_B(\theta_{B} \alpha_{n_1},  \alpha_{n_2})=(-1)^{n_2-1}g_B(\theta_B^{n_1+n_2-1}\alpha_1, \alpha_1).
$$ 
Notice that $n_1+n_2-1\geq 2k_i+1$. When $n_1+n_2-1\geq 2k_i+2$, we have $t\mid g_B(\theta_B^{n_1+n_2-1}\alpha_1, \alpha_1)$ since $e_{B,i}=2k_i+2$ and when $n_1+n_2-1= 2k_i+1$, we have $g_B(\theta_B^{n_1+n_2-1}\alpha_1, \alpha_1)=0$ since $2k_i+1$ is odd. Thus $t\mid g_B(\theta_B u, v)$ for any $u, v\in K_{C,i}$.

By Proposition \ref{coker dimension}, the restriction of $g_B$ on $K_{B, i}$ is $1$-degenerate. Hence the pairing $g_B(-,-)$ is $(2k_i+1)$-degenerate on on $K_{C,i}$, which implies $g_B(\theta_B -, -)$ is $2+2k_i$ degenerate. Thus $g_C(-, -)=g_B(\theta_B-, -)/t$ defines a non-degenerate skew-symmetric bilinear form on $K_{C, i}$.
\end{proof}
Now, we conclude with the following proposition.

\begin{proposition}\label{part 1 coro}
     Let $\bf{O}_B$ be a special nilpotent orbit with partition $\bf{d}_B$. For $\theta_B$ such that $\chi_{\theta_B}(\lambda)$ is generic in $ L\mathfrak{c}_{\bf{O}_B}$, there exists a morphism 
    $$
    l_{BC}:\mathbf{Gr}_{\theta_B,\bf{O}_B}\rightarrow \mathbf{Gr}_{\theta_C,{}^S\bf{O}_B},
    $$
    which sends $(E_B\subset \sK^{2n+1})$ to $(\oplus_{i=1}^kK_{C, i}\subset \sK^{2n})$ as defined in Proposition \ref{part 1.2}.
\end{proposition}

\subsubsection{From C-side to B-side}\label{subsubsec:C to B}
In this following, we address Part 2: determining the fibers of the map $l_{BC}$. 

We fix $(E_C\subset \sK^{2n})\in \mathbf{Gr}_{\theta_C,{}^S\bf{O}_B}$. By Theorem~\ref{Thm.type C decomposition}, we have a natural decomposition: 
\[
    E_C\cong \oplus_{i=1}^kK_{C,i}.
\] 
We now reverse the process described in Proposition~\ref{part 1.2} to reconstruct $E_B$.
\begin{proposition}\label{part 2.1}
    For each $K_{C, i}$, there exists a canonical submodule $K_{B,i}$ in $t^{-1}K_{C,i}$, containing $K_{C,i}$, such that 
\begin{itemize}
    \item If $\deg f_{C, i}$ is odd, consider the $j$ such that $f_{C, j}(\lambda)=-f_{C,i}(-\lambda)$. Then the pairing $tg_C(\theta_C^{-1}-,-)$ is a well defined non-degenerate symmetric bilinear form on $K_{B,i}\oplus K_{B,j}$.
    \item If $\deg f_{C, i}$ is even, the pairing $tg_C(\theta_C^{-1}-,-)$ is a well-defined $1$-degenerate symmetric bilinear form on $K_{B, i}$.
\end{itemize}
\end{proposition}

\begin{proof}
    For $i\geq 1$, define $K_{B,i}$ as the kernel of 
    $$
    t^{-1}K_{C,i}\longrightarrow (t^{-1}K_{C,i})(0)/\Ker \theta_{C,i}(0)^{k_i}
    $$ 
    where $k_i=\lfloor \dfrac{e_{B,i}-1}{2}\rfloor$. Then the proof will be similar to that of Proposition~\ref{part 1.2}.
\end{proof}

\begin{remark}
    To reverse the process in Proposition~\ref{part 1.2}, it is quite natural to use the pairing $tg_C(\theta_C^{-1}-,-)$. By Theorem~\ref{Thm.type C decomposition}, we see that $t\theta_C^{-1}$ is well-defined over $\sO$.
\end{remark}

We simply set $K_{B,0}=\mathcal{O}$ and define a pairing on $K_{B,0}$ as: 
\[
    g(u,v)=(a_{C,2n}/t^{\lfloor \#(\bf{d}_C)/2 \rfloor})\cdot uv.
\] 
Now, we have
\begin{itemize}
    \item Morphisms $\theta_{B, i}: K_{B,i}\rightarrow K_{B,i}$ with $\chi(\theta_{B,i})=f_{C,i}(\lambda)=f_{B,i}(\lambda)$ for $i\geq 1$, and $\theta_{B, 0} = 0$. Thus, $f_{B, 0}:= \chi(\theta_{B, 0}) = \lambda$.
    \item Pairings defined on $\oplus_{i=0}^kK_{B,i}$ induce the following exact sequence: 
    $$
    0\longrightarrow \oplus_{i= 0}^kK_{B,i}\longrightarrow \oplus_{i= 0}^kK_{B,i}^{\vee} \longrightarrow Q\longrightarrow 0.
    $$
\end{itemize}

It is clear that submodules of $\oplus_{i= 0}^kK_{B,i}^{\vee} $ containing $\oplus_{i= 0}^kK_{B,i}$ correspond to subspaces of $Q$. For any $W\subset Q$, denote the submodule by $E_W$. Thus, the fiber $l_{BC}^{-1}((E_C\subset \sK^{2n}))$ consists of such kind of submodules $E_W$ satisfying:
 \begin{itemize}
     \item  There exists an induced nondegenerate pairing $g_W$ on $E_W$.
     \item  The induced morphism $\theta_W$ on $E_W$ mod t, i.e., $\theta_W(0)$, has a partition matching $\bf{d}_B$.
 \end{itemize}

\begin{definition}\label{def:iota_isotropic}
     For a subspace $W$ in $Q$, if the inverse image $E_W$ in $\oplus_{i= 0}^kK_{B,i}^{\vee}$ lies in $l_{BC}^{-1}((E_C\subset \sK^{2n}))$, then $W$ is called $\iota$-isotropic.
\end{definition}

\begin{remark}
    Lemma 9.9 of \cite{Yun21} describes the set of all possible $\iota$-isotropic subspaces. However, we require a more detailed description of $\iota$-isotropic subspaces (see Proposition \ref{description of W} below) to aid subsequent discussions.
\end{remark}

We now determine all $\iota$-isotropic subspaces in $Q$. Firstly, for $i$ such that $e_{B, i}$ is odd, let $j$ be the index such that there is a nondegenerate pairing on $K_{B, i}\oplus K_{B,j}$ (see Proposition~\ref{part 2.1}). By Proposition~\ref{coker dimension}, $K_{B, i}\oplus K_{B,j}$ is already a $\theta_W$-direct summand of $E_W$. If $\ord_t a_{B,2n}$ is even, then $K_{B,0}$ is also a $\theta_W$-direct summand of $E_W$.

For simplicity, we can remove all the $\theta_W$-direct summands, focusing on elliptic classes. Thus, we assume that special partition $\bf{d}_B$ satisfies: 
\begin{itemize}
    \item[$\diamondsuit$] There is no type B1 in $\bf{d}_B$, and if $\ord_t a_{B,2n}$ is even, then there is no type B3 in $\bf{d}_B$ (see Lemma \ref{structure of special partition} for the definition of the types.) Note: The $\bf{d}_B$ here may not represent a partition in type B, but we retain the notation for consistency. 
\end{itemize}

We mention that there is a nondegenerate pairing on $Q$ as follows: the pairing $g_{B,i}$ on $K_{B,i}$ induces a pairing $\oplus_{i=0}^{k}K_{B,i}^{\vee}\otimes \oplus_{i=0}^{k}K_{B,i}^{\vee}\rightarrow t^{-1}\mathcal{O}$. This yields a nondegenerate pairing $Q\otimes Q\rightarrow t^{-1}\mathcal{O}/\mathcal{O}\cong \Bbbk$. 

\begin{proposition}\label{isotropic}
     There exists an induced pairing on $E_W$ if and only if $W \subset Q$ is isotropic under the pairing of $Q$. Furthermore, the induced pairing is nondegenerate if and only if $W$ is maximal isotropic.
\end{proposition}
 
\begin{proof}
  $E_W$ is a submodule of $\oplus_{i=0}^{k}K_{B,i}^{\vee}$, so we have an induced pairing $g_W: E_W\otimes E_W\rightarrow t^{-1}\mathcal{O}$. Since $E_W$ is the inverse image of $W\subset Q$, then $g_W$ factor though $\mathcal{O}\subset t^{-1}\mathcal{O}$ if and only if the restriction of pairing on $Q$ to $W$ is zero, i.e., $W$ is isotropic. And it is easy to see $g_W$ is nondegenerate if and only if $W$ is maximal isotropic.
\end{proof}

Now we deal with another condition on $W$, i.e., the partition of $\theta_W(0)$ equals $\bf{d}_{B}$. By our assumption on $\bf{d}_B$, the space $Q$ decomposes as 
\[
    Q = \bigoplus_{i=0}^k Q_i = \bigoplus_{i=0}^k \Coker(K_{B, i}\rightarrow K_{B, i}^{\vee}).
\]
For $i\geq 1$, $\Phi_i$ is defined as multiplication by $(f_{B, i}(0)/t)|_{t=0}$ on $Q_i$. For $i=0$, we define $\Phi_0$ as $(a_{B, 2n}/t^{\operatorname{ord}_t(a_{B,2n})})|_{t=0}$. We use $\Phi$ to denote the direct sum of $\Phi_i$, which is a linear morphism on $Q$.

To determine the structure of $\iota$-isotropic subspaces in $Q$, we analyze the decomposition and filtration of $Q$, as well as the associated constraints on subspaces $W \subset Q$.

\begin{definition} \label{F geq leq i}
     Let $\bf{d}_B=[\bf{T}_{1}, \bf{T}_{2}, \ldots]$ satisfies $\diamondsuit$. Then 
     $$
     \bf{deg}_B=[e_{B, 1}, e_{B,2}. \ldots] = [{}^S\bf{T}_{1}, {}^S\bf{T}_{2}, \ldots].
     $$ 
     For each $i\in \mathbb{N}$, we define the ascending and descending filtrations on $Q$ and $W$:
     \begin{itemize}
         \item $F_{\geq i}=\oplus_{e_{B,j}\geq i}Q_j$;
         \item $F^{\leq i}=\oplus_{e_{B,j}\leq i}Q_j$; 
         \item $F_{=i}=F_{\geq i}\cap F^{\leq i}$; 
         \item We will use $\Phi_{=i}$ to denote the restriction of $\Phi$ on $F_{=i}$;
         \item $W_{\geq i}=W\cap F_{\geq i}$, $W^{\leq i}= W\cap F^{\leq i}$ and $W_{=i}=W\cap F_{=i}$;
         \item We use $\overline{W_{\geq i}}$, $\overline{W^{\leq i}}$ and $\overline{W_{= i}}$ to denote the projections of $W$ to $F_{\geq i}$, $F^{\leq i}$ and $F_{=i}$ respectively.
     \end{itemize}
\end{definition}

 We fix a basis for $K_{B,i}$ such that the matrix of $\theta_{B,i}$ is given by $R(f_{B,i}(\lambda))$ and hence the matrix of $\theta_{B,i}^{\vee}$ is given by $R(f_{B,i}(\lambda))^{\top}$. The matrix form of $\oplus_{i=0}^k\theta_{B,i}^{\vee}$ is given by $\Theta^{\vee}=\text{diag}\{R(f_{B,i}(\lambda))^{\top}\}_{0\leq i \leq k}$. Hence the matrix of $\theta_W$ is given by $\Theta_W=P_W^{-1}\Theta^{\vee} P_W$ for some $P_W\in \mathrm{Mat}(\mathcal{O})$ determined by $W$.

\begin{lemma}
   Using the basis we fixed before, we have a decomposition 
   $$
   (\oplus_{i=0}^{k}K_{B, i}^{\vee})(0)\cong \Im \oplus_{i=0}^k\theta^\vee_{B, i}(0) \oplus Q.
   $$ 
   And then $\Im (P_W(0))=\Im \oplus_{i=0}^k\theta^\vee_{B, i}(0)\oplus W$.
   
   On the other hand, we define $\mathcal{T}_i=\text{diag}\{1, \cdots , 1, t\}$ with size $e_{B,i}$ for each $0\leq i \leq k$ and let $\mathcal{T}=\text{diag}\{\mathcal{T}_0, \cdots , \mathcal{T}_{k}\}$. Then $\tilde{P_W^{-1}}=P_W^{-1}\mathcal{T}\in \mathrm{Mat}(\mathcal{O})$ and $\Ker (\tilde{P_W^{-1}}(0))=W$. Moreover, we put $\tilde{\Theta^{\vee}}=\mathcal{T}^{-1}\Theta^\vee$, and we have $\tilde{\Theta^{\vee}}(0)$ is invertible over $\mathcal{O}$.
\end{lemma}

\begin{proof}
    All the statements in this Lemma can be easily deduced by our choice of basis.
\end{proof}

We recall that $\bf{d}_{B}= \bf{T}_B=[\bf{T}_1, \ldots , \bf{T}_s]$ as in Lemma \ref{structure of special partition}. Here $\bf{T}_i$ is of either type B1, B2, or B3. We decompose $Q$ further with respect to the $\bf{T}_j$-types in $\bf{d}_B$ as: 
\begin{align*}
    Q=\oplus_{j=1}^sQ^{\bf{T}}_j, \quad Q^{\bf{T}}_j= 
    \begin{cases}
        \oplus_{e_{B,i}\in ^S\bf{T}_j}Q_i, \quad \text{for} \quad 1 \leq j < s, \\
        \oplus_{e_{B,i}\in ^S\bf{T}_s}Q_i \oplus Q_0 , \quad \text{for} \quad j=s.
    \end{cases}
\end{align*}

Our definitions here are slightly subtle, so we shall give an example to clarify all the notions above.

\begin{example} \label{exmaple_for_explanation}
Consider the partition $\bf{d}_{B}=[7,6,6,4,4,2,2,1,1]$, its Springer dual is given by $^S\bf{d}_{B}=[6,6,6,4,4,2,2,2]$. Then $\bf{deg}_B=[e_{B, 1}, \ldots ,e_{B,8}, e_{B,0}]= [6,6,6,4,4,2,2,2,1]$. By Proposition \ref{part 2.1} and the pair defined on $K_{B, 0}$, we have a $1$ dimensional space $Q_i$ for $1\leq i \leq 8$. So $Q=\oplus_{i=0}^8Q_i$ Then $F_{\geq 6}=F_{\geq 5}=Q_1\oplus Q_2 \oplus Q_3$, $F_{\geq 4}=\oplus_{i=1}^5Q_i$ and etc. Hence we have an ascending filtration as $$F_{\geq 6}\subset F_{\geq 4}\subset F_{\geq 2}\subset F_{\geq 1}=Q.$$ On the other hand, we have $F^{\leq 1}=Q_0$, $F^{\leq 2}=Q_6\oplus Q_{7}\oplus Q_8\oplus Q_0$, $F^{\leq 3}=F^{\leq 4}=\oplus_{i=4}^8Q_i\oplus Q_0$ and etc. Hence we have a descending filtration as $$Q=F^{\leq 6}\supset F^{\leq 4}\supset F^{\leq 2} \supset F^{\leq 1}.$$ Moreover, $F_{=6}=Q_1\oplus Q_2 \oplus Q_3$, $F_{=4}=Q_4\oplus Q_5 $ and etc. . $Q^{\bf{T}}_1=\oplus_{i=1}^8Q_i$ and $Q^{\bf{T}}_2= Q_0$.
\end{example}

Now we are going to determine the structure of $\iota$-isotropic subspaces.

\begin{lemma}\label{equation of W}
    Each $\iota$-isotropic subspace $W$ has a decomposition 
    \begin{align*}
    W=\bigoplus_{j=1}^s(W\cap Q^{\bf{T}}_j).
    \end{align*}
\end{lemma}

\begin{proof}
        We calculate the ranks of $\Theta_W^i(0)$ for each $i$ to locate its partition. In order to make notations simpler, in this proof, for any matrix $A\in \mathrm{Mat}(\mathcal{O})$, we denote the matrix $A(0)\in \mathrm{Mat}(\Bbbk)$ also by $A$. We denote $m=\dim W$.
    
    Then for any $l\geq 1$, 
\begin{align*}
        \text{rk} (\Theta_W)^l =& \rk \tilde{P_W^{-1}} \tilde{\Theta^{\vee}} (\Theta^{\vee})^{l-1}P_W \\
                             =& \rk \tilde{\Theta^{\vee}} (\Theta^{\vee})^{l-1}P_W-\dim (\Im \tilde{\Theta^{\vee}} (\Theta^{\vee})^{l-1}P_W\cap \Ker \tilde{P_W^{-1}}) \\
                             =& \rk (\Theta^{\vee})^{l-1}P_W-\dim (\Im \tilde{\Phi} (\Theta^{\vee})^{l-1}P_W\cap W) \\
                             =& \rk P_W -\dim (\Im P_W\cap \Ker (\Theta^{\vee})^{l-1}) \\
                              &-\dim (\Im \tilde{\Theta^{\vee}} (\Theta^{\vee})^{l-1}P_W\cap W) \\
                             =& N-m-\dim \big((\Im (\Theta^{\vee}) \oplus W)\cap \Ker (\Theta^{\vee})^{l-1}\big) \\
                              &- \dim (\Im \tilde{\Theta^{\vee}} (\Theta^{\vee})^{l-1}P_W\cap W) \\
                             =& N-m-\dim \big(\Im (\Theta^{\vee})\cap \Ker (\Theta^{\vee})^{l-1}\big) -\dim(W\cap \Ker (\Theta^{\vee})^{l-1})\\
                              &-\dim (\Im \tilde{\Theta^{\vee}} (\Theta^{\vee})^{l-1}P_W\cap W) \\
                             =& N-m+\rk (\Theta^{\vee})^l-\rk (\Theta^{\vee}) -\dim(W\cap \Ker (\Theta^{\vee})^{l-1}) \\
                              &-\dim (\Im \tilde{\Theta^{\vee}} (\Theta^{\vee})^{l-1}P_W\cap W) \\
                             =& \rk (\Theta^{\vee})^l+m-\dim(W\cap \Ker (\Theta^{\vee})^{l-1}) \\
                              &-\dim (\Im \tilde{\Theta^{\vee}} (\Theta^{\vee})^{l-1}P_W\cap W) \\
                             =& \rk (\Theta^{\vee})^l+m-\dim(W\cap \Ker (\Theta^{\vee})^{l-1}) \\
                              &-\dim \big( \tilde{\Theta^{\vee}} (\Theta^{\vee})^{l-1}(\Im (\Theta^{\vee}) \oplus W)\cap W\big) \\
                             =& \rk (\Theta^{\vee})^l+m-\dim(W\cap \Ker (\Theta^{\vee})^{l-1}) \\
                              &-\dim \big(\tilde{\Theta^{\vee}} (\Im (\Theta^{\vee})^l \oplus (\Theta^{\vee})^{l-1}W)\cap W\big) \\
                             =& \rk (\Theta^{\vee})^l+m-\dim(W\cap \Ker (\Theta^{\vee})^{l-1}) \\
                              &-\dim \Big(\big(\tilde{\Theta^{\vee}} (\Im (\Theta^{\vee})^l) \oplus \tilde{\Theta^{\vee}} ((\Theta^{\vee})^{l-1}W)\big)\cap W\Big) \\
                             =& \rk (\Theta^{\vee})^l+\dim W-\dim(W\cap F^{\leq l-1}) \\
                              &-\dim \Big(W\cap \big(F_{\geq l+1}\oplus \Phi_{=l}(\overline{W_{\geq l}}\cap \overline{W_{=l}})\big)\Big).
\end{align*}
    Notice that the partition of $\Theta^{\vee}$ is $^S\bf{d}_B$ and we want the partition of $\Theta_W$ to be $\bf{d}_B$. So for those $l$ equal to the last part of some $\bf{T}_j$ minus $1$, we have 
    $$
    \dim W=\dim (W\cap F^{\leq l-1})+\dim (W\cap F_{\geq l+1}).
    $$ 
    This implies that $W=(W\cap F^{\leq l-1})\oplus (W\cap F_{\geq l+1})$. When we consider all the $l$'s, we will have our result.
\end{proof}

This lemma has a valuable consequence that will simplify our proof in the description of $W$ and give us a better understanding of the structure of $\theta_B$.

Since $\bf{deg}_B = [ {}^S \bf{T}_B, 1 ]$, here ${}^S\bf{T}_B = [{}^S\bf{T}_1, \ldots, {}^S\bf{T}_s]$. Suppose ${}^S\bf{T}_i = [\deg f_{B, i1}(\lambda), \deg f_{B, i2}(\lambda), \ldots ]$, let
\begin{align} \label{T_i}
    T_i(\lambda) = 
    \begin{cases}
        \prod_{j \geq 1} f_{B, ij}(\lambda), \quad \text{for} \quad  1 \leq I <s, \\
        \lambda\prod_{j \geq 1} f_{B, sj}(\lambda), \quad \text{for} \quad  i=s.
    \end{cases}
\end{align}
Then we have a decomposition $\chi (\theta_B)=T_1(\lambda) \cdots T_s(\lambda)$.

\begin{corollary}\label{rough decomposition of E}
    Using the above notations, the module $E_B$ naturally decomposes as $E_B \cong \bigoplus_{i=1}^s E_{B, i}$, where $E_{B, i} = \Ker T_i(\theta_B)$. Moreover, the restriction of $g$ on $E_{B, i}$ is nondegenerate, and we have exact sequences: 
    $$
    0\longrightarrow \bigoplus_{f_{B, j}\mid T_i}\Ker f_{B, j}(\theta_B) \longrightarrow E_{B, i} \longrightarrow Q^{\bf{T}}_i\longrightarrow 0
    $$ 
    with $\dim_{\Bbbk}Q^{\bf{T}}_i=\lceil \#\{\text{degree even factors of $T_i$}\}/2 \rceil$.
\end{corollary}
\begin{proof}
    It follows from Proposition \ref{coker dimension} and Lemma \ref{equation of W}.
\end{proof}

The following Lemma will be helpful in the later proof.

\begin{lemma}\label{solution of vector spaces}
   Let $V$ be a finite-dimensional vector space and $\Phi$ be an invertible linear operator on $V$. If we have subspaces $0\neq U\subset Z \subseteq V$ such that $\dim Z-\dim U=1$, $\Phi(U)\subseteq Z$ and there is no $\Phi$ invariant subspace in $U$ other that $0$. Then we have a vector $z\in Z$, unique up to scalar, such that $U= \langle \Phi^{-1} z, \cdots , \Phi^{-\dim U}z \rangle $ and $Z= \langle z,  \Phi^{-1} z, \cdots , \Phi^{-\dim U}z \rangle$.
\end{lemma}

\begin{proof}
    We may assume $U\neq 0$; otherwise, it is trivial. We have the following chain of vector spaces: 
    $$
    U\supseteq U\cap \Phi U\supseteq \cdots \supseteq \cap_{i=0}^n\Phi^{i}U\supseteq \cdots .
    $$ 
    We claim that $\dim \cap_{i=0}^j\Phi^{i}U-\dim \cap_{i=0}^{j+1}\Phi^{i}U=1$ unless $\cap_{i=0}^j\Phi^{i}U=0$. Since $\Phi$ is invertible, $\dim \Phi U = \dim U$. Then $U=U\cap Z\supseteq U\cap \Phi U$ and $\dim Z-\dim \Phi U=1$ tell that $\dim U-\dim U\cap \Phi U\leq 1$, so $\dim U-\dim U\cap \Phi U = 1$. Otherwise, $U=U\cap \Phi U$, then $\Phi U=U$ contradicts to assumption. Consider 
    $$
    \Phi(\cap_{i=0}^{j}\Phi^{i}U)=\cap_{i=1}^{j+1}\Phi^{i}U\supseteq \cap_{i=0}^{j+1}\Phi^{i}U.
    $$ 
    Then if $\cap_{i=0}^j\Phi^{i}U=\cap_{i=0}^{j+1}\Phi^{i}U$, we see that $\cap_{i=0}^j\Phi^{i}U$ is a $\Phi$ invariant subspace of $U$, which must be $0$ by assumption. Thus $\dim \cap_{i=0}^j\Phi^{i}U-\dim \cap_{i=0}^{j+1}\Phi^{i}U\geq 1$ unless $\cap_{i=0}^j\Phi^{i}U=0$. Notice that $\cap_{i=0}^{j+1}\Phi^{i}U=U\cap \Phi \cap_{i=0}^{j}\Phi^{i}U$ and $\cap_{i=0}^{j}\Phi^{i}U=U\cap \Phi \cap_{i=0}^{j-1}\Phi^{i}U$ and $\dim \cap_{i=0}^j\Phi^{i-1}U-\dim \cap_{i=0}^{j}\Phi^{i}U= 1$ by induction. Hence $\dim \cap_{i=0}^j\Phi^{i}U-\dim \cap_{i=0}^{j+1}\Phi^{i}U\leq 1$. Then, we have proved our claim.
    
    Thus we have $\dim \cap_{i=1}^{\dim U-1}U=1$ and we take a nonzero vector $u$ in it. Then $u\in \Phi^i U$ tells that $\Phi^{-i}u\in U$ for $0\leq i \leq \dim U-1$. Notice that 
    $$
    \{u, \Phi^{-1} u, \cdots , \Phi^{1-\dim U}u\}
    $$ 
    is linearly independent. Otherwise, $U$ would have nontrivial $\Phi$ invariant subspace. So 
    \begin{align*}
        U &= \langle u, \Phi^{-1} u, \cdots , \Phi^{1-\dim U}u \rangle, \\
        \Phi U &= \langle \Phi u, u,  \Phi^{-1} u, \cdots , \Phi^{2-\dim U}u \rangle.
    \end{align*}
    Since $\Phi U\neq U$ so $Z=U+\Phi U$, i.e., $Z = \langle \Phi u, u,  \Phi^{-1} u, \cdots , \Phi^{1-\dim U}u \rangle$. Then taking $z=\Phi u$ we arrive at the conclusion.
\end{proof}

Let $ \bf{d}_B $ satisfies $\diamondsuit$. Let $\bf{deg}_B = [\bf{deg}_1, \bf{deg}_2, \ldots, \bf{deg}_N]$, here $\bf{deg}_i$ is a partition consisting of the degree of the irreducible factor in $T_i$, see (\ref{T_i}). Let $\bf{deg}_i = [e_{i,1}, e_{i,2}, \ldots, e_{i,2k_i}] = [r_{i,1}^{m_{i, 1}}, \cdots , r_{i, q_i}^{m_{i, q_i}}] $ such that $r_{i, 1} > \cdots > r_{i, q_i}$. Now we can give a detailed description of $W$ as follows.

\begin{proposition}\label{description of W}
    Assume that $\chi_{\theta_B}(\lambda)$ is generic in $L\mathfrak{c}_{\bf{O}_B}$.\footnote{This generic condition needs more description than the one in Theorem \ref{Thm.type B decomposition}, as we will see in the proof.} For each $\iota$-isotropic subspace $W\subseteq Q$, we have, for $i=1, \ldots, N$ and $j=1, \ldots, q_i$, vectors $w_{i,j}\in F_{=r_{i, j}}$, unique up to scalar, such that 
    \begin{align*}
        W= \bigoplus_{i=1}^N \langle & \Phi^{-1}w_{i, 1}, \cdots , \Phi^{-d_{i, 1}}w_{i, 1}, a_{i, 1}\Phi^{-d_{i, 1}-1}w_{i, 1}+b_{i, 2}w_{i, 2}, \\
           & \cdots , \\
           & \Phi^{-1}w_{i, j}, \cdots , \Phi^{-d_{i, j}}w_{i, j}, a_{i, j} \Phi^{-d_{i, j}-1}w_{i, j} + b_{i, j+1} w_{i, j+1}, \\ 
           & \cdots , \\
           & \Phi^{-1}w_{i, q_i}, \cdots , \Phi^{-d_{i, q}}w_{i, q_i}\rangle.\\ 
    \end{align*}
    Here all the $a_{i, j}$'s, $b_{i, j}$'s are nonzero and $d_{i,j}= \lfloor \frac{m_{i,j}}{2} \rfloor-1$. As a consequence, we have exactly $2^{\beta(\bf{d}_B) - c(\bf{d}_B)}$ many $\iota$-isotropic subspaces.
\end{proposition}

\begin{remark}\label{W is defined over general field}
    The description of $W$ in Proposition \ref{description of W} also works over a non-algebraically closed field, as long as the subspaces $F_{\geq i}$ and $F^{\leq i}$ are defined over the non-algebraically closed field.
\end{remark}

The proof of this Proposition is subtle, so a suggestion is to keep an example in mind. The simplest example to capture the essence of the proof is $\bf{d}_B=[7,6,6,4,4,2,2,1,1]$. In this case, $F_{\geq i}$ and $F^{\leq i}$ are given as in Example \ref{exmaple_for_explanation}.

\begin{proof}
    By Lemma \ref{equation of W}, we only need to determine each subspace $W\cap Q^{\bf{T}}_j$. From now on, we shall assume the partition $\bf{d}_B=[e_1+1, e_2, \ldots, e_{2m-1}, e_{2m}-1]$ is a partition of type B2. Then $ \bf{deg}_B = {}^S\bf{d}_B $ and we write $${}^S\bf{d}_B=[e_1, e_2, \ldots, e_{2k-1}, e_{2k}] = [r_1^{m_1}, \ldots, r_q^{m_q}].$$
    The type B3 case follows similarly.
    
    By Proposition \ref{isotropic}, an $\iota$-isotropic subspace must be isotropic. By taking $f$ generic in $L\mathfrak{c}_{ \bf{O}_B}$, we may assume that the eigenvalues of $\Phi$ are pairwise distinct. So any nonzero subspaces of $W$ can not be $\Phi$ invariant. 

    We want the partition of $\Theta_W=P_W^{-1}\Theta^{\vee} P_W$ to be the given $$\bf{d}_B=[e_1+1, e_2, \cdots , e_{2m-1}, e_{2m}-1],$$while the partition for $\Theta^{\vee}$ is ${}^S\bf{d}_B = [e_1, e_2, \cdots , e_{2m-1}, e_{2m}]$. Then $\rk \Theta_W^l-\rk (\Theta^{\vee})^l=1$ if $e_{2m}\leq l \leq e_1$ and $\rk \Theta_W^l-\rk (\Theta^{\vee})^l=0$ otherwise. As shown in the proof of Lemma \ref{equation of W}, for $e_{2m}\leq l \leq e_1$,
    \begin{align}
    \dim W-\dim W^{\leq l-1}-\dim \Big(W\cap \big(F_{\geq l+1}\oplus \Phi (\overline{W_{\geq l}}\cap \overline{W_{=l}})\big)\Big)=1. \label{key_identity}
    \end{align}
       
    \textbf{Case I.} We begin with $l=e_{2m}$ . In this case, equation (\ref{key_identity}) implies  
    $$
    \dim W-\dim \Big(W\cap \big(F_{\geq e_{2m}+1}\oplus \Phi (W_{=e_{2m}})\big)\Big)=1.
    $$ 
    Notice that $W = W_{\geq e_{2m}}$ and 
    \begin{align*}
     &\dim \Big(W\cap \big(F_{\geq e_{2m}+1}\oplus \Phi (W_{=e_{2m}})\big)\Big) \\
    =&\dim W_{\geq e_{2m}+1}+\dim \Phi (W_{=e_{2m}})\cap \overline{W_{=e_{2m}}}.
    \end{align*}
    Hence 
    $$\dim \overline{W_{=e_{2m}}}-\dim \Phi (W_{=e_{2m}})\cap \overline{W_{=e_{2m}}}=1.
    $$ 
    Then we take $l=e_{2m}+1$. Since $e_i \equiv 0$, there is no $e_i$ equals $e_{2m}+1$, then $\overline{W_{=e_{2m}+1}} = 0 $. From (\ref{key_identity}) we have $\dim \overline{W_{=e_{2m}}}-\dim W_{=e_{2m}}=1$. Hence $W_{=e_{2m}}\subseteq \overline{W_{=e_{2m}}}$, $\Phi (W_{=e_{2m}})\subseteq \overline{W_{=e_{2m}}}$ and $\Phi (W_{2m}) \neq W_{=e_{2m}} $. By Lemma \ref{solution of vector spaces}, we have a vector $w_{2m}\in F_{=e_{2m}}$, unique up to scalar, such that
    \begin{align*}
        W_{=e_{2m}} &= \langle \Phi^{-1}w_{2m} \cdots , \Phi^{-d_{2m}}w_{2m} \rangle, \\
        \overline{W_{=e_{2m}}} &= \langle w_{2m}, \Phi^{-1}w_{2m} \cdots , \Phi^{-d_{2m}}w_{2m} \rangle.
    \end{align*}
    Here $d_{2m}$ is the dimension of $W_{=e_{2m}}$, which will be determined later.
    
   \textbf{Case II.} When $e_{2m}<l=e_i<e_1$ for some $e_i$, the equation (\ref{key_identity}) becomes 
    $$
    \dim W-\dim W^{\leq e_i-1}-\dim \Big(W\cap \big(F_{\geq e_i+1}\oplus \Phi (\overline{W_{\geq e_i}}\cap \overline{W_{=e_i}})\big)\Big)=1,
    $$
    which can be simplified as follows
    $$
    \dim\overline{W_{\geq e_i}}-\dim W_{\geq e_i+1}-\dim\Phi (\overline{W_{\geq e_i}}\cap \overline{W_{=e_i}})\cap \overline{(W_{\geq e_i})_{=e_i}}=1.
    $$ 
    Here $\overline{(W_{\geq e_i})_{=e_i}}$ is the projection of $W_{\geq e_i}$ to $F_{=e_i}$. We have an exact sequence 
    $$
    0\longrightarrow \overline{W_{\geq e_i}}\cap F_{=e_i}\longrightarrow \overline{W_{\geq e_i}} \longrightarrow \overline{W_{\geq e_i+1}}\longrightarrow 0.
    $$ 
    Notice that $\overline{W_{\geq e_i}}\cap F_{=e_i}=\overline{W_{\geq e_i}}\cap \overline{W_{=e_i}}$. 
    
    Taking $l=e_i+1$, notice that $\overline{W_{e_i+1}}=0$, then from (\ref{key_identity}) we have 
    \begin{align}
        \dim \overline{W_{\geq e_i+1}}-\dim W_{\geq e_i+1}=1. \label{key_identity_2}
    \end{align}
    Hence we have $\dim \overline{W_{\geq e_i}}\cap \overline{W_{=e_i}}=\dim\Phi (\overline{W_{\geq e_i}}\cap \overline{W_{=e_i}})\cap \overline{(W_{\geq e_i})_{=e_i}}$ from the argument in the above $l=e_i$ case. Then 
    $$
    \overline{W_{\geq e_i}}\cap \overline{W_{=e_i}} \subseteq \Phi^{-1}\overline{(W_{\geq e_i})_{=e_i}}.
    $$ 
    Thus we obtain the following relation between $W_{=e_i}$ and $\overline{W_{=e_i}}$ 
    \begin{equation}\label{inclutions}
    W_{=e_i}\subseteq \overline{W_{\geq e_i}}\cap \overline{W_{=e_i}} \subseteq \Phi^{-1}\overline{(W_{\geq e_i})_{=e_i}}\subseteq \Phi^{-1}\overline{W_{=e_i}}.
    \end{equation}
    We will analyze the inclusions in (\ref{inclutions}) to see which parts are identities and which are proper inclusions.
    
    Firstly, we show that $\dim \overline{(W_{\geq e_i})_{=e_i}}-\dim W_{=e_i}=1$. We use the following equalities
    \begin{align}
    \begin{split}
        \dim W_{\geq e_i} & = \dim W_{\geq e_i+1} + \dim \overline{(W_{\geq e_i})_{=e_i}} \\
                          & = \dim W_{=e_i} + \dim \overline{(W_{\geq e_i})_{\geq e_i+1}}. \label{identity_3}
    \end{split}
    \end{align}
    Here $\overline{(W_{\geq e_i})_{\geq e_i+1}}$ is the projection of $W_{\geq e_i}$ to $F_{\geq e_i+1}$. So we have $W_{\geq e_i+1}\subseteq \overline{(W_{\geq e_i})_{\geq e_i+1}} \subseteq \overline{W_{\geq e_i+1}}$. Together with equations (\ref{key_identity_2}) and (\ref{identity_3}) we obtain 
    $$
    \dim \overline{(W_{\geq e_i})_{=e_i}}-\dim W_{=e_i}\leq 1.
    $$

    If $\dim \overline{(W_{\geq e_i})_{=e_i}} = \dim W_{=e_i}$, notice that $W_{=e_i}\subseteq \Phi^{-1}\overline{(W_{\geq e_i})_{=e_i}}$, so we have $\overline{(W_{\geq e_i})_{=e_i}}= W_{=e_i}$. Hence $W_{=e_i}$ is $\Phi$ invariant. By the generality of $\Phi$, this forces $\overline{(W_{\geq e_i})_{=e_i}}= W_{=e_i}=0$. We exclude the latter case using a dimension argument, which will be done in the part of the dimension discussion later in this proof. 
    
    Then $\dim \overline{(W_{\geq e_i})_{=e_i}}-\dim W_{=e_i}= 1$, together with $W_{=e_i}\subseteq \overline{(W_{\geq e_i})_{=e_i}}$ and $\Phi W_{=e_i}\subseteq \overline{(W_{\geq e_i})_{=e_i}}$, by Lemma \ref{solution of vector spaces}, we have a vector $w_{i}\in F_{=e_i}$, unique upto scalar, such that
    \begin{align*}
        W_{=e_i} &= \langle \Phi^{-1} w_{i}, \cdots , \Phi^{-d_{i}} w_{i} \rangle, \\
        \overline{(W_{\geq e_i})_{=e_i}} &= \langle w_{i}, \Phi^{-1} w_{i}, \cdots , \Phi^{-d_{i}} w_{i} \rangle.
    \end{align*}
    Here $d_i$ is the dimension of $W_{=e_{i}}$, which will be determined later. We remark here that $W_{=e_i}$ can be zero and in this case $\overline{(W_{\geq e_i})_{=e_i}} = \langle w_{i} \rangle$. Now we claim that 
    \begin{claim}
    $$
    \overline{W_{= e_i}} = \langle w_{i} \rangle +W_{=e_i}+ \langle \Phi^{-d_{i}-1} w_{i} \rangle 
    $$
    such that $w_{i} \notin \overline{W_{\geq e_i}}$, $w_{i}\in \overline{(W_{\geq e_i})_{=e_i}}$ and $\Phi^{-d_{i}-1} w_{i} \in \overline{W_{\geq e_i}}$, $\Phi^{-d_{i}-1} w_{i} \notin \overline{(W_{\geq e_i})_{=e_i}}$.
    \end{claim}
    
    \textbf{Case III.} With this claim in mind, if we consider the case $l=e_1$ further, we will have a description of $W$. For $l=e_1$, we have 
    $$
    \dim W- \dim (W\cap F^{\leq e_1-1})- \dim \Big(W\cap \big(F_{\geq e_{1}+1}\oplus \Phi(\overline{W_{\geq e_{1}}}\cap \overline{W_{=e_{1}}})\big)\Big)=1,
    $$ 
    which simplifies to 
    $$
    \dim \overline{W_{=e_1}}-\dim W_{=e_1}\cap \Phi(\overline{W_{=e_1}})=1.
    $$ 
    Similar as the case of $l=e_{2m}$, we have a vector $w_{1}$, unique up to scalar, such that $W_{=e_1} = \langle \Phi^{-1}w_{1}, \cdots , \Phi^{-d_{1}}w_{1} \rangle$ and $\overline{W_{=e_1}}=W_{=e_1}+ \langle \Phi^{-d_{1}-1}w_{1} \rangle$.
    
    Recall that $[e_1, e_2, \ldots, e_{2m}]=[r_1^{m_1}, \cdots , r_q^{m_q}]$ as partitions where $r_1 > \cdots > r_q$. Now it is easy to see that 
    \begin{align*}
        W=\langle& \Phi^{-1}w_{1}, \cdots , \Phi^{-d_{1}}w_{1}, a_{1}\Phi^{-d_{1}-1}w_{1}+b_{2}w_{2}, \\
           & \cdots , \\
           & \Phi^{-1}w_{j}, \cdots , \Phi^{-d_{j}}w_{j}, a_{j} \Phi^{-d_{j}-1}w_{j} + b_{j+1} w_{j+1}, \\ 
           & \cdots , \\
           & \Phi^{-1}w_{q}, \cdots , \Phi^{-d_{q}}w_{q}\rangle.\\ 
    \end{align*}Here all the $a_{j}$'s, $b_{j}$'s are nonzero.

    Now we need to prove the claim.    
    
    Recall that we have $\dim \overline{W_{\geq e_i+1}}-\dim W_{\geq e_i+1}=1$, then we have a vector $w_{\geq e_i+1}\in \overline{W_{\geq e_i+1}}\setminus W_{\geq e_i+1}$ such that 
    \begin{align}
        W_{\geq e_i}= W_{\geq e_i+1} + W_{= e_i} + \langle w_{\geq e_i+1} + w_{i} \rangle. \label{W>=e_i}
    \end{align} 
    Then we want to show that 
    $$
    \dim \overline{W_{=e_i}}-\dim \overline{(W_{\geq e_i})_{=e_i}}=1.
    $$
    Similarly, taking $l=e_i-1$ in equation (\ref{key_identity}), we obtain $\dim \overline{W_{\geq e_i}} -\dim W_{\geq e_i}=1$. So $\dim \overline{W_{=e_i}}-\dim \overline{(W_{\geq e_i})_{=e_i}}\leq 1$.
    
    If $\overline{W_{=e_i}}= \overline{(W_{\geq e_i})_{=e_i}}$, by the fact $\dim \overline{W_{\geq e_i}} -\dim W_{\geq e_i}=1$ we have $\dim \overline{W_{\geq e_i}}\cap F_{\geq e_i+1}-\dim W_{\geq e_i}\cap F_{\geq e_i+1}=1$. Notice that $W_{\geq e_i}\cap F_{\geq e_i+1}=W_{\geq e_i+1}$ and $\dim \overline{W_{\geq e_i+1}} -\dim W_{\geq e_i+1}=1$ so we have $\overline{W_{\geq e_i}}\cap F_{\geq e_i+1}=\overline{W_{\geq e_i+1}}$ and hence $w_{\geq e_i+1} \in \overline{W_{\geq e_i}}$. This means that 
    $$ 
    \overline{W_{\geq e_i}} = W_{\geq e_i+1} + W_{= e_i} + \langle w_{\geq e_i+1}, w_{i} \rangle.
    $$ 
    It can not happen. If so, we have 
    $$
    \overline{W_{\geq e_i}}\cap \overline{W_{=e_i}}=\overline{W_{\geq e_i}}\cap \overline{(W_{\geq e_i})_{=e_i}}=\overline{(W_{\geq e_i})_{i}}.
    $$ 
    Recall that the inclusion (\ref{inclutions}): $\overline{W_{\geq e_i}}\cap \overline{W_{=e_i}} \subseteq \Phi^{-1}\overline{(W_{\geq e_i})_{=e_i}}$. Thus $\overline{(W_{\geq e_i})_{=e_i}}=W_{=e_i}+ \langle w_{i} \rangle$ is $\Phi_{=e_i}$ invariant. $w_{i}\in \overline{W_{\geq e_i}}\setminus W_{\geq e_i}$ means that we have $w^{\leq e_i-1}\in F^{\leq e_i-1}$ such that $w_{=e_i}+w^{\leq e_i-1}\in W$. From (\ref{W>=e_i}) we have $w_{\geq e_i+1}+w_{i}\in W_{\geq e_i}\subseteq W$, since $W$ is isotropic, hence $w_{i}$ is also isotropic, which leds that $\overline{(W_{\geq e_i})_{=e_i}}$ is also isotropic. This is a contradiction since we have argued that $\overline{(W_{\geq e_i})_{=e_i}}$ is $\Phi$ invariant and $\overline{(W_{\geq e_i})_{=e_i}}\neq 0$. 
    
    So we have $\dim \overline{W_{=e_i}}-\dim \overline{(W_{\geq e_i})_{=e_i}}=1$. 
    
    Finally, we claim that 
    $$
    \overline{W_{\geq e_i}}\cap \overline{W_{=e_i}}=\Phi^{-1}\overline{(W_{\geq e_i})_{=e_i}}.
    $$ 
    Otherwise, we would have $W_{=e_i}=\overline{W_{\geq e_i}}\cap \overline{W_{=e_i}}$. By our discription of $W_{\geq e_i}$, i.e., (\ref{W>=e_i}), we see that $W_{=e_i}=W_{\geq e_i}\cap \overline{W_{\geq e_i}}$, so $\dim \overline{W_{\geq e_i+1}}-\dim \overline{(W_{\geq e_i})_{\geq e_i+1}}=1$, Thus $\overline{(W_{\geq e_i})_{\geq e_i+1}}=W_{\geq e_i+1}$. But from (\ref{W>=e_i}), we know $\overline{(W_{\geq e_i})_{\geq e_i+1}}\neq W_{\geq e_i+1}$. So $\overline{W_{\geq e_i}}\cap \overline{W_{=e_i}}=\Phi^{-1}\overline{(W_{\geq e_i})_{=e_i}}$, and (\ref{inclutions}) becomes: 
    \begin{equation}\label{real inclutions}
    W_{=e_i}\subset \overline{W_{\geq e_i}}\cap \overline{W_{=e_i}} = \Phi^{-1}\overline{(W_{\geq e_i})_{=e_i}}\subset \Phi^{-1}\overline{W_{=e_i}},
    \end{equation}
    where all the proper inclusions are codimension $1$ subspaces. 
    
    Then we have $\Phi^{-1}\overline{(W_{\geq e_i})_{=e_i}}\subseteq \overline{W_{= e_i}}$, $\overline{(W_{\geq e_i})_{=e_i}}\subseteq \overline{W_{= e_i}}$, and  $\dim \overline{W_{=e_i}}-\dim \overline{(W_{\geq e_i})_{=e_i}}=1$. If $\Phi^{-1}\overline{(W_{\geq e_i})_{=e_i}}=\overline{(W_{\geq e_i})_{=e_i}}$ then $w_{i}\in \Phi^{-1}\overline{(W_{\geq e_i})_{=e_i}}=\overline{W_{\geq e_i}}\cap \overline{W_{=e_i}}\subseteq \overline{W_{\geq e_i}}$. One can show that $\overline{(W_{\geq e_i})_{=e_i}}$ is isotropic, which is a contradiction.
    
    So $\Phi^{-1}\overline{(W_{\geq e_i})_{=e_i}}\neq \overline{(W_{\geq e_i})_{=e_i}}$. Recall that $\overline{(W_{\geq e_i})_{=e_i}} = \langle w_{i} \rangle + W_{=e_i}$, then $\Phi^{-1}\overline{(W_{\geq e_i})_{=e_i}}=W_{=e_i} + \langle \Phi^{-d_{i}-1} w_{i} \rangle$ and hence 
    $$
    \overline{W_{= e_i}} = \langle w_{i} \rangle +W_{=e_i}+ \langle \Phi^{-d_{i}-1} w_{i} \rangle 
    $$ 
    such that $w_{i} \notin \overline{W_{\geq e_i}}$, $w_{i}\in \overline{(W_{\geq e_i})_{=e_i}}$ and $\Phi^{-d_{i}-1} w_{i} \in \overline{W_{\geq e_i}}$, $\Phi^{-d_{i}-1} w_{i} \notin \overline{(W_{\geq e_i})_{=e_i}}$. Hence the claim follows.
    
    Now we discuss the dimensions $d_{j}$. With the form of $W$, we see that $\dim W=\sum_{j=1}^q(d_{j}+1)-1=\frac{1}{2}\sum_{j=1}^qm_j$. Notice that $W_{=r_j}$ is an isotropic subspace of $F_{=r_j}$, so $d_{j}\leq \lfloor \frac{m_j}{2} \rfloor$. For $1< j < q$, then we have two another two linearly independent vectors which are orthogonal to $W_{=r_j}$, this tells that $d_{j}\leq \lfloor \frac{m_j}{2} \rfloor-1$. Thus we must have equalities $d_{j}= \lfloor \frac{m_j}{2} \rfloor-1$ and these equalities match with $\dim W=\sum_{j=1}^q(d_{j}+1)-1=\frac{1}{2}\sum_{j=1}^qm_j$ since $m_1$ and $m_q$ are odd while the rest $m_j$'s are even.
    
    Recall that we remain to deal with the case $\overline{(W_{\geq e_i})_{=e_i}}=W_{=e_i}=0$. We also have $\dim \overline{W_{=e_i}}-\dim \overline{(W_{\geq e_i})_{=e_i}}\leq 1$ and hence $\dim \overline{W_{=e_i}}\leq 1$. So it is easy to see that $\dim W$ would not attach $\frac{1}{2}\sum_{i=1}^qm_i$ if $\overline{(W_{\geq e_i})_{=e_i}}=W_{=e_i}=0$ for any $1< i< q$.
    
    At the end of this proof, we determine all possible $\iota$-isotropic subspaces. By our description of $W$, we see that all possible $\iota$-isotropic subspace would be represented by points in the following space: 
    $$
    \prod_{j=1}^q\mathbb{P}^{m_j-1}\times \prod_{j=1}^{q-1}\mathbb{G}_m.
    $$
    This space has dimension $\sum_{j=1}^qm_j-1$. Notice that we have $<v_1, \Phi v_2>_Q=<\Phi v_1, v_2>_Q$ for any $v_1$, $v_2$ in $Q$. So the isotropic condition on $W$ gives $2d_{1}+1+2d_{2}+2+ \cdots +2d_{q-1}+2+2d_{q}=\sum_{j=1}^qm_j-1$ many quadratic equations. Since $\chi_{\theta_B}(\lambda)$ is generic in $L\mathfrak{c}_{\bf{O}_B}$, we would have exactly $2^{\sum_{j=1}^qm_j-1}$ many $\iota$-isotropic subspaces.
\end{proof}

From the proof we can summarize the following useful corollaries which will be used later: 

\begin{corollary}\label{vector ponent is not zero}
    The component of $w_{j}$ in $Q^{\bf{T}}_{i}$ is not zero for any $j$ such that $e_{B,i}=r_j$. And the different choices of $W$ are given by the change of sign of the vector component of $w_{j}$.
\end{corollary}

\begin{corollary}\label{why factor though}
    If $\mathbf{T}_i=[d_a, d_{a+1}]$ has length $2$, then $W_i\subset Q_a\oplus Q_{a+1}$ is given by an isotropic line, so the choices of $W_i$ lies in $\OG(1, Q_a\oplus Q_{a+1})$.
\end{corollary}

\begin{corollary}\label{cor:E_W}
    With the same notations in Proposition \ref{description of W}. Denote one inverse image of $\Phi^{-1}\omega_{i,1}$ in $E_W$ by $u_{i,1}$ (If $m_{i,1}=1$ then we may take the inverse image of $a_{i,1}\omega_{i,1}+b_{i,2}\omega_{i,2}$) in $\oplus_{j=0}^{2k_i} K_{B,i}^\vee$, and denote one inverse image of $\Phi^{-d_{i, q_i}-1}\omega_{i,q_i}$ by $v_{i,q_i}$ (If $m_{i, q_i}=1$. Then we may take the inverse image of $a_{i, q_i-1}\Phi^{-d_{i, q_i-1}-2}\omega_{i, q_i-1}+b_{i, q_i}\Phi^{-1}\omega_{i, q_i}$) in $\oplus_{j=0}^{2k_i} K_{B,i}^\vee$. Then $u_{i,1}\in E_W$ with 
    $$ 
    \theta^{r_{i,1}}(u_{i,1})/t\notin E_W, \;\text{and } \  \theta^{r_{i,1}+1}(u_{i,1})/t \in E_W.
    $$ 
    For $v_{i,q_i}$, we see that $v_{i,q_i}\notin E_W$ and  
    $$ 
    \theta(v_{i,q_i})\in E_W,\; \text{and }\  \theta^{r_{i,q_i}}(v_{i,q_i})/t \in E_W.
    $$
\end{corollary}

\begin{proof}
    Since $\Phi^{-1}\omega_{i,1} \in W$(resp. $a_{i,1}\omega_{i,1}+b_{i,2}\omega_{i,2}\in W$), we see that $u_{i,1}\in E_W$. Now the image of $\theta^{r_{i,1}}(u_{i,1})/t$ in $Q$ is $\Phi(\Phi^{-1}(\omega_{i,1}))=\omega_{i,1}$(resp. $\Phi(a_{i,1}\omega_{i,1}+b_{i,2}\omega_{i,2})$). And $\theta^{r_{i,1}+1}(u_{i,1})/t\in \oplus_{i=0}^k K_{B,i}$ hence belongs to $E_W$. 
    
    On the other hand, the image of $v_{i,q_i}$ in $Q$ does not belong to $W$, hence $v_{i,q_i}\notin E_W$. Notice that $\theta(v_{i,q_i})\in \oplus_{i=0}^k K_{B,i}$ hence belongs to $E_W$, and the image of $\theta^{r_{i,q_i}}(v_{i,q_i})/t$ in $Q$ is $\Phi(\Phi^{-d_{i, q_i}-1}\omega_{i,q_i})=\Phi^{-d_{i, q_i}}\omega_{i,q_i}\in W$(resp. $\Phi(a_{i, q_i-1}\Phi^{-d_{i, q_i-1}-2}\omega_{i, q_i-1}+b_{i, q_i}\Phi^{-1}\omega_{i, q_i})=a_{i, q_i-1}\Phi^{-d_{i, q_i-1}-1}\omega_{i, q_i-1}+b_{i, q_i}\omega_{i, q_i}\in W$) and hence $\theta^{r_{i,q_i}}(v_{i,q_i})/t \in E_W$.
\end{proof}

\section{New perspective of Lusztig's canonical quotient}\label{Sec:seesaw_for_group_action}

In this section, we focus on the Richardson case and study the relations between affine Spaltenstein fibers and residually nilpotent local Higgs bundles. Finally, we prove Theorem~\ref{Thm:intro local level}.

\subsection{Preliminary on Spaltenstein fibers}\label{s.spal_fiber}

\begin{definition}\label{def:dual_pol}
    Two parabolic subgroups are called associated if they have conjugate Levi subgroups. Furthermore, $P_B<\SO_{2n+1}$ and $P_C<\Sp_{2n}$ are called dual if they are associated to Langlands dual parabolic subgroups.
\end{definition}

Suppose $\bf{O}_{B,R}$, $\bf{O}_{C,R}$ are Springer dual Richardson orbits with dual polarizations $(P_B$, $P_C)$, i.e., $P_B$ and $P_C$ are dual. We have the following relation between Springer dual and Langlands duality:

\begin{proposition}[Proposition 3.1 \cite{FRW24} and Corollary 3.6 \cite{FRW24}] \label{Prop:Seesaw_prop}
\begin{align}
\begin{split}\label{intro.diagram} 
\xymatrix{ 
 T^*( \SO_{2n+1}/P_B ) \ar[rd]  \ar[dd]_{\mu_{P_B}} & \stackrel{Langlands \  dual}{\leftrightsquigarrow} & T^*( \Sp_{2n} / {P_C} ) \ar[ld] \ar[dd]^{\mu_{P_{C}}}  \\   
  &  (Z_{P_B}, Z_{P_C}) \ar[ld] \ar[rd] &  \\
 \bar{ \bf{O}}_{B, R} &\stackrel{Springer  \ dual}{\leftrightsquigarrow}  & \bar{\bf{O}}_{C, R} 
} 
\end{split}.
\end{align}
The degrees of the maps satisfy the following seesaw property 
\begin{align*}
    \deg \mu_{P_B} \cdot \deg \mu_{P_C} = \#\bar{A}(\bf{O}_{B,R}) = \#\bar{A}(\bf{O}_{C,R}), \label{eq.intro_seesaw}
\end{align*}
where $\bar{A}(-)$ represents Lusztig's canonical quotient. Furthermore, the Stein factorization pair $(Z_{P_B}, Z_{P_C})$ is a mirror pair, i.e., they share the same stringy Hodge numbers.
\end{proposition}

For later use, we will give a concrete description of generic fibers of $\mu_{P_{B/C}}$. Let $\bf{d}_{B/C}=[d_1, \ldots, d_{r}]$ denote the partition of (any nilpotent orbit) $\bf{O}_{B/C}$. The partition $\bf{d}_{B/C}$ defines a Young tableau $\text{Y}(\bf{d}_{B/C}) \subset \mathbb{Z}^2_{>0}$, where $(i, j) \in \text{Y}(\bf{d}_{B/C}) $ if and only if $1 \leq j \leq r$ and $1 \leq i \leq d_j$.

Let $\nu = 1$ for type B and $\nu=0$ for type C. We choose a Jordan basis of $\mathbb{C}^{2n+ \nu}$, $\left\{ e(i,j) \right\}_{(i,j) \in Y(\bf{d}_{B/C})}$, for $X \in \bf{O}_{B/C}$ as follows
\begin{itemize}
    \item $X \cdot e(i,j) = e(i-1, j) $ for $i > 1$, and $X \cdot e(1,j) = 0 $.
    \item $\langle e(i,j), e(p,q)  \rangle_{\nu} \neq 0 $ if and only if $i+p=d_j+1$ and $q=\tau(j)$. Here $\langle -, - \rangle_{\nu}$ is the pairing on $\mathbb{C}^{2n+\nu}$ and $\tau$ is a permutation of $\{1, \ldots, r \}$ such that $\tau^2=\id$, $d_{\tau(j)}=d_j$, and $\tau(j) \neq j$ if $d_j \not\equiv \nu$. 
\end{itemize}
In the following, we choose a $\tau$ such that $\tau(j) = j$ if $d_j \equiv \nu$.

Let $P_{B/C}$ be one of the polarizations of $\bf{O}_{B/C, R}$ with Levi type $(p_1, \ldots, p_k; q)$, i.e., the Levi subalgebra of $\mathfrak{p}_{B/C} = \rm{Lie}(P_{B/C})$ is is isomorphic to $\gl_{p_1} \oplus \cdots \oplus \gl_{p_k} \oplus \mathfrak{g}'$, here $\frak{g}'=\frak{so}_q$ or $\frak{sp}_q$. 

Define
\begin{itemize}
    \item $\operatorname{ord}(p_1, \ldots, p_k; q) = [d_1, d_2, \ldots]$ where 
    \[
    d_i = \# \left\{ j \mid p_j \in \{p_0 :=q, p_1,\ldots, p_k, p_{k+1}:=p_{k}, \ldots, p_{2k}:=p_1 \}, \; p_j \geq i \right\},
    \]
    
    \item $I(P_B) :=\{j \in \mathbb{N} \mid j \equiv d_j \equiv 0, d_j \geq d_{j+1}+2 \}$,
    
    \item $I(P_C) := \{ j \in \mathbb{N} \mid j\equiv d_j \equiv 1, d_j \geq d_{j+1} +2 \}$,
    
    \item Use Jordan basis $\{e(i,j) \}_{(i,j) \in \text{Y}(\bf{d}_B)}$, let 
    \[
    V_{B,j} = \mathbb{C}e \left( \frac{ d_{B,j}+1}{2}, j \right) \oplus \mathbb{C} e\left( \frac{d_{B,j+1}+1}{2}, j+1 \right), \quad \text{for} \quad j \in I(P_B),
    \]
    
    \item Use Jordan basis $\{e(i,j) \}_{(i,j) \in \text{Y}(\bf{d}_C)}$, let 
    \[
    V_{C,j} = \mathbb{C}e \left(\frac{d_{C,j}}{2}, j \right) \oplus \mathbb{C} e \left(\frac{d_{C,j+1}}{2}, j+1 \right), \quad \text{for} \quad j \in I(P_C).
    \]
\end{itemize}

By \cite[Theorem 7.1]{He78}, we have
\begin{proposition}\label{p.Springer-fiber}
    Using the above notations, the generic fiber of the generalized Springer map
    \begin{align*}
        T^*(G/P)  \xrightarrow{\mu_{P_{B/C}}} \bar{\bf{O}}_{B/C, R}
    \end{align*}
    is isomorphic to 
    \begin{align*}
    \begin{cases}
        & \prod_{j \in I(P_B)} \OG(1, V_{B, j}), \quad \text{for} \quad G = \SO_{2n+1},\; P=P_B, \\
        & \prod_{j \in I(P_C)} \OG(1, V_{C, j}), \quad \text{for} \quad G = \Sp_{2n},\; P=P_C.
    \end{cases}    
    \end{align*}
    Thus, $\deg \mu_{P_{B/C}} = 2^{ \# I(P_{B/C})}$.
\end{proposition}

In the following subsections, we extend Proposition~\ref{Prop:Seesaw_prop} and \ref{p.Springer-fiber} to affine Spaltenstein fibers.

\subsection{Affine Spaltenstein fibers for types B and C}

Let $G=\Sp_{2n}$, and let $P_C < \Sp_{2n}$ be a parabolic subgroup with Levi type $(p_1, \ldots, p_r; q)$. This induces a filtration $\Fil^\bullet_{P_C}$ of $\Bbbk^{2n}$ given by
\begin{align*}
    \Bbbk^{2n} = F^0 \supset F^1 \supset F^2 \supset \ldots \supset F^{r} \supset (F^{r})^\perp \supset \ldots \supset (F^{1})^\perp \supset (F^{0})^\perp = 0,
\end{align*}
where the dimensions satisfy
\begin{align*}
    \dim F^r/(F^r)^\perp = q, \quad \dim F^{i-1}/F^i = p_i \quad \text{for} \quad i=1, \ldots, r.  
\end{align*}
Then Definition \ref{def:affine_Spaltenstein_fiber} can be interpreted as:
$$
{\mathbf{Spal}}_{\theta_C, P_C}=\left\{(E_C, \Fil^{\bullet}_{P_C}) \ \left| \
\begin{aligned}
    &E_C \subset \sK^{2n} \;\text{ is a rank 2n lattice}, \\
    &\text{on which the skew-symmetric}\\
    &\text{pairing} \ g_C \; \text{is perfect}, \\
    &g_C(\theta_C v,w)+g_C(v,\theta_C w)=0, \\
    &\Fil^{\bullet}_{P_C} \;\text{is a filtration as above}, \\
    &\theta_C(0)(F^i) \subset F^{i+1}.
\end{aligned}\right.\right\}.
$$
Similarly, we have ${\mathbf{Spal}}_{\theta_B, P_B}$. 

\begin{lemma}\label{lem:fiber_of_LHiggs_P_B/C} 
    Suppose $\theta_B$ and $\theta_C$ satisfy the relation~\eqref{theta_B/C}, and $\chi(\theta_B)$ is generic in $\mathbf{Char}(\bf{O}_B)$. Then, by the above definition, we have
    \begin{equation} \label{diagram_LHiggs}
        \begin{tikzcd}
            {\mathbf{Spal}}_{\theta_B, P_B}\ar[d, "\nu_{P_B}", swap] & {\mathbf{Spal}}_{\theta_C, P_C}\ar[d, "\nu_{P_C}"] \\
            {\mathbf{Gr}}_{\theta_B, \bf{O}_{B, R}} \ar[r, "l_{BC}"] \ar[ru] & {\mathbf{Gr}}_{\theta_C, \bf{O}_{C, R}}
        \end{tikzcd}.
    \end{equation}
    Moreover,
    \begin{enumerate}
        \item For a $(E_C, \theta_C) \in {\mathbf{Gr}}_{\theta_C, \bf{O}_{C, R}}$, we have 
    \begin{align*}
        {\mathbf{Spal}}_{\theta_C, P_C} \cong
        \{ 
        (\sL, I(P_C)_{\sL}) \mid \sL \in {\mathbf{Gr}}_{\theta_C, \bf{O}_{C, R}}
        \},
    \end{align*}
    i.e.,
    \begin{align*}
        \nu_{P_C}^{-1}((E_C, \theta_C)) & \cong \prod_{j \in I(P_C)} \OG(1, V_{C, j}).
    \end{align*}
    Here $V_{C, j}$'s are as in Proposition \ref{p.Springer-fiber}.
    \item For a $(E_B, \theta_B) \in {\mathbf{Gr}}_{ \theta_B, \bf{O}_{B, R}}$, we have
    \begin{align*}
        {\mathbf{Spal}}_{\theta_B, P_B} \cong
        \{ 
        ( \sL, W_\sL, I(P_B)_{W_\sL}) \mid (\sL, W_{\sL}) \in {\mathbf{Gr}}_{\theta_B,\bf{O}_{B, R}}
        \},
    \end{align*}
    i.e.,
    \begin{align*}
        \nu_{P_B}^{-1}((E_B, \theta_B) )  & \cong \prod_{j \in I(P_B)} \OG(1, V_{B, j}).
    \end{align*}
    Here $V_{B, j}$'s are as in Proposition \ref{p.Springer-fiber}.
    \end{enumerate}
\end{lemma}
\begin{proof}
    Notice that if $\bf{O}_R$ is a Richardson orbit with polarization $P<G$, there is a surjective map $\bf{Spal}_{\theta, P} \rightarrow \bf{Gr}_{\theta, \bf{O}_R}$ with fibers isomorphic to the generic Spaltenstein fibers of $T^*(G/P) \rightarrow \overline{\bf{O}}_R$.
    
    For the type B, recall that the map $l_{BC}$ is governed by a set $W$ defined in Proposition~\ref{description of W}. In fact, we have:
     \[
    {\mathbf{Gr}}_{\theta_B, \bf{O}_{B}} \cong
        \{ 
        (\sL, W_{\sL}) \mid \sL \in {\mathbf{Gr}}_{\theta_C, \bf{O}_{C}}.
        \}
    \]
    Here we write $W_{\sL}$ to emphasize the dependence of $\sL$, defined in Definition \ref{def:iota_isotropic}.
    
    Moreover, in Richardson cases, by Lemma~\ref{lem:structure_of_Richardson_partition} $W$ is generated by $W_i$ defined in Corollary~\ref{why factor though}. Finally, by definition of $V_{C,j}$, the map $l_{BC}$ factors through $\bf{Spal}_{\theta_C, P_C}$.
\end{proof}

We now construct a cover 
$$
\widehat{\mathbf{Gr}}_{\theta_C, \bf{O}_{C, R},} \longrightarrow {\mathbf{Spal}}_{\theta_B, P_B}.
$$ 
which can be treated as the ``dual'' of ${\mathbf{Gr}}_{\theta_C, \bf{O}_{C, R}}$. We define the set $\widehat{I}$ as follows
\[
    \widehat{I} = 
    \left\{ j \equiv 0 \;\left|
    \begin{array}{c}
    d_{B, j} \in \bf{d}_{B} \;  \text{such that} \ d_{B, j} \equiv 1, \\
    \text{and} \; d_{B, j-1}>d_{B, j}  
    \end{array} \right.
    \right\}.
\]

\begin{lemma}
    With the above notations, we have $\widehat{I} \supset I(P_B)$ and $\# \widehat{I}(P_B) = c(\bf{d}_{B})$.  
\end{lemma}
\begin{proof}
    The first statement comes from the definition of $\widehat{I}$ and $I(P_B)$. Notice that $\widehat{I}$ labels the last part in $\bf{T}_i$ of type B2. Then, by Lemma \ref{lem:Lusztig's_quotient}, we conclude.
\end{proof}
Let
    \begin{align*}
        I(P_B)^\complement = \widehat{I} \setminus I(P_B).
    \end{align*}
Then, by Proposition \ref{Prop:Seesaw_prop} and the above lemma, we have $\# I(P_B)^\complement = \# I(P_C) $.

Then we define $\widehat{\mathbf{Gr}}_{\theta_C, \bf{O}_{C, R},}$ to be
\begin{align}\label{tilde_LH}
    \{ 
        ( \sL, W_\sL, I(P_B)_{W_\sL}, I(P_B)^\complement_{{W_\sL}}) \mid ( \sL, W_\sL, I(P_B)_{W_\sL}) \in {\mathbf{Gr}}_{\SO_{2n+1},P_B, \theta_B}
    \},
\end{align}
such that the fiber of $\widehat{\mathbf{Gr}}_{\theta_C, \bf{O}_{C, R},} \stackrel{\nu_{P_C}^{\vee}}{\longrightarrow} {\mathbf{Spal}}_{\theta_B, P_B}$ is $\small  \prod_{j \in I(P_B)^\complement} \OG(1, \widehat{V}_{B, j})$. Here 
\begin{align*}
    \widehat{V}_{B, j} = \Bbbk e \left( \frac{ d_{B,j}+1}{2}, j \right) \oplus \Bbbk e\left( \frac{d_{B,j+1}+1}{2}, j+1 \right), \quad \text{for} \quad j \in I(P_B)^\complement,
\end{align*}
similar to ${V}_{B, j}$ in Proposition \ref{p.Springer-fiber}.

\begin{proposition} \label{prop:mirror_position}
    Then we have the following diagram
    \begin{equation}\label{local diagram}
    \begin{tikzcd}
        \widehat{\mathbf{Gr}}_{\theta_C,\bf{O}_{C, R}} \ar[d, "\nu_{P_C}^\vee"']  & \\
        {\mathbf{Spal}}_{\theta_B, P_B}\ar[d] & {\mathbf{Spal}}_{\theta_C, P_C}\ar[d, "\nu_{P_C}"] \\
        {\mathbf{Gr}_{\theta_B,\bf{O}_{B, R}}} \ar[r, "l_{BC}"] \ar[ru] & {\mathbf{Gr}_{\theta_C,\bf{O}_{C, R}}}
    \end{tikzcd}.
\end{equation}
    where $\widehat{\mathbf{Gr}}_{\theta_C, \bf{O}_{C, R}}$ and ${\mathbf{Gr}}_{\theta_C, \bf{O}_{C, R}}$ only depend on the nilpotent orbit $\bf{O}_{C,R}$. For different choices of dual polarizations $(P_B, P_C)$, we always have 
    \[
        \deg \nu_{P_C}^{\vee} = \deg \nu_{P_C}.
    \]
\end{proposition}

\begin{proof}
    Notice that the set $\widehat{I}$ is not affected by the choice of parabolic subgroups. Thus, so is $\widehat{\mathbf{Gr}}_{\theta_C, \bf{O}_{C, R},}$. The degree of maps being equal results from the construction process.
\end{proof}

\subsection{Group action}\label{s.gp_action}

Let $G=\Sp_{2n}$, and $f \in L\gc_{\bf{O}_{C}}$ where $\bf{O}_C$ is a nilpotent orbit of type C and is not necessary special, and the partition of $\bf{O}_{C} $ is $\bf{d}_{C}$. Suppose $f(\lambda) = \prod_{i} f_i(\lambda)$ satisfies Assumption \ref{main assumption on char}. Then we have $\overline{\sO}_f$ and $\sK_{f}$ as in Definition~\ref{def:Obar}. The following proposition is basically an interpretation of Theorem \ref{Thm.type C decomposition}.

\begin{proposition}\label{p.Stab_C}
    Centralizer $Z_{\Sp_{2n}(\sK)}(\theta_{C})$ acts transitively on $\mathbf{Gr}_{\theta_C,\mathbf{O}_C}$ with the stabilizer $Z_{\Sp_{2n}(\sO)}(\theta_{C})$. 
\end{proposition}
\begin{proof}
    By the local parabolic Beauville--Narasimhan--Ramanan correspondence in Theorem \ref{Thm.type C decomposition}, we have the following bijection:
    \[
    \mathbf{Gr}_{\theta_C,\mathbf{O}_C} \cong     
       \{ (\sL\hookrightarrow \sK_f, \sigma^*\sL\cong\sL^{\vee}) \mid \sL \text{ is a  rank $1$ free module of} \;\bar{\sO}_f.
     \}
    \]
    where $\sigma$ is the involution $\lambda\mapsto -\lambda$. Now we know that $\{h\in\sK_f^{\times} \mid \sigma^*h\cdot h=1\}$ acts transitively on the right-hand side which actually isomorphic to $Z_{\Sp_{2n}(\sK)}(\theta)$.
    
    The statement for the stabilizer is easy to check using rank $1$ free modules.
\end{proof}
\begin{remark}
    The theorem also holds for $G=\GL_n,\SL_n$ due to the local parabolic BNR correspondence Proposition \ref{Thm.type A decomposition}.
\end{remark}

Notice that $Z_{\Sp_{2n}(\sO)}(\theta_C)$ is not connected, and the component group can be described as follows. Denote by $\bf{Even}$ the set of indexes of even parts in $\bf{d}_C$. Consider the elementary 2-group with a basis consisting of elements $b_j$, $j \in \N$. 

Let
\[
    \sA(\theta_C):=\{b = b_{j_1} + b_{j_2} + \cdots + b_{j_k} \mid j_i \in \bf{Even} \}.
\]

\begin{lemma} \label{lem:aff_comp_gp}
    Denote by  $ Z_{\Sp_{2n}(\sO)}(\theta_C)^{\circ}$ the connected component, then we have
    \begin{equation*}
         Z_{\Sp_{2n}(\sO)}(\theta_{C}) = Z_{\Sp_{2n}(\sO)}(\theta_{C})^\circ \times \sA(\theta_C).
    \end{equation*}
\end{lemma}

\begin{proof}
    We identify the positive loop group of centralizer $Z_{\Sp_{2n}(\sO)}(\theta_C)$ with the following simple group:
    \[
    \{h\in \bar{\sO}_{f}^{\times} \mid h\cdot \sigma^{*}h=1\}.
    \]
    If we write $h=\prod h_i$ where $h_i\in\bar{\sO}_{f,i}^{\times}$. By Proposition \ref{type C generic}, $\sigma$ preserves each $f_i(\lambda)$ when $e_i=\deg f_i$ is even. 

    Now we have an isomorphism $\bar{\sO}_{f,i}^{\times}\cong \Bbbk[\![\lambda]\!]^{\times}$. When $e_i=\deg f_i$ is even, the action of $\sigma$ on $\Bbbk[\![\lambda]\!]^{\times}$ is still $\lambda\mapsto -\lambda$. Hence it is easy to see that we have:
    \[
    \{h\in \bar{\sO}_{f,i}^{\times} \mid h\cdot \sigma^{*}h=1\}=\{\pm 1\}\times \{h\in \bar{\sO}_{f,i}^{\times} \mid h\cdot \sigma^{*}h=1, h(0)=1\}
    \]
    In particular, we have a natural surjection:
    \[
     Z_{\Sp_{2n}(\sO)}(\theta_C)\twoheadrightarrow \sA(\theta_C).
    \]
    Clearly, it splits.
\end{proof}

If $\mathbf{O}_B$ and $\mathbf{O}_C$ are special and under Springer dual, then we have a map $l_{BC}: \mathbf{Gr}_{\theta_B,\mathbf{O}_B } \rightarrow \mathbf{Gr}_{\theta_C,\mathbf{O}_C}$, as shown in Theorem \ref{Thm.type B decomposition}. By Lemma~\ref{structure of special partition}, we have
\[
    \bf{d}_{B} 
     = [d_{B,1}, d_{B, 2}, \ldots] 
     = [\bf{T}_{1},\bf{T}_{2}, \ldots, \bf{T}_{l-1}, \bf{T}_{l}, \ldots, \bf{T}_{k}, \bf{T}_{k+1}].
\]
Suppose $\bf{T}_{l_1}, \ldots, \bf{T}_{l_m}$ are of type B2. Let 
\[
   \vec{\bf{b}}_j = \{ b_i \mid d_{B, i} \in \bf{T}_{l_j}, \; j = 1, \ldots, m \}.
\]
Notice the Springer dual partition is given by ${}^S\bf{T}_B$ as in Lemma~\ref{structure of special partition}. Then we define a subgroup of $\sA(\theta_C)$ as follows
$$
    \sA(W):=\left\{b \in \sA(\theta_C) \; \left| \; 
    \begin{aligned}
        &\text{vector components consist only those in}\; \vec{\bf{b}}_j \\
        &\text{for}\; j=1,\ldots,m.\\
        &\text{Moreover, vector components in}\; \vec{\bf{b}}_j \\ 
        &\text{appear or vanish simultaneously in}\; b.
    \end{aligned} \right. \right\}.
$$
\begin{proposition} \label{p.Stab_B}
    Centralizer $Z_{\Sp_{2n}(\sK)}(\theta_{C})$ acts transitively on $\mathbf{Gr}_{\theta_B,\mathbf{O}_B}$ with the stabilizer $Z_{\Sp_{2n}(\sO)}(\theta_{C})^\circ \times \sA(W)$.
\end{proposition}
\begin{proof}
    By Proposition \ref{part 2.1}, $K_{B, i}$ is a submodule of $K_{C,i}$ for $i \neq 0$. Then it induces an action of $\sA(\theta_C)$ on $K_{B,i}$, here we require on $K_{B,0}$, the action is trivial. Recall that
    $$
    0\longrightarrow \oplus_{i= 0}^kK_{B,i}\longrightarrow \oplus_{i= 0}^kK_{B,i}^{\vee} \longrightarrow Q\longrightarrow 0.
    $$
    Then it follows from Corollary \ref{why factor though}, \ref{cor:E_W}.
\end{proof}

Now we come back to the Richardson case. Let $\bf{O}_{C, R}$, $\bf{O}_{B, R}$ be Springer dual Richardson orbits. Let $P_C$ be a polarization of $\bf{O}_{C, R}$. Denote by 
\begin{align}\label{A(P_C)}
    \sA(P_C):=\left\{b \in \sA(\theta_C) \; \left| \; 
    \begin{aligned}
        &b_j, b_{j+1} \; \text{appear or vanish}\\ 
        &\text{simultaneously in}  \; b, \\ 
        &\text{for} \; j \in I(P_C)
    \end{aligned} \right. \right\}
\end{align}
a subgroup of $\sA(\theta_C)$. And let $P_B$ be the polarization (Langlands dual to $P_C$) of $\bf{O}_{B, R}$. Denote by 
\begin{align}\label{A(P_B)}
    \sA(P_B):=\left\{b \in \sA(W) \; \left| \; 
    \begin{aligned}
        &b_j, b_{j+1} \; \text{appear or vanish }\\
        &\text{simultaneously in}\; b, \\ 
        &\text{for} \; j \in I(P_B)
    \end{aligned} \right. \right\}
\end{align}
a subgroup of $\sA(W)$. 

\begin{lemma}\label{lem:A(P_B)}
    $\sA(P_B)$ is generated by $b_j+b_{j+1}$ for $i \in I(P_C)$. 
\end{lemma}
\begin{proof}
    By the definition of $I(P_B)$, it can be shown that for $i \in I(P_B)$, $b_i$ is the last part of certain $\bf{T}_{l_j}$ of type B2, and $b_{i+1}$ is the beginning part in the adjacent type B2 $\bf{T}_{l_j+1}$. Moreover, the beginning part in the only type B3 will be $b_{i+1}$ for some $i \in I(P_B)$ if $I(P_B) \neq \emptyset$. Then it follows from the definition of $\sA(W)$ and \cite[Proposition 3.5]{FRW24}.
\end{proof}

Then we have

\begin{theorem}\label{Thm.gp_action}
Centralizer $Z_{\Sp_{2n}(\sK)}(\theta_{C})$ acts transitively on each space in the following diagram
    \begin{equation*}
         \begin{tikzcd}
        \widehat{\mathbf{Gr}}_{\theta_C,\bf{O}_{C, R}} \ar[d, "\nu_{P_C}^\vee"']  & \\
        {\mathbf{Spal}}_{\theta_B, P_B}\ar[d] & {\mathbf{Spal}}_{\theta_C, P_C}\ar[d, "\nu_{P_C}"] \\
        {\mathbf{Gr}_{\theta_B,\bf{O}_{B, R}}} \ar[r, "l_{BC}"] \ar[ru] & {\mathbf{Gr}_{\theta_C,\bf{O}_{C, R}}}
    \end{tikzcd},
    \end{equation*}
and we have
\begin{itemize}
    \item[1.] The fiber of $\nu_{P_C}$ is a $\sA(\theta_C)/\sA(P_C)$ torsor,
    \item[2.] The fiber of $l_{BC}$ is a $\sA(\theta_C)/\sA(W)$ torsor,
    \item[3.] The fiber of $\nu_{P_B}$ is a $\sA(W)/\sA(P_B)$ torsor,
    \item[4.] The fiber of $\nu_{P_C}^\vee$ is a $\sA(P_B)$ torsor. 
\end{itemize}
Moreover, 
\[
    (\sA(W)/\sA(P_B)) \times (\sA(\theta_C)/\sA(P_C)) \cong \bar{A}(\bf{O}_{C,R}).
\]
Here $\bar{A}(\bf{O}_{C,R})$ is the Lusztig's canonical quotient.
\end{theorem}
\begin{proof}
The second follows from Proposition \ref{p.Stab_B}. By Diagram  \eqref{local diagram} and Proposition \ref{p.Stab_C}, we know that $Z_{\Sp_{2n}(\sO)(\theta_{C})}$ acts transitively on ${\mathbf{Spal}}_{\theta_C, P_C}$, ${\mathbf{Gr}}_{\theta_B,\bf{O}_{B, R}}$, ${\mathbf{Spal}}_{\theta_B, P_B}$, and $\widehat{\mathbf{Gr}}_{\theta_C,\bf{O}_{C, R}}$.  We only need to describe the stabilizers for each action. By Theorem \ref{Thm.type C decomposition}, $E_C = \oplus_{i=1}^k K_{C,i}$, where $E_{C,i}=\Ker f_{C,i}(\theta_C)$, and $\sA(\theta_C)$ acts on $K_{C,i}$. Then the stabilizer for the action on ${\mathbf{Spal}}_{\theta_C, P_C}$ (resp. ${\mathbf{Spal}}_{\theta_B, P_B}$) is $Z_{\Sp_{2n}(\sO)}(\theta_{C})^\circ \times \sA(P_C)$ (resp. $Z_{\Sp_{2n}(\sO)}(\theta_{C})^\circ \times \sA(P_B)$) by Lemma \ref{lem:fiber_of_LHiggs_P_B/C} . Hence we have the first and the third arguments. The fourth one and the product formula are due to the construction of $\widehat{\mathbf{Gr}}_{\theta_C,\bf{O}_{C, R}}$ in \eqref{tilde_LH}, Proposition \ref{Prop:Seesaw_prop} and \cite[Proposition 3.5]{FRW24}.
\end{proof}

\section{Moduli space associated with the nilpotent orbit closure} \label{S.JM_moduli_space}

\subsection{Moduli space of parabolic Higgs bundles}\label{subs:moduli of parabolic}
In this section, we shift our focus to the moduli space of parabolic Higgs bundles.

Let $\Sigma$ be a smooth projective algebraic curve of genus $g$, let $\{t_1, t_2, \ldots, t_l\}$ be a finite set of different points of $\Sigma$, and let $D = t_1 + t_2 + \ldots + t_l$ be the corresponding effective divisor. We always require that $2g-2+l>0$.

Let $G$ be a reductive group, and $P_i < G$, $i=1, \ldots, l$ be the parabolic subgroups. 
\begin{definition}
    A parabolic $G$-Higgs bundle over $\Sigma$ associated with $P_i$'s is a tuple $(\sE, \theta, \{\sE_{P_i}\}_{i=1}^l)$:
\begin{itemize}
    \item $\sE$ is a principal $G$-bundle over $\Sigma$,
    \item $\theta$ is a section of $\Ad(\sE) \otimes \omega_{\Sigma}(D):= \sE \times_{G, \Ad} \g \otimes \omega_{\Sigma}(D)$,
    \item $\sE_{P_i}$ is a $P_i$-reduction of $\sE$ at $t_i$,
\end{itemize}
such that $\Res_{t_i} \theta \in \sE_{P_i} \times_{P_i, \Ad} \gn(P_i)$, where $\gn(P_i)$ is the nilradical of the Lie algebra of $P_i$. 
\end{definition}

For simplicity, we consider the case where the divisor $D$ consists of a single point $x \in \Sigma$. This allows us to describe the parabolic Higgs bundle in terms of a single filtration.

If $G=\Sp_{2n}$, let $P_C < \Sp_{2n}$ be a parabolic subgroup with Levi type $(p_1, p_2, \ldots, p_k; q)$. The $\Sp_{2n}$-Higgs bundle associated with $P_C$ is equivalent to a quadruple $(E_C, g_C,\Fil^\bullet_{P_C}, \theta_C)$:

\begin{itemize}
    \item $E_C$ is a rank-$2n$ bundle over $\Sigma$,
    \item $\theta_C$ is a Higgs field that $\theta_C: E_C \rightarrow E_C \otimes \omega_\Sigma (x)$, 
    \item $g_C$ is a non-degenerate skew-symmetric pairing satisfying $g_C(\theta_C v,w)+g_C(v,\theta_C w)=0$,
    \item $\Fil^\bullet_{P_C}$ is a filtration of $E_C|_x$, defined as: $E_C|_x=F^0 \supset F^1 \supset \ldots \supset F^k \supset (F^k)^\perp \supset \ldots \supset (F^1)^\perp \supset (F^0)^\perp = 0 $, with $\dim F^{i-1}/F^i = p_i$ for $i=1, \ldots, k$, and $\dim (F^k)^\perp/F^k = q$,
    \item $\Res_x \theta_C$ strongly preserves the filtration, i.e., $\Res_x \theta_C (F^i) \subset F^{i+1}$.
\end{itemize}

The $\operatorname{par-\mu}$ stability is defined as follows. Let
\begin{align} \label{pardeg}
    \operatorname{par-deg} E_C := \deg E_C + \sum_{i=1}^{k} \alpha_{i} \cdot p_i + \alpha_{0} \cdot q, 
\end{align}
be the \emph{parabolic degree} for $\vec{\alpha}_C=(\alpha_{0}, \alpha_{1}, \ldots, \alpha_{k})$, $0\le \alpha_{0}<\alpha_{1}<\ldots< \alpha_{k}\le 1$. Hence, it is also called \emph{$\vec{\alpha}$-degree}. And the \emph{parabolic slop} or \emph{$\vec{\alpha}$-slope} is defined as follows
\begin{align*}
    \operatorname{par-\mu} E_C = \frac{\operatorname{par-deg} E_C}{\rk E_C}.
\end{align*}

We say that a parabolic $\Sp_{2n}$-Higgs bundle $(E_C, g_C,\Fil^\bullet_{P_C}, \theta_C)$ is \emph{stable} (resp. semi-stable) if, for any proper $\theta_C$-invariant isotropic subbundle $E'_C \subset E_C$, the following inequality holds
\[
    \operatorname{par-\mu} E'_C < \operatorname{par-\mu} E_C \quad (resp. \ \leq),
\]
the parabolic structure on $E'_C$ is inherited from $\Fil^\bullet_{P_C}$.

For simplicity, denote by $\bf{Higgs}_{P_C}$ the moduli space of stable Higgs bundles. Similarly, we have $\bf{Higgs}_{P_B}$, but since $\pi_1(\SO_{2n+1}) = \Z_2$, it has two connected components, which are denoted by $\bf{Higgs}_{P_B}^{+}$ and $\bf{Higgs}_{P_B}^{-}$. Please refer to Proposition \ref{p.h_P_B} to see how to distinguish these two components.

\subsection{Construction via Jacobson--Morozov resolution}\label{Sec:via_JM}
In this subsection, we construct the moduli space associated with any nilpotent orbit closure. Since such closures can be highly singular, we adopt the Jacobson–Morozov resolution for our construction:

Consider a nilpotent element $X \in \g $ ($\g = \mathrm{Lie}(G)$, where $G$ is any complex semisimple Lie group). There exists a standard triple $\{ X, H, Y \} \simeq \mathfrak{sl}_2$. The action of $\ad_H$ on $\g$ induces the decomposition
\begin{align*}
    \g = \oplus_{i \in \mathbb{Z}} \g_i.
\end{align*}
Define the parabolic subalgebra $\mathfrak{p}_{JM} = \oplus_{i \geq 0} \g_i$ and $\mathfrak{n}_2 = \oplus_{i \geq 2} \g_i$. Let $P_{JM} < G$ be the parabolic subgroup with Lie algebra $\mathfrak{p}_{JM}$. The map
\begin{align*}
    G \times_{P_{JM}} \mathfrak{n}_2 \longrightarrow \overline{\bf{O}}_X
\end{align*}
is known as the \emph{Jacobson--Morozov resolution}. 

The space $\{(\sE_P, \Res_x \theta )\}$ is isomorphic to $G \times_P \frak{n} = T^*(G/P)$. Taking $P=P_{JM}$, we have the inclusion
\[
    G\times_{P_{JM}} \frak{n_2} \hookrightarrow G \times_{P_{JM}} \frak{n}.
\]

\begin{definition}
    A $\overline{\bf{O}}$-Higgs bundle over $\Sigma$ is a triple $(\sE, \theta, \sE_{P_{JM}})$ as parabolic $G$-Higgs bundle associated with $P_{JM}$, but we require $\Res_{x}\theta \in \sE_{P_{JM}}\times_{P_{JM}, \Ad} \frak{n}_2$.
\end{definition}

Requiring the residue of  $\theta$ to lie in $\mathfrak{n}_2$ does not affect the stability of $(E_C,g_C, \theta_C,\Fil^{\bullet}_{P_{C,JM}})$. Therefore, we have the following result:

\begin{proposition}
    The moduli space of stable $\overline{\bf{O}}_C$-Higgs bundles ($\overline{\bf{O}}_B$-Higgs bundles) exists. It is denoted by  $\bf{Higgs}_{\overline{\bf{O}}_C} $ (resp.  $\bf{Higgs}_{\overline{\bf{O}}_B} $) and forms a closed subvariety of $\bf{Higgs}_{P_{C, JM}}$ (resp. $\bf{Higgs}_{P_{B, JM}}$). As before, $\bf{Higgs}_{\overline{\bf{O}}_B}$ has two connected components, $\bf{Higgs}_{\overline{\bf{O}}_B}^{+}$ and $\bf{Higgs}_{\overline{\bf{O}}_B}^{-}$, which can be distinguished using Proposition \ref{L_BC}.
\end{proposition}

Now, we calculate the dimensions of relevant moduli spaces. Here, we omit the notation indicating type B or type C for simplicity since the methods are the same. Following \cite[2.9.2]{BBAMY}\footnote{In \cite{BBAMY}, Moy-Prasad filtrations are considered which is more general than parabolic cases here. }, for a given $(E, \Fil^{\bullet}_{P_{JM}},\theta)\in \bf{Higgs}_{\overline{\bf{O}}}$, its infinitesimal deformation in $\bf{Higgs}_{P_{JM}}$ (resp. in $\bf{Higgs}_{\overline{\bf{O}}}$) is controlled by a complex $\sK_{P_{JM}}^{-1}\xrightarrow{\ad\theta}\sK_{P_{JM}}^{0}\otimes\omega(x)$ (resp. $\sK_{\overline{\bf{O}}}^{-1}\xrightarrow{\ad\theta}\sK_{\overline{\bf{O}}}^{0}\otimes\omega(x)$), which fits into the following diagrams:
\[
\small\begin{tikzcd}       
\sK_{P_{JM}}^{-1}\ar[d,"\ad(\theta)"]\ar[r, hook]&E(\mathfrak{g})\ar[d,"\ad(\theta)"]\ar[r, two heads]&(\mathfrak{g}/\mathfrak{p}_{JM})_x\\
\sK_{P_{JM}}^{0}\otimes\omega(x)\ar[r, hook]& E(\mathfrak{g})\otimes\omega(x)\ar[r, two heads] &(\mathfrak{g}/\mathfrak{n})_x
\end{tikzcd},
   \small \begin{tikzcd}       \sK_{\overline{\bf{O}}}^{-1}\ar[d,"\ad(\theta)"]\ar[r,"\cong"]&\sK_{P_{JM}}^{-1}\ar[d,"\ad(\theta)"]\\  \sK_{\overline{\bf{O}}}^{0}\ar[r,hook]&\sK_{P_{JM}}^{0}\ar[r]&(\mathfrak{n}/\mathfrak{n}_2)_x  
    \end{tikzcd}.
    \]
where $E(\mathfrak{g})$ is the adjoint bundle, and see \cite[\S 2.3]{Wang23} for more details. Suppose now that $(E, \Fil^{\bullet}_{P_{JM}},\theta)$ is a smooth point in $\bf{Higgs}_{P_{JM}}$, i.e., $(E, \Fil^{\bullet}_{P_{JM}},\theta)$ is stable, and the obstruction $\mathbb{H}^{1}(\sK_{P_{JM}}^{-1}\xrightarrow{\ad\theta}\sK_{P_{JM}}^{0})=0$.  Then, by the exact sequence between the tangent complexes, we know that
    \[
    \mathbb{H}^{1}(\sK_{\overline{\bf{O}}}^{-1}\xrightarrow{\ad\theta}\sK_{\overline{\bf{O}}}^{0})=0.
    \]
Hence $(E, \Fil^{\bullet}_{P_{JM}},\theta)$ when is also a smooth point of $\bf{Higgs}_{\overline{\bf{O}}}$. Now we can calculate the dimension of $\bf{Higgs}_{\overline{\bf{O}}}$. Notice that
    \[
    \dim \mathbb{H}^{1}(\sK_{P_{JM}}^0\xrightarrow{\ad\theta}\sK_{P_{JM}}^{1})=\dim \bf{Higgs}_{P_{JM}}
    \]
By the smoothness, we have
    \[
    \dim\bf{Higgs}_{P_{JM}}-\dim \bf{Higgs}_{\overline{\bf{O}}}=
    \dim \mathfrak{n}/\mathfrak{n}_2.
    \]
Since we have:
    \[
    \dim \bf{Higgs}_{P_{JM}}=(2g-2)\dim G+\dim T^{*}(G/P_{JM})
    \]
then 
    \begin{equation}\label{eq:dim of moduli}
    \dim \bf{Higgs}_{\overline{\bf{O}}}=(2g-2)\dim G+\dim\overline{\bf{O}},
    \end{equation}
because $\dim\overline{\bf{O}}=\dim G/P_{JM}+\dim\mathfrak{n}_2$.

\subsection{New geometric interpretation of Springer duality}
The Hitchin fibration provides a powerful tool to study the geometry of moduli spaces of Higgs bundles. By assigning to a Higgs field the coefficients of its characteristic polynomial, we obtain the Hitchin map for the moduli space of $\overline{\bf{O}}_C$-Higgs bundles:
\begin{align*}
    h_{\overline{\bf{O}}_C}: \bf{Higgs}_{ \overline{\bf{O}}_C} \longrightarrow {{\bf{H}_{\overline{\bf{O}}_C} }}
\end{align*}

\begin{lemma}
Let $\bf{d}_C =[d_{C,1}, \ldots, d_{C,r}]$ denote the partition corresponding to the nilpotent orbit $\bf{O}_C$. Define the sequence $\eta_C = \{\eta_{C, 2i}\}_{i=1}^n$ as follows
\begin{align} \label{delta_sing}
    \eta_{C,2i}=\min \left\{ j \mid \sum_{l=1}^j d_{C, l}\geq 2i\right\}.
\end{align}
Then the Hitchin base $\bf{H}_{\overline{\bf{O}}_C}$ is explicitly given by
\begin{align*}
    {{\bf{H}_{\overline{\bf{O}}_C} }} = \bigoplus_{i=1}^n \operatorname{H}^0(\Sigma, \omega_{\Sigma}^{2i} \otimes (2i - \eta_{C, 2i})x ).
\end{align*}
\end{lemma}
\begin{proof}
${{\bf{H}_{\overline{\bf{O}}_C} }}$ lies in $\bigoplus_{i=1}^n \operatorname{H}^0(\Sigma, \omega_{\Sigma}^{2i} \otimes (2i - \eta_{C, 2i})x )$ is guaranteed by Proposition~\ref{order of coe}. Since the Hitchin map is proper, by Proposition~\ref{Prop:fiber_JM_moduli}, we conclude.
\end{proof}

Similarly, we have moduli space of stable $\overline{\bf{O}}_B$-Higgs bundles, and   
\begin{align*}
    h_{\overline{\bf{O}}_B}: \bf{Higgs}_{\overline{\bf{O}}_B} \longrightarrow  \bf{H}_{\overline{\bf{O}}_B}  \subset \bigoplus_{i=1}^n \operatorname{H}^0(\Sigma, \omega_{\Sigma}^{2i} \otimes (2i - \eta_{B, 2i})x ),
\end{align*}
where $\eta_B =\{ \eta_{B, 2i} \}_{i=1}^n$ is defined similarly as \eqref{delta_sing}. However, the above may not be equal. 

\begin{lemma}\label{lem:B=C}
    If $\bf{O}_B$ is special, and let $\bf{O}_C={}^S \bf{O}_B$ be its Springer dual orbit. Then, $\eta_B=\eta_C$ and
    \[
        \bf{H}_{\overline{\bf{O}}_B}  = \bigoplus_{i=1}^n \operatorname{H}^0(\Sigma, \omega_{\Sigma}^{2i} \otimes (2i - \eta_{B, 2i})x ).
    \]
\end{lemma}
\begin{proof}
    If $\bf{O}_B = {}^S\bf{O}_C$, then, by Proposition~\ref{p.Springerdual}, we have $\bf{d}_C = {}^S\bf{d}_B = (\bf{d}_B^-)_C$. Then, by a little computation, we have $\eta_C=\eta_B$. Finally, by Proposition~\ref{L_BC}, we have $\bf{H}_{\overline{\bf{O}}_B}=\bf{H}_{\overline{\bf{O}}_C}$.
\end{proof}

Moreover, if $\bf{O}_B$ and $\bf{O}_C$ are special, we have the following theorem, which shows the importance of special orbits in the SYZ mirror symmetry. It gives a new geometric interpretation of Springer duality.

\begin{theorem}\label{Thm:why special}
    The following are equivalent:
    \begin{enumerate}
        \item The nilpotent orbits $\bf{O}_B$ and $\bf{O}_C$ are special and Springer dual.
        \item The Hitchin bases $ \bf{H}_{\overline{\bf{O}}_{B}}$ and $ \bf{H}_{\overline{\bf{O}}_{C}}$ are canonically isomorphic.
    \end{enumerate}
\end{theorem}

\begin{corollary}\label{half dimension}
    For a special nilpotent orbit $\bf{O}_{B/C}$ of type B/C, we have $$\dim \bf{H}_{\overline{\bf{O}}_{B/C}}=\dfrac{1}{2}\dim \bf{Higgs}_{ \overline{\bf{O}}_{B/C}}.$$ 
\end{corollary}

\begin{proof}
    By the general formula for the dimension of $\bf{Higgs}{\overline{\bf{O}}_{B/C}}$ (see \eqref{eq:dim of moduli}), we know
    \[
        \frac{1}{2}\dim \bf{Higgs}_{\overline{\bf{O}}_{B/C}} = (2n^2+n)(g-1)+\frac{1}{2}\dim \bf{O}_{B/C}. 
    \]
    On the other hand, by Riemann--Roch, the dimension of the Hitchin base is 
    \[
        \dim \bf{H}_{\overline{\bf{O}}_{B/C}} = (2n^2+n)g-n^2 - \sum_{i=1}^n \eta_{B/C,2i}.
    \]
    It is known that Springer duality preserves dimension and that $\eta_B = \eta_C$. it suffices to verify the equality for type C. Let $\bf{d}_C = [d_1, d_2, \ldots]$ denote the partition of $\bf{O}_C$, and let $\bf{d}_C^t = [s_1, s_2, \ldots]$ be its transpose. Using standard results (e.g., \cite[Corollary 6.1.4]{CM93}), we compute:
    \[
        \dim \bf{O}_C = 2n^2+n - \frac{1}{2}\sum_{i} s^2_i - \frac{1}{2} \sum_{i \;\text{odd}} r_i,
    \]
    where $r_i = \#\{j \mid d_j = i\}$.
    
    Finally, from \cite[Equation (3.3)]{SWW22}, we deduce
    \[
        \sum_{i=1}^n \eta_{C, 2i} = \frac{1}{4} \sum_j s_j(s_j + 1) + \frac{1}{4} \sum_{i \text{ odd}} r_i.
    \]
    Substituting this into the expressions for $\dim \bf{H}_{\overline{\bf{O}}_C}$ and $\dim \bf{Higgs}_{\overline{\bf{O}}_C}$ completes the proof.
\end{proof}

\begin{remark}
    The above dimension equation (Corollary~\ref{half dimension}) holds for all nilpotent orbits, not necessarily special. For type C, the proof follows directly using analogous arguments. However, for type B, the situation is more subtle, as the presence of Type B1* in the partitions introduces quadratic relations that affect the geometry of the Hitchin base. These non-special cases will be addressed in a separate paper.
\end{remark}

Before proving the above theorem, we need the following useful lemmas.
\begin{lemma}\label{lem:O_B special}
    The Hitchin base $\bf{H}_{\overline{\bf{O}}_{B}}$ is an affine space if and only if $\bf{O}_B$ is special.
\end{lemma}
\begin{proof}
    By Lemma \ref{structure of special partition}, $\bf{O}_B$ is special if and only if its corresponding partition does not contain any parts of type B1*. Equivalently, $\bf{O}_B$ is special if and only if there are no even parts in the sequence $\alpha_B$, where the Kazhdan--Lusztig label of $\bf{O}_B$ is denoted by $\KL_B(\bf{O}_B) = (\alpha_B, \beta_B)$. 
    
    By a similar argument as in Proposition 5.3 of \cite{BK18}, we see that if $\bf{O}_B$ is non-special, then there would be at least a homogeneous quadratic equation in the definition of Hitchin base. 
    
    Conversely, if $\bf{O}_B$ is special, Proposition \ref{Prop:fiber_JM_moduli} and Proposition \ref{L_BC} ensure that the Hitchin base equals 
    \[
        \bf{H}_{\overline{\bf{O}}_B} = \bigoplus_{i=1}^n \operatorname{H}^0(\Sigma, \omega_{\Sigma}^{2i} \otimes (2i - \eta_{B, 2i})x ),
    \] 
    which is an affine space. 
\end{proof}

\begin{lemma} \label{lem:O_C special}
Let $\bf{d}_{C}=[d_1, d_2, \ldots, d_k ]$ be a partition of type C. If there exists a partition $\bf{d}_{B}$ of type B such that $\eta_B=\eta_C$, then
\begin{enumerate}
	\item  $\bf{d}_{C}$ is special. 
	\item  Let $\bf{O}_B$ and $\bf{O}_C$ denote the nilpotent orbits corresponding to the partitions $\bf{d}_B$ and $\bf{d}_C$, respectively. Then, $\dim \bf{O}_{B}\leq \dim \bf{O}_{C}$, the equality holds only when $\bf{O}_{B}$ is Springer dual to $\bf{O}_{C}$.
\end{enumerate}	
\end{lemma}
\begin{proof}[Proof]

Here we figure out how to obtain a partition $\bf{d}_{B}=[d_1', d_2', \ldots, d_k']$ from $ \bf{d}_C$ such that $\eta_B=\eta_C$. Let $ \bf{d}_C = [\bf{d}_0, \bf{d}_1, \ldots, \bf{d}_{2l}] $, where $\bf{d}_{2i+1}$ consists of odd parts and $\bf{d}_{2i}$ consists of even parts. We allow $\bf{d}_0$ to be empty, and set $\bf{d}_{2l}$ to be $[ \star\ldots ,\star, 0, 0, 0, \ldots]$.

We start form $\bf{d}_0=[d_1, \ldots, d_{k_0}]$, all $d_i$ in $\bf{d}_0$ are even. For each even part $d_i$, according to the rules of $\eta_C$, we need to pick $d_i/2$ many $i$'s in the Young tableau of $\bf{d}_{0}$, especially we need to pick up the last one.
 By $\eta_B=\eta_C$, for $d_1$, we have the following two choices
\begin{enumerate}
	\item Set $d_1'=d_1$, then $d_1'$ is even. Since $\bf{d}_{B}$ is of type B, we have $d_2'=d_1'$. By $\eta_B=\eta_C$, we can set $d_1'=d_1$ only when $d_1=d_2$. And if we do so, $d_2'=d_2=d_1=d_1'$, especially, $d_2$ is even.
	\item Set $d_1'=d_1+1$. Now, if $d_2$ is odd, we have $d_3=d_2$ since $\bf{d}_{C}$ is of type C, which makes $\eta_B=\eta_C$ impossible. So $d_2$ should be even, $k_0\geq 2$. If $\bf{d}_C\neq \bf{d}_0$, then there exists an integer $r\leq k_0/2$, such that $d_{2r}>d_{2r+1}$, $d_{2s}=d_{2s+1}$, $0<s<r$. By $\eta_B=\eta_C$, $d_{2s}$ should be even and $d_{2r}'=d_{2r}-1$, $d_i'=d_i$, for $1<i<2s$. If $\bf{d}_C=\bf{d}_0$, such $r$ may not exist. This only happens when there exists $u\in \mathbb{N}$ such that $d_{2u}=0$, $d_{2s}=d_{2s+1}$, $0<s<u$. In this case, we set $d_i'=d_i$ for $i > 1$.
\end{enumerate}
In both choices above, even number $d_i'$s are determined, and the last $i$ is picked up in the $i$-th row of the Young tableau of $\bf{d}_B$. Suppose $d_{2r+1}$ is the first undetermined one. If $d_{2r+1}$ is even, then the discussion of $d_{2r+1}$ is exactly the same as the above $d_1$. It is easy to see the combinations of the above two elementary changes are all the ways to handle $\bf{d}_0$. Furthermore, if $\bf{d}_C\neq \bf{d}_0$, then $k_0$ is even for $\bf{d}_0=[d_1, \ldots, d_{k_0}]$. 

Now we consider $\bf{d}_1=[d_{k_0+1}, \ldots, d_{k_1}]$, all $d_i$s here are odd. Recall that $\eta_B=\eta_C$, and the last $k_0$ is picked up in the Young tableau of $\bf{d}_B$ and $\bf{d}_C$. We must set $d_i'=d_i$ for $k_0+1\leq i \leq k_1$. Furthermore, the last $k_1$ is picked up in the Young tableau of $\bf{d}_B$ and $\bf{d}_C$.

The discussion of $\bf{d}_{2i+1}$ is exactly the same as $\bf{d}_1$. The discussion of $\bf{d}_{2i}$ is exactly the same as $\bf{d}_0$. Especially, the number of parts in $\bf{d}_{2i}$ is even unless $i=l$. So $\bf{d}_C$ is special by the alternative description of special partition below Definition \ref{def:special nilpotent}. The first statement is proven.

The above argument shows two ways to change $\bf{d}_{2i}$. If we always choose the way (2), then we obtain canonical partition $\bf{d}_{B}^{(0)}$. It is straightforward to check that $\bf{d}_B^{(0)} = ( \bf{d}_C^+ )_B$, which is known to be special and Springer dual to $\bf{d}_C$, see \cite[Proposition 2.1]{FRW24}.

The set $\mathcal{P}(N)$ is partially ordered as follows: $\bf{d}=[d_1, \ldots, d_N]\geq \bf{f}=[f_1, \ldots, f_N]$ if and only if $\sum_{j=1}^k d_j\geq \sum_{j=1}^k f_j$, for all $1\leq k \leq N$. This induces a partial order on $\mathcal{P}_{\epsilon}(N)$, which coincides with the partial ordering on nilpotent orbits given by the inclusion of closures (cf. \cite{CM93})

If we choose way (1) somewhere, $\bf{d}_B$ we obtain will satisfies $\bf{d}_B < \bf{d}_B^{(0)}$, which confirms that $\dim \bf{O}_{B}< \dim \bf{O}_{d_B^{(0)}}$ by \cite[Theorem 6.2.5]{CM93}. Then, by the fact that Springer dual is dimension-preserving \cite{Spa06}, we have $\dim \bf{O}_{B}< \dim \bf{O}_{d_B^{(0)}}=\dim \bf{O}_{C}$. Then we conclude.
\end{proof} 

\begin{proof}[Proof of Theorem~\ref{Thm:why special}]\
\begin{itemize}
	\item[1.] $(1) \Rightarrow (2)$: If $\bf{O}_B$ and $\bf{O}_C$ are special and Springer dual, then by Lemma~\ref{lem:B=C}, $\eta_B=\eta_C$. Consequently, $\bf{H}_{\overline{\bf{O}}_B}=\bf{H}_{\overline{\bf{O}}_C}$, yielding a canonical isomorphism of Hitchin bases.
	\item[2.] $(2) \Rightarrow (1)$: If $\bf{H}_{\overline{\bf{O}}_B} \cong \bf{H}_{\overline{\bf{O}}_C}$, then by Corollary~\ref{half dimension}, $\bf{O}_B$ and $\bf{O}_C$ must have the same dimension. Lemma~\ref{lem:O_B special} and \ref{lem:O_C special} ensure that $\bf{O}_B$ and $\bf{O}_C$ are both special, and their dimensions coincide only if they are Springer dual.
\end{itemize}
\end{proof}

\section{Parabolic BNR correspondence and generic Hitchin fibers}\label{Sec:Par_BNR}

Let $\Tot \omega_{\Sigma}(D)$ be the total space of the line bundle $\omega_{\Sigma}(D)$ and denote by $\pi:\Tot \omega_{\Sigma}(D)\rightarrow \Sigma$. Let $\lambda \in \operatorname{H}^0(\Tot \omega_{\Sigma}(D), \pi^*\omega_{\Sigma}(D))$ be the tautological section. 

\begin{definition}
    Given $a= (a_{2}, a_{4}, \ldots, a_{2n}) \in \bf{H}_{\overline{\bf{O}}_C}$, we have a section
\begin{align*}
    \lambda^{2n} + a_{2} \lambda^{2n-2} + \ldots + a_{2n} \in \operatorname{H}^0(\Tot \omega_{\Sigma}(D), \pi^*\omega^{2n}_{\Sigma}(2nD)).
\end{align*}
We define the \emph{spectral curve} $\Sigma_a$ as its zero divisor  in $\Tot \omega_{\Sigma}(D)$. And we denote by $\pi_a:\Sigma_a\rightarrow\Sigma$ the projection and by $\sigma$ the involution on $\Sigma_a$ induced by $\lambda\rightarrow -\lambda$.
\end{definition}

Notice that, if $(E_C,\theta_C)$ is a $\Sp_{2n}$-Higgs bundle lying in the fiber over $a$, then its characteristic polynomial is:
\[
\det(\lambda-\pi^*\theta_C)=\lambda^{2n} + a_{2} \lambda^{2n-2} + \ldots + a_{2n}.
\]
Similarly, if $(E_B,\theta_B)$ is a $\SO_{2n+1}$-Higgs bundle lying in the fiber over $a$, then its characteristic polynomial is:
\[
\det (\lambda - \pi^*\theta_B) = \lambda(\lambda^{2n} + a_{2} \lambda^{2n-2} + \ldots + a_{2n}).
\]
\begin{remark}
    Spectral curves defined in this way are finite flat covers of the base curve $\Sigma$. 
    Moreover,
    \begin{itemize}
        \item Without parabolic structure, if $a$ is generic, then $\Sigma_{a}$ is smooth, which is important in the classical BNR correspondence.
        \item When (non-regular) nilpotent (parabolic) structures exist, $\Sigma_{a}$ is always singular at marked points. 
    \end{itemize}
    To prove the parabolic BNR correspondence, we need to consider the normalization of spectral curves. In type A cases, it is shown in \cite{SS95, She18, SWW22}.
\end{remark}

\subsection{Prym varieties for ramified double covers}\label{Ap:Prym_var}
In this subsection, we shall analyze the properties of Prym varieties and their duals. We can deal with this in general, i.e., without talking about spectral curves. To simplify the notation, we denote our curve by $S$ with an involution $\sigma:S \rightarrow S$. We denote the quotient map by $\pi:S\rightarrow S/\sigma$.  We use $s_0, s_1, \cdots, s_{2N-1}$ to denote the fixed points of $\sigma$, which are also ramification points of $\pi$, and we denote $z_i=\pi(s_i)$.

The Prym variety, defined as
\[
\Prym(S,S/\sigma):=\{\sL\in\Jac(S) \mid \sigma^*\sL\cong\sL^{\vee}\}.
\]
If no ambiguity is caused, we write $\Prym$. The dual of $\Prym$ (under the polarization inherited from $\Jac(S)$), denoted by $\Prym^\vee$, can be understood through the following commutative diagram of exact sequences:
\begin{equation}\label{eq:dual exact sequence}
    \begin{tikzcd}
    1\ar[r]&\Jac(S/\sigma)\ar[r,"\pi^*"]&\Jac(S)\ar[r]\ar[d,"="]&\Prym^{\vee}\ar[r]&1\\
    1\ar[r]&\Prym\ar[r]&\Jac(S)\ar[r,"\Nm"]&\Jac(S/\sigma)\ar[r]&1
\end{tikzcd}.
\end{equation}

It is well known that we have a factorization
\begin{equation}\label{eq:mult 2 map}
\begin{tikzcd}
   \Prym\ar[r]\ar[d, "{[2]}"']&\Prym^{\vee}\ar[dl]\\
   \Prym
\end{tikzcd},
\end{equation}
the map is induced by the natural polarization on $\Jac(S)$, and consider the map $[2]$ defined by $[2](\mathcal{L}) = \mathcal{L}^2$. Moreover, the kernel of $\Prym \rightarrow \Prym^{\vee}$ is $\Prym \cap \pi^*\Jac(S/\sigma)$. 

 We choose $s_0$ as the base point of the Abel-Jacobian map: 
 \[
     S\rightarrow \Jac(S),\quad x\mapsto \mathcal{O}_S(x-s_0).
 \]
 Let $\mathcal{P}$ denote the Poincar\'e line bundle on $S\times \Jac(S)$. Restricting $\mathcal{P}$ to $S\times \Prym$, which we still denote by $\mathcal{P}$, retains the symmetry induced by the involution $\sigma$, satisfying $\sigma^*\mathcal{P}\cong \mathcal{P}^\vee$. 
 
 For $\sigma$-fixed point $s_i$, $1\leq i \leq 2N-1$, let $\mathcal{P}_{s_i}$ denote the restriction of $\sP$ to $\{s_i\}\times \Prym\cong\Prym$, in particular, $\sP_{s_0}$ is trivial. Moreover, $\sigma^*\sP_{s_i}\cong \sP_{s_i}$, equivalently $\sP_{s_i} \cong \sP_{s_i}^{\vee}$, and hence $\sP_{s_i} \in \Prym^\vee[2]$.

 We also use $\sP_{s_i-s_j}$ to denote $\sP_{s_i}\otimes \sP_{s_j}^\vee$. Using the Abel-Jacobian map we chose before, we see that $\sP_{s_i}$ is the image of $\sO_S(s_i-s_0)\in \Jac(S)$ under the map $\Jac(S)\rightarrow \Prym^\vee$. 

 Furthermore, since $\sP_{s_i} \cong \sP_{s_i}^\vee$, $\sO\oplus \sP_{s_i}$ is a sheaf of algebra on $\Prym$. In another point of view, the isomorphism $\sP_{s_i} \cong \sP_{s_i}^\vee$ also induces a nondegenerate symmetric pairing on $\sO\oplus \sP_{s_i}$, where we choose the pairing on $\sO$ by a nonzero scalar. Then we have:

 \begin{lemma}\label{two kinds of double cover and factor through}
     We use $\operatorname{OG}(1, \mathcal{O}\oplus \mathcal{P}_{s_i})$ to denote the relative isotropic Grassmannian bundle as we regard $\mathcal{O}\oplus \mathcal{P}_{s_i}$ as an orthogonal bundle. Then we have natural isomorphism $\operatorname{OG}(1, \mathcal{O}\oplus \mathcal{P}_{s_i})\cong \Spec(\mathcal{O}\oplus \mathcal{P}_{s_i})$, which is a degree $2$ \'etale cover of $\Prym$. This cover is trivial if and only if $\sP_{s_i} \cong \sO$. Moreover, for any morphism $f:X \rightarrow \Prym$, $f$ can factor through $\Spec(\mathcal{O}\oplus \mathcal{P}_{s_i})$ if and only if $f^*\mathcal{P}_{s_i}$ is trivial on $X$.
 \end{lemma} 

 Hence, we make the following definition.
\begin{definition}\label{def:double cover}
    For each element $e\in\Prym^{\vee}[2]$, there is a corresponding double cover of $\Prym$, denoted by $\Prym_{e} = \Spec(\mathcal{O} \oplus \mathcal{P}_{e})$, where $\sP_e$ is the line bundle corresponding to $e$. The natural involution on $\Prym_e$ is denoted by $\iota_e$.
 \end{definition}

 \begin{lemma}\label{lem:depends only on subspace}
     Let $\Prym^\vee[2]$ be viewed as a finite-dimensional $\mathbb{F}_2$-vector space. For any vector subspace $V \subset \Prym^\vee[2]$, let $\{e_i\}$ be a basis of $V$. Then the fiber product 
     \[
     \prod\Prym_{e_i}=\Prym_{e_1}\times_{\Prym}\Prym_{e_2}\times\ldots
     \]
     is independent of the choice of the basis.
 \end{lemma}
 \begin{proof}
     Let $\pi_{V}:\prod\Prym_{e_i}\rightarrow \Prym$ be the projection. The pushforward of the structure sheaf of the fiber product satisfies
     \[
     (\pi_{V})_*\sO\cong \bigotimes_{e_i \in V}(\sO\oplus\sP_{e_i}).
     \]
    Since the tensor product of line bundles depends only on the subspace $V$ and not the specific choice of basis, the result follows.
 \end{proof}
\begin{Notation}\label{not:fiber product over subspace}
    As before, we view $\Prym^\vee[2]$ as a $\mathbb{F}_2$-vector space.
    \begin{enumerate}
        \item For a subspace $V\subset \Prym^\vee[2]$, we denote by $\Prym_V$ the fiber product in the above lemma.
        \item For simplicity, if the choice of the basis is $\{\sP_{s_i}\}_{i\in I}$ for a subset $I\in\{1,2,\ldots,2N-1\}$, we may also denote the fiber product by $\Prym_I$.
    \end{enumerate}
\end{Notation}

In the following, we want to study the connectedness of fiber products of $\Prym_{s_i}$'s over $\Prym$.

\begin{lemma}\label{lem:criterion of connectedness}
    Let $\Prym_I$ be the fiber product of $\Prym_{s_i}$ for $s_i \in I \subseteq \{s_1, \cdots, s_{2N-1} \}$ over $\Prym$. Then, $\Prym_I$ is connected if and only if $\{\sP_{s_i}\}_{s_i\in I}$ is linearly independent in $\Prym^\vee[2]$. 
\end{lemma}
\begin{proof}
    From the definition of $\Prym_{I}$, the pushforward of the structure sheaf satisfies
    \[
    \pi_{I*}\sO_{\Prym_I}\cong \otimes_{s_i\in I}(\sO\oplus\sP_{s_i}).
    \]
    Hence, $\Prym_I$ is connected if and only if
    \[
    \dim \mathrm{H}^{0}(\Prym,\otimes_{s_i \in I}(\sO\oplus\sP_{s_i}))=1.
    \]
    In particular, this holds if and only if
    \[
    \otimes_{s_j\in J\subset I} \sP_{s_j}\ne\sO
    \]
    for any non-empty subset $J\subset I$. This amounts to saying that $\{\sP_{s_i}\}_{s_i\in I}$ are linearly independent in $\Prym^\vee[2]$.
\end{proof}

Consider the 2-torsion point of \eqref{eq:dual exact sequence}:
\begin{equation}\label{eq:2-torsion dual exact sequence}
    \begin{tikzcd}
    1\ar[r]&\Jac(S/\sigma)[2]\ar[r,"\pi^*"]\ar[d, "\iota"]&\Jac(S)[2]\ar[r]\ar[d,"="]&\Prym^{\vee}[2]\ar[r]\ar[d,"\beta"]&1\\
    1\ar[r]&\Prym[2]\ar[r]&\Jac(S)[2]\ar[r,"\Nm"]&\Jac(S/\sigma)[2]\ar[r]&1
\end{tikzcd},
\end{equation}
where $\iota$ is the natural inclusion of $\Jac(S/\sigma)[2]\hookrightarrow \Prym[2]$. Since $\Nm\circ\pi^*=[2]$ on $\Jac(S/\sigma)$, there is an induced map $\beta:\Prym^{\vee}[2]\rightarrow\Jac(S/\sigma)[2]$.
\begin{lemma}
    The kernel of the map $\Prym^{\vee}\rightarrow\Prym$ is canonically isomorphic  to $\ker\beta$.
\end{lemma}
\begin{proof}
    From \eqref{eq:dual exact sequence}, we know that:
    \[
    1\rightarrow\Jac(S/\sigma)[2]\rightarrow\Prym\rightarrow\Prym^{\vee}\rightarrow 1.
    \]
    From \eqref{eq:2-torsion dual exact sequence}, we know that:
    \[
    1\rightarrow\Jac(S/\sigma)[2]\rightarrow\Prym[2]\rightarrow\ker\beta\rightarrow 1.
    \]
    Since $\Prym\xrightarrow{[2]}\Prym$ factors through $\Prym^{\vee}$, we conclude.
\end{proof}

\begin{lemma}
    Each $\sP_{s_i}$ lies in $\ker\beta$, and  $\{\sP_{s_i}\}$ are linearly dependent in $\ker\beta$.
\end{lemma}
\begin{proof}
    Consider again the diagram \eqref{eq:2-torsion dual exact sequence}, we may choose a square root $\sO_{S/\sigma}(z_i-z_0)^{-1/2}$ of $\sO_{S/\sigma}(z_i-z_0)^{-1}$. Then $\sP_{s_i}$ is the image of $\sO_S(s_i-s_0)\otimes \pi^*\sO_{S/\sigma}(z_i-z_0)^{-1/2}$ and $\sO_S(s_i-s_0)\otimes \pi^*\sO_{S/\sigma}(z_i-z_0)^{-1/2}\in \Jac(S)[2]$. Notice that
    \[
    \Nm\big(\sO_S(s_i-s_0)\otimes \pi^*\sO_{S/\sigma}(z_i-z_0)^{-1/2}\big)\cong\sO_{S/\sigma}.
    \]
    Hence $\sP_{i}$ lies in $\ker\beta$.

Let $g(S/\sigma)$ be the genus of $S/\sigma$. Then
\[
\dim\Prym=g(S/\sigma)+N-1.
\]
Hence, as a $2$-torsion abelian group:
\[
\#\ker\beta=2^{2N-2}.
\]
As a result,  $\{\sP_{s_i}\}$ are linearly dependent in $\ker\beta$.
\end{proof}
We will see later in Proposition \ref{prop:non-trivial relation} that they sum to 0. Furthermore, this is the unique relation.

To check when two abelian varieties between $\Prym$ and $\Prym^{\vee}$, for example $\Prym_I$ and $\Prym_J$, are dual, we need to check how the polarization restricts to them. For example, it turns out to be related to the intersection number of homology if we work over the complex field. And in the general setting, we need to work with Weil pairing. Fortunately, we only need to work out Weil pairing for $2$-torsion points coming from the Abel-Jacobian map.

For a pair of dual abelian varieties $(A,A^{\vee})$, suppose $\ell$ is invertible in $\Bbbk$, then there is a perfect pairing
\[
e: A[\ell]\times A^{\vee}[\ell]\rightarrow \mu_\ell.
\]
It can be defined as follows, given an $\sL\in A^{\vee}[\ell]$, we know that $[\ell]^{*}\sL$ is trivial on $A$.\footnote{Here $[\ell]$ is the map $[\ell]: A \rightarrow A, \ \mathcal{L} \mapsto \mathcal{L}^{\ell}$. However, ``$[\ell]$'' in $A[\ell]$ means the $\ell$ torsion points in $A$.} Hence there is a character $\chi_{\sL}:A[\ell]\rightarrow\mathbb{G}_m$ such that
\[
\sL\cong A\times_{A[\ell],\chi_{\mathcal{L}}}\mathbb{A}^{1}.
\]
Then, the Weil pairing is defined as
$$e(\cdot,\mathcal{L}) = \chi_{\mathcal{L}}(\cdot).$$ 

\begin{remark} \label{rmk:trivial pairing value}
For any $K\subset A[\ell]$, if $e(K,\sL)=1$, then $\sL$ is trivial when pulled back to $A/K$ and
\[
\sL\cong A/K \times_{A[\ell]/K,\chi_{\mathcal{L}}}\mathbb{A}^{1}.
\]
\end{remark}

Back to our case when $A=\bf{Prym}$, recall \eqref{eq:2-torsion dual exact sequence}, then we have

\begin{proposition}
    The Weil pairing induces a perfect pairing
    \[
    e:\ker\beta\times\ker\beta\rightarrow\{\pm 1\}.
    \]
\end{proposition}
\begin{proof}
    Considering the following diagram
    \[
    \begin{tikzcd}      1\ar[r]&\Jac(S/\sigma)\ar[r,"\pi^*"]&\Jac(S)\ar[r]\ar[d,"="]&\Prym^{\vee}\ar[r]&1\\
    1\ar[r]&\Prym\ar[r]&\Jac(S)\ar[r,"\Nm"]&\Jac(S/\sigma)\ar[r]&1 
    \end{tikzcd}.
    \]
    And for their 2-torsion points, we have
    \[
    \begin{tikzcd}
    1\ar[r]&\Jac(S/\sigma)[2]\ar[r,"\pi^*"]\ar[d,"\iota"]&\Jac(S)[2]\ar[r]\ar[d,"="]&\Prym^{\vee}[2]\ar[r]\ar[d,"\beta"]&1\\
    1\ar[r]&\Prym[2]\ar[r]&\Jac(S)[2]\ar[r,"\Nm"]&\Jac(S/\sigma)[2]\ar[r]&1
    \end{tikzcd}.
    \]
    By the functorial property of Weil pairing, we know that
    \begin{equation}\label{eq:perfect pairing on involution point}
      e(\iota(\Jac(S/\sigma)[2]),\ker\beta)=1.  
    \end{equation}
   Then by the isomorphism 
   \begin{equation}\label{eq:ker beta}
   \ker\beta\cong \Prym[2]/\iota(\Jac(S/\sigma)[2]),
   \end{equation}
   the Weil pairing provides a perfect pairing on $\ker\beta$.
\end{proof}
Consider the following sequence of abelian varieties
 \[
    \begin{tikzcd}
        & \Prym^{\vee}\ar[ld]\ar[rd]&\\
        A_1\ar[rd]&&A_2\ar[ld]\\
        & \Prym
    \end{tikzcd}.
    \]
We want to know when $A_1, A_2$ are dual abelian varieties under the polarization induced from that on $\Prym$. We denote the kernels of $\Prym^{\vee}\rightarrow A_i$ by $\Delta_i\subset \ker\beta$.
\begin{lemma}\label{lem:criterion of duality}
    $A_1$ and $A_2$ are dual if and only if the annihilator of $\Delta_1$ is $\Delta_2$ with respect to the perfect pairing $e:\ker\beta\times\ker\beta\rightarrow\{\pm 1\}$.
\end{lemma}
\begin{proof}
    We denote the dual of $A_1$ by $A_1^{\vee}$ and have the following commutative diagram
    \[
    \begin{tikzcd}
        & \Prym^{\vee}\ar[ld]\ar[rd]\ar[rrd]&\\
        A_1\ar[rd]&&A_1^{\vee}\ar[ld]&A_2\ar[lld]\\
        & \Prym
    \end{tikzcd}.
    \]
    We now show that the kernel $\Prym^{\vee}\rightarrow A_1^{\vee}$ is the annihilator of $\Delta_1$ in $\ker\beta$. 
    
    Now consider the following diagram
    \[
    \begin{tikzcd}
        \Prym^{\vee}[2]\ar[d]&\times&\Prym[2]\ar[r]&\mu_{2}\\
        A_1[2]&\times&A_1^{\vee}[2]\ar[u]\ar[ur]&
    \end{tikzcd},
    \]
    by the functorial property of Weil pairing, we know that the image of $A_1^{\vee}[2]$ in $\Prym[2]$ is exactly the annihilator of $\Delta_1$. By \eqref{eq:ker beta}, we have the following
    \[
    \Im (A_1^\vee[2]\rightarrow\Prym[2]\rightarrow\ker\beta\subset\Prym^{\vee}[2]\rightarrow A_1^{\vee}[2]),
    \]
    and the composition is a zero map. The Weil pairing on $\ker\beta$ is induced from that on $\Prym[2]\times\Prym^{\vee}[2]$, hence the kernel $\Prym^{\vee}\rightarrow A_1^{\vee}$ is the annihilator of $\Delta_1$, i.e., coincides with $\Delta_2$.
\end{proof}
To simplify notation, we use $\tilde{\sP}_{s_i}$ to refer to the 2-torsion line bundle $\sO_S(s_i-s_0)\otimes \pi^*\sO_{S/\sigma}(z_i-z_0)^{-1/2}$. In particular, $\tilde{\sP}_{s_i}\in \Jac(S)[2]$ and then we can consider Weil pairing for them. The following lemma is easy to check via the functorial property of Weil pairings.
\begin{lemma}\label{lem:simple weil pairing}
    $e(\sP_{s_i},\sP_{s_j})=e(\tilde{\sP}_{s_i},\tilde{\sP}_{s_j})$.
\end{lemma}
Before we determine the Weil pairing and the connected components of the fiber product $\Prym_{I}$ where $I$ runs all over line bundles $\{\sP_{i}\}$, we need the following results on canonical line bundles of $S$, or equivalently, of ramification divisors.
\begin{lemma}\label{lem:canonical is a pull back, sum is trivial}
    The canonical line bundle $\omega_{S}$ lies in $\pi^*\Pic(S/\sigma)$. To be more precise, $\omega_{S}\cong \pi^*(\omega_{S/\sigma}\otimes(\det\pi_*\sO_{S})^{-1})$. In particular, we have $\sum_{i=1}^{2N-1}\sP_{s_i}=0$  in $\ker\beta$.
\end{lemma}
\begin{proof}
    We have $\pi_{*}\sO_{S}=\sO\oplus\det\pi_*\sO_{S}$ since $\pi$ is a double cover. In particular, we may view $S$ as a spectral curve over $S/\sigma$ contained in the total space of the line bundle $\det\pi_*\sO_{S}^{-1}$. Then by the usual way to determine the canonical line bundle of a spectral curve, we have $\omega_{S}\cong \pi^*(\omega_{S/\sigma}\otimes(\det\pi_*\sO_{S})^{-1})$.
    
    Now let $R = \sum_{i=0}^{2N-1} s_i$, then we have
    \[
    \omega_{S}=\pi^*\omega_{S/\sigma}\otimes\sO_{S}(R),
    \]
    then $\sO_{S}(R)\cong \pi^*(\det\pi_*\sO_{S})^{-1}$ and hence $\sum_{i=1}^{2N-1}\sP_{s_i}=0$  in $\ker\beta$.
\end{proof}

Now we can calculate the Weil pairing on $\ker\beta$, which will be used to obtain dual abelian varieties for various Hitchin systems.
\begin{proposition}\label{prop:weil pairing}
    If $i\ne k,\ell,j\ne k,\ell$, i.e., disjoint support, then $$e(\sO_{S}(s_i-s_j)\otimes\pi^*\sO_{S/\sigma}(z_i-z_j)^{-1/2},\sO_{S}(s_k-s_\ell)\otimes\pi^*\sO_{S/\sigma}(z_k-z_\ell)^{-1/2})=1.$$ In particular,  $e(\tilde{\sP}_{s_i},\tilde{\sP}_{s_j})=-1$ (hence, $e(\sP_{s_i},\sP_{s_j})=-1$ by Lemma \ref{lem:simple weil pairing}) if and only if $i\ne j$.
\end{proposition}

\begin{proof}    
    First we can find two rational functions $f_{ij}$ and $f_{k\ell}$ on $S/\sigma$ such that
    \begin{align*}
        \Div(f_{ij})=z_i-z_j+D_{ij}, \\
        \Div(f_{k\ell})=z_k-z_\ell+D_{k\ell}.
    \end{align*}
    and $ D_{ij},D_{k\ell}$ have disjoint supports from each other and also from the divisor $\sum _{i=0}^{2N-1}y_i$.

    Notice that $D_{ij}$ (resp. $D_{k\ell}$) is equivalent to $\sO_{S/\sigma}(z_i-z_j)^{-1}$ (resp. $\sO_{S/\sigma}(z_k-y_\ell)^{-1}$). We can choose a divisor, denoted by $D_{ij}^{1/2}$ (resp. $D_{k\ell}^{1/2}$), such that $\sO_{S/\sigma}(D_{ij}^{1/2})=(\sO_{S/\sigma}(z_i-z_j))^{-1/2}$ (resp. $\sO_{S/\sigma}(D_{k\ell}^{1/2})=(\sO_{S/\sigma}(z_k-z_\ell))^{-1/2}$). 

    By \cite[Theorem 1]{Howe96}, the Weil pairing for Jacobians, which is what we need, can be calculated as follows 
    \begin{align*}
        &e(\sO_{S}(s_i-s_j)\otimes\pi^*\sO_{S/\sigma}(z_i-z_j)^{-1/2},\sO_{S}(s_k-s_\ell)\otimes\pi^*\sO_{S/\sigma}(z_k-z_\ell)^{-1/2})\\
        =& \frac{\pi^*f_{ij}(s_k-s_{\ell}+\pi^*D_{k\ell}^{1/2})}{\pi^*f_{k\ell}(s_i-s_j+\pi^*D_{ij}^{1/2})}\\
        =&\frac{f_{ij}(z_k-z_{\ell}+D_{k\ell})}{f_{k\ell}(z_i-z_j+D_{ij})}=1.
    \end{align*}
    Since we assume the supports are disjoint, the last equality is due to the Weil reciprocity. Here, we use the following notation for rational functions: let $D = \sum_{i} a_i s_i$, and let $f$ be a rational function on $S$, then we define 
    \[
     f(D) = \prod_{s_i \in D} f(s_i)^{a_i}
    \]

    Notice that $\sO_{S}(s_i-s_j)\otimes\pi^*\sO_{S/\sigma}(z_i-z_j)^{-1/2}$ can be treated as $\tilde{\sP}_{s_i}\otimes\tilde{\sP}_{s_j}^{-1}$. Hence, the last statement is due to that  
    \[ 
    \sum \sP_{s_i}=0.
    \]
    See Lemma \ref{lem:canonical is a pull back, sum is trivial}.
\end{proof}
Now, combined with Lemma \ref{lem:criterion of connectedness}, we arrive at the following proposition. 
\begin{proposition}\label{prop:non-trivial relation}
    There  is a unique nontrivial linear relation between $\{\sP_{s_i}\}_{i=1}^{2N-1}$ in $\Prym^\vee$, i.e.,
    \[
    \sum_{i=1}^{2N-1}\sP_{s_i}=0.
    \]
    Hence the fiber product $\Prym_{I}$ over $\Prym$ has two connected components if and only if $I=\{1,2,\ldots, 2N-1\}$, i.e., runs all over the involution-fixed points. 
\end{proposition}
\begin{proof}
    By Lemma \ref{lem:canonical is a pull back, sum is trivial},
    \[
    \sum_{i=1}^{2N-1}\sP_{s_i}=0.
    \]
    in $\Prym^{\vee}[2]$. Now we show that this is unique.
    
    Suppose they have another relation:
    \[
    \sum_{i\in I} \sP_{s_i}=0.
    \]
    where $I\ne \{1,2,\ldots,2N-1\}$. Choose $j\in\{1,2,\ldots, 2N-1\}$. Then we have
    \[
    e(\sP_{s_j},\sum_{i\in I} \sP_{s_i})=0.
    \]
    If $j\notin I$, this implies that $\#I$ has to be even. But if $j\in I$, then $\#I$ has to be odd. Hence $I=\{1,2,\ldots,2N-1\}$. 
\end{proof}  

\begin{proposition}\label{prop:two connected components}
    The dual abelian variety $\Prym^{\vee}$ is isomorphic to a connected component of the fiber product $$\Prym_{\ker\beta}:=\Prym_{s_1}\times_{\Prym}\Prym_{s_2}\times_{\Prym}\cdots\times_{\Prym}\Prym_{s_{2N-1}}.$$
\end{proposition}
\begin{proof}
    By Remark \ref{rmk:trivial pairing value}, we know that the pullback of $\sP_{s_i}$ to $\Prym^{\vee}$ is trivial. In particular, this implies that the natural map
    \[
    \Prym^{\vee}\rightarrow\Prym
    \]
    factors through the fiber product. Combined with Proposition \ref{prop:non-trivial relation}, $\Prym^{\vee}$ is isomorphic to a connected component.
\end{proof}

Recall from Definition \ref{def:double cover}, for each $s_i$, we denote the involution by $\iota_{s_i}$ on $\Prym_{s_i}$, which then acts on the fiber product $\Prym_{\ker\beta}$.
\begin{lemma}\label{lem:permute conn comp}
    For each $s_i$, the natural involution $\iota_{s_i}$ permutes the two connected components of the fiber product.
\end{lemma}
\begin{proof}
    It is easy to see that the quotient of 
    \[   \Prym_{s_1}\times_{\Prym}\Prym_{s_2}\times_{\Prym}\cdots\times_{\Prym}\Prym_{s_{2N-1}}
    \]
    by $\iota_{s_i}$ is connected. Hence $\iota_{s_i}$ permutes the 2 connected components of $\Prym_{\ker\beta}$.
\end{proof}
Later we will show that various generic fibers can be constructed via successive double covers of generic fibers of type C. Hence we can also discuss the duality between them.

\subsection{Generic fiber of type C}

We first define the ``Kazhdan--Lusztig'' open locus in the Hitchin base, over which singularities of corresponding spectral curves can be controlled.
\begin{definition}\label{Def:H^KL}
    Let $\bf{O}_C$ be a nilpotent orbit (not necessarily special) of type C. Define the open subset $\bf{H}^{\text{KL}}\subset \bf{H}_{\overline{\bf{O}}_C}$ consists of all $a\in \bf{H}_{\overline{\bf{O}}_C}$ satisfying:
    \begin{itemize}
        \item[(a)] The spectral curve $\pi_a:\Sigma_a\rightarrow \Sigma$ is smooth outside of marked points. 
        
        \item[(b)] Over the marked points, the local equation of $\Sigma_a$ around each marked point is generic in $\mathbf{Char}_{\operatorname{\mathbf{O}_C}}$, i.e., Assumption \ref{main assumption on char} and the assumption in Proposition \ref{description of W} hold.
    \end{itemize}
\end{definition}

Let $\bar{\Sigma}_a$ denote the normalization of the spectral curve $\Sigma_a$, which inherits the natural involution $\sigma$. The quotient of $\bar{\Sigma}_a$ by this involution is denoted $\bar{\Sigma}_a / \sigma$, and the map $\bar{\pi}_a: \bar{\Sigma}_a \to \Sigma$ represents the natural projection.

\begin{lemma}\label{the inverse image of $x$}
    If the partition of $\bf{O}_C$ is given by $[d_{C, 1}, \cdots , d_{C,r}]$, then $\bar{\pi}_a^{-1}(x)=\{x_1, \cdots , x_r\}$. The ramification index of $x_i$ is $d_{C, i}$.
\end{lemma}
\begin{proof}
     By Proposition \ref{type C generic}, we have
     \[
     \widehat{\sO}_{\overline{\Sigma}_a,\bar{\pi}_a^{-1}(x)}\cong \prod \Bbbk[\![t]\!][\lambda]/f_{C,i}.
     \]
     Then we conclude.
\end{proof}

\begin{definition}\label{def:Prym var ramification divisor}
Denote the Prym variety 
$$
    \Prym_a:=\Prym(\bar{\Sigma}_a,\bar{\Sigma}_a/\sigma)=\{\mathcal{L}\in \Jac(\bar{\Sigma}_a) \mid \sigma^*\mathcal{L}\cong \mathcal{L}^{\vee}\}.
$$
Let $\sR_a$ denote the ramification divisor of $\bar{\pi}_a$, given by $\omega_{\bar{\Sigma}_a}\otimes \bar{\pi}_a^*\omega_{\Sigma}^{\vee}$. 
\end{definition}
We start from Hitchin fiber $h_{\overline{\bf{O}}_C}^{-1}(a)$. 

\begin{proposition} \label{Prop:fiber_JM_moduli}
    The fiber $h_{\overline{\bf{O}}_C}^{-1}(a)$ is canonically isomorphic to
    \[
    \{\sL\in\Pic^{\frac{\deg \sR_a}{2}}(\bar{\Sigma}_a) \mid \sigma^*\sL\cong\sL^{\vee}(\sR_a)\}.
    \]
    It is a torsor over the Prym variety $\Prym_a$. Hence, the Hitchin base $\bf{H}_{\overline{\bf{O}}_C}$ is the image of Hitchin map $h_{\overline{\bf{O}}_C}$.
\end{proposition}

\begin{proof}
    We use $\Sigma^\circ$ to denote $\Sigma\setminus D$ and then we cover the curve $\Sigma$ with $\Sigma^\circ$ and formal neighborhoods $\Spec \widehat{\sO}_{x}$ for each $x\in D$. In particular, $\pi_a^{-1}(\Sigma^{\circ})\subset \Sigma_a$ is smooth, and  $\bar{\Sigma}_a$ is the gluing of $\pi_a^{-1}(\Sigma^{\circ})$ and normalizations of $\Spec \widehat{\sO}_{x}$'s. By the genericity of $a$, the characteristic polynomial locally around each $x$ satisfies the Assumption \ref{main assumption on char}. Then, by Theorem \ref{Thm.type C decomposition} and Beauville--Narasimhan--Ramanan correspondence, we conclude that, for $(E_C, g_C, \Fil^{\bullet}_{P_{C,JM}}, \theta_C) \in h_{\overline{\bf{O}}_C}^{-1}(a)$, there is a line bundle $\mathcal{L}$ on $\bar{\Sigma}_a$ such that $(\bar{\pi}_{a})_* \mathcal{L}\cong E_C$. By the Grothendieck--Serre duality, we have
    \[
    (\bar{\pi}_{a*}(\sL))^{\vee}\cong \bar{\pi}_{a*}(\sL^{\vee} \otimes \sO(\sR_a)).
    \]
    Hence $\bar{\pi}_{a*}\sL$ is a symplectic bundle if and only if
    \[
    \sigma^{*}\sL\cong\sL^{\vee}\otimes\sO(\sR_a).
    \]
    Here $\sigma^*$ is used to make sure we have skew-symmetric pairing.

    Conversely, given a line bundle $\sL$ on $\bar{\Sigma}_a$, satisfying $\sigma^{*}\sL\cong \sL^{\vee}\otimes \sO(\sR_a)$, then we have a natural isomorphism
	\[
	(\bar{\pi}_{a})_*\sigma^{*}\sL\cong (\bar{\pi}_{a})_*(\sL^{\vee}\otimes \sO(\sR_a)).
	\]
    Applying the Grothendieck--Serre duality to $\bar{\pi}_a: \bar{\Sigma}_a\rightarrow \Sigma$, we have $((\bar{\pi}_{a})_*\sL)^{\vee}\cong (\bar{\pi}_{a})_*(\sL^{\vee}(\sR_a))$.
    Since $(\bar{\pi}_{a})_*\sigma^{*}\sL=(\bar{\pi}_{a})_*\sL$, we have
    \[
	(\bar{\pi}_{a})_*\sL\cong ((\bar{\pi}_{a})_*\sL)^{\vee},
    \]
    i.e., we have a nondegenerate bilinear form on $(\bar{\pi}_{a})_*\sL$. The existence of $\sigma$ shows that the pairing is skew-symmetric.
    
    We now show that there is a canonical filtration on each $(\bar{\pi}_{a})_*\sL$ via ramifications of $\bar{\Sigma}_a$ over $\Sigma$. Recall that by Proposition \ref{type C generic} we have the decomposition locally around each marked point
    \[
    f=\prod f_{i}, \; \deg f_i=e_i=d_i.
    \]
    Then on each branch $\sO_{f,i}$ of $\bar{\Sigma}_a$ over a neighborhood of $x$, we have the filtration:
    \begin{equation}\label{eq:ram-fil}
         \sL\supset \mathfrak{m}_{f,i}\sL\supset \mathfrak{m}^{2}_{f,i}\sL\supset\cdots\supset \mathfrak{m}^{d_i-1}_{f,i}\sL\supset \mathfrak{m}_{f,i}^{d_i}\sL,
    \end{equation}
    which corresponds to a $\mathfrak{sl}_2$ representation of dimension $d_i$. Then we obtain the Jacobson--Morozov filtration.    

    Hence we have a canonical identification
    \[
    h_{\overline{\bf{O}}_C}^{-1}(a)\cong \{\sL\in\Pic(\bar{\Sigma}_a) \mid \sigma^*\sL\cong\sL^{\vee}(\sR_a)\}.
    \]
    In particular, this shows that $h_{\overline{\bf{O}}_C}^{-1}(a)$ is a torsor of $\Prym_a$.
\end{proof}

\begin{definition}\label{def:ram-fil}
    Let
    \[
    \Prym(\sR_a):=\{\sL\in\Pic(\bar{\Sigma}_a) \mid \sigma^*\sL\cong\sL^{\vee}(\sR_a)\}.
    \]
    We denote by $\sF^{\bullet}_{ram, a}$ the filtration, induced by \eqref{eq:ram-fil}, on the push forward of line bundles from normalized spectral curve $\bar{\Sigma}_a$. 
\end{definition}

We now describe the relative version of the above discussions. Over the open subvariety $\textbf{H}^{\text{KL}}$, we have a family of  spectral curves (resp. normalized spectral curves) $\Xi$ (resp. $\bar{\Xi}$):
\[
\begin{tikzcd}
    \bar{\Xi}\ar[r]\ar[rd,"\bar{\pi}", swap]&
    \Xi\ar[d,"\pi"]\\
    &\Sigma\times \textbf{H}^{\text{KL}}
\end{tikzcd}.
\]
Over $\textbf{H}^{\text{KL}}$, we denote by $\sR$ the ramification divisor of the natural quotient map $\bar{\Xi}\rightarrow\bar{\Xi}/\sigma$. Then we have the relative $\Prym(\sR)$, a torsor of the relative Prym variety $\Prym$. Moreover, we have a relative Poincar\'e line bundle $\sP$\footnote{Notice that this need not be unique, and this accounts for SYZ mirror symmetry.} (over the relative Prym variety $\Prym$) which fits into the following diagram:
\[
\begin{tikzcd}
   \sP\ar[r]\ar[d]&\Prym(\sR_a)\times\bar{\Xi}\ar[d,"\bar{\pi}"]\\
   ((\id\times \bar{\pi})_{*}\sP, \sF^{\bullet}_{ram})\ar[r]&\Sigma\times \textbf{H}^{\text{KL}}\ar[d]\\
   &\textbf{H}^{\text{KL}}
\end{tikzcd}.
\]
Now we conclude:
\begin{theorem}\label{family version of filtration}
    The moduli space $\bf{Higgs}_{\overline{\bf{O}}_C}$, restricted to $\textbf{H}^{\text{KL}}$, is a torsor over the relative Prym variety $\Prym$. Moreover, the restriction of the universal family $(\sE_{\overline{\bf{O}}_C},\theta_C, \mathcal{F}^{\bullet})$ is isomorphic to $(\bar{\pi}_{*}\sP, \mathcal{F}_{ram}^{\bullet})$ where the filtration $\mathcal{F}_{ram}^{\bullet}$ is completely determined by the ramification of normalized spectral curves. 
\end{theorem}
Later, we will show that $\Prym(\sR_a)$ is a \emph{trivial torsor} over the relative Prym variety $\Prym$, as predicted by SYZ mirror symmetry.

\subsection{Generic fiber of type B} 

To study $h_{\overline{\bf{O}}_B}^{-1}(a)$, we need to require that the nilpotent orbits $\overline{\bf{O}}_B$ and $\overline{\bf{O}}_C$ are special and under Springer duality. By Theorem \ref{Thm:why special}, $\bf{H}_{\overline{\bf{O}}_B} = \bf{H}_{\overline{\bf{O}}_C}$. Take $a \in \bf{H}^{\text{KL}}$, it corresponds to the following characteristic polynomial  
\[
    \lambda(\lambda^{2n} + a_{2} \lambda^{2n-2} + \ldots + a_{2n-2}\lambda^2 + a_{2n}).
\]

\begin{definition}
    For later use, we define $\delta = \lfloor \dfrac{\operatorname{ord}_{x}a_{2n}}{2}\rfloor$. 
\end{definition}

\begin{proposition}\label{L_BC}
    Fix a choice of square root of $\omega_{\Sigma}$, we have a morphism 
    \[
    L_{BC}: h_{\overline{\bf{O}}_B}^{-1}(a)\longrightarrow h_{\overline{\bf{O}}_C}^{-1}(a),
    \]
    with degree
    \[
    2^{2n(2g-2)+\beta(\bf{d}_B)-c(\bf{d}_B)-1}.
    \] 
    Moreover, $h_{\overline{\bf{O}}_B}^{-1}(a)$ has two connected components: $h_{\overline{\bf{O}}_B}^{-1}(a)^+$ and $h_{\overline{\bf{O}}_B}^{-1}(a)^-$, and each one is a torsor of $\Prym_{\overline{\bf{O}}_B, a}$, which is a finite cover of $\Prym_a$ (defined in Proposition \ref{type B JM and different components}). Over $h_{\overline{\bf{O}}_B}^{-1}(a)^+$ we have a canonical point.
\end{proposition}

This proposition is a global version of Theorem \ref{Thm.type B decomposition}. We will divide the proof into three parts:

\begin{proof}[Proof of Proposition \ref{L_BC} (first part: existence of $L_{BC}$)]
    Let $(E_B, g_B, \Fil^{\bullet}_{P_{B, JM}}, \theta_B) \in h_{\overline{\bf{O}}_B}^{-1}(a)$, we firstly modify $E_B$ as follows:
    $$
    0\longrightarrow E_B^\prime \longrightarrow E_B \longrightarrow R \longrightarrow 0
    $$
    where the morphism $E_B\rightarrow R$ is defined as in Proposition \ref{coker dimension} and thus $\det E_B^\prime\cong \mathcal{O}(-\lceil\dfrac{\beta(\bf{d}_B)}{2}\rceil x)$. Then the restriction of $\theta_B$ on $E_B^\prime$(which we also denote as $\theta_B$) has partition $[{}^S\bf{d}_B, 1]$. Then we consider the exact sequence:
    $$
    0\longrightarrow \Ker \theta_B \longrightarrow E_B^\prime \longrightarrow \bar{E} \longrightarrow 0.
    $$
    By a similar argument as in Section 4 in \cite{Hit07}, one can show that $\Ker \theta_B\cong \omega_{\Sigma}^{-n}((\delta-n)x)$\footnote{We can use Pfaffian to obtain a generator of the kernel, and the order at $x$ is determined then.} and hence $\det \bar{E}\cong \omega_{\Sigma}^n((n-\delta-\lceil\dfrac{\beta(\bf{d}_B)}{2}\rceil)x)$. We write the induced morphism of $\theta_B$ on $\bar{E}$ as $\theta_C$. Notice that by our construction of $E_B^\prime$, this exact sequence is split around $x$, and hence the partition of $\theta_C$ at $x$ is ${}^S\bf{d}_B=\bf{d}_C$.

    Now we define a skew-symmetric bilinear form on $\bar{E}$ as: $\bar{g}_C(\bar{u}, \bar{v})=g_B(\theta_B u,v)$, for any $\bar{u}, \bar{v}\in \bar{E}$ and $u, v \in E_B^\prime$ are preimages of $\bar{u}, \bar{v}$. Notice that $\bar{g}_C$ takes values in $\omega_{\Sigma}(x)$.

    Finally, we define $E_C^\prime$ to be the kernel of the following:
    $$
    \bar{E}\twoheadrightarrow \bar{E}|_x\twoheadrightarrow \bigoplus_{i\geq 1} K_{B,i}(0)/\Im \theta_{B,i}(0)^{k_i}
    $$ 
    where $K_{B,i}(0)/\Im \theta_{B,i}(0)^{k_i}$ is given in Proposition \ref{part 1.2}. We see that the restriction of $\bar{g}_C$ on $E_C^\prime$ takes values in $\omega_{\Sigma}\subset \omega_{\Sigma}(x)$. Moreover, we have $\det E_C^\prime\cong \omega_{\Sigma}^n$ since $n-\delta-\lceil\dfrac{\beta(\bf{d}_B)}{2}\rceil-\sum_{e_{B, i} \text{even}}k_i-\sum_{e_{B, i} \text{odd}}k_i=0$. Thus, the restriction of $\bar{g}_C$ on $E_C^\prime$ is nondegenerate. By tensoring $E_C^\prime$ with the inverse of the square root of $\omega_{\Sigma}$ we choose and since the Jacobson--Morozov resolution one to one on $\bf{O}_C$, we get $(E_{C}, g_C, \Fil^{\bullet}_{P_{C,JM}}, \theta_C)\in h_{\overline{\bf{O}}_C}^{-1}(a)$.
\end{proof}

\begin{proof}[Proof of Proposition~\ref{L_BC} (second part: construction of $h_{\overline{\bf{O}}_B}^{-1}(a)$)]

    We need to reverse the process in the first part and figure out how many choices we need to make in each step. 
    
    Let $(E_{C}, \theta_C, \dots)\in h_{\overline{\bf{O}}_C}^{-1}(a)$, we tensor $E_C$ with the square root of $\omega_{\Sigma}$ we have chosen to get $E_C^\prime$, then we have a nondegenerate skew-symmetric bilinear form $g_C:E_C^\prime\otimes E_C^\prime\rightarrow \omega_{\Sigma}$, $\theta_C: E_C\rightarrow E_C\otimes \omega_{\Sigma}(x)$ and $\det E_C^\prime \cong \omega_{\Sigma}^n$.

    Now we consider $\bar{E}$ as the kernel of the following:
    $$
    E_C^\prime(x)\twoheadrightarrow (E_C^\prime(x))|_{x}\twoheadrightarrow \bigoplus_{i\geq 1} (t^{-1}K_{C,i})(0)/\Ker \theta_C(0)^{k_i},
    $$
    where $(t^{-1}K_{C,i})(0)/\Ker \theta_C(0)^{k_i}$ is given in Proposition \ref{part 2.1}. Thus $\det \bar{E}\cong \omega_{\Sigma}^n((2n-\sum_{i\geq 1}k_i)x)$ and the restriction of $g_C$ on $\bar{E}$ takes values in $\omega_{\Sigma}(2x)$.

    Now we have a morphism $\bar{E}\otimes \omega_{\Sigma}(x)^\vee\rightarrow \bar{E}^\vee(x)$ induced by the pairing on $\bar{E}$. This is an isomorphism on $\Sigma\setminus \{x\}$. By the construction of $\bar{E}$, we have an induced mophism $\bar{E}^\vee(x)\otimes \sO_{\Sigma}(-2x)=\bar{E}^\vee(-x)\rightarrow \bar{E}\otimes \omega_{\Sigma}(x)^\vee$ which is also an isomorphism on $\Sigma\setminus \{x\}$. Compose with $\theta_C: \bar{E}\rightarrow \bar{E}\otimes \omega_{\Sigma}(x)$ and we denote this morphism as $\bar{g}: \bar{E}^\vee(-x)\rightarrow \bar{E}$. Notice that $\bar{g}$ gives a generically nondegenerate symmetric pairing on $\bar{E}^\vee$ over $\Sigma\setminus \{x\}$.

    We mention that the global section $a_{2n}$ gives a natural morphism $\times a_{2n}:\omega_{\Sigma}^{-n}((\delta-n)x)\rightarrow \omega_{\Sigma}^{n}((n-\delta)x)$. Hence we have 
    $$
    0\rightarrow \bar{E}^\vee(-x)\oplus \omega_{\Sigma}^{-n}((\delta-n)x)\xrightarrow{\times a_{2n}} \bar{E}\oplus \omega_{\Sigma}^{n}((n-\delta)x) \rightarrow Q_x\oplus \bigoplus_{y\in Y} Q_y \rightarrow 0.
    $$
    Now, each $Q_y$ is a direct sum of the cokernel of $\bar{g}$ and cokernel of $\times a_{2n}$ at $y$. Hence, it is a two-dimensional skyscraper sheaf supporting at $y$ with a nondegenerate pairing and has two maximal isotropic subspaces. We now choose a maximal isotropic subspace $\mathfrak{I}_y$  at each $y\in Y$, then we have $2^{2n(2g-2)-1}$ choices.

    We define $E_B^\prime$ to be the inverse image of $\Coker (\times a_{2n})|_x\oplus \oplus_{y\in Y}\mathfrak{I}_y$ in $\bar{E}\oplus \omega_{\Sigma}^{n}((n-\delta)x)$, and then $\det E_B^\prime \cong \mathcal{O}_{\Sigma}((n+\delta-\sum_{i\geq 1}k_i)x)$. We claim that there is a natural symmetric pairing on $E_B^\prime$, which is nondegenerate over $\Sigma \setminus \{x\}$.

    By Lemma \ref{isotropic}, there is a natural symmetric pairing on $E_B^\prime$ over $\Sigma \setminus \{x\}$ which is nondegenerate, we need to show that this pairing extends over $\Sigma$. Locally, this pairing is given by the restriction of $g_C\circ \theta_C^{-1}\oplus (\times a_{2n})^{-1}$, which is well defined when restricts to $E_B^\prime$ over $\Sigma \setminus \{x\}$, so we only need to show that $g_C\circ \theta_C^{-1}\oplus (\times a_{2n})^{-1}$ is well defined near $x$. Near $x$, $E_B^\prime$ is given by the direct sum of $\bar{E}$ and $\omega_{\Sigma}^{-n}((\delta-n)x)$ thus by Proposition \ref{part 2.1} we see that $g_C\circ \theta_C^{-1}$ is well defined and actually factor through $\sO_{\Sigma}\subset \sO_{\Sigma}(x)$ and hence we proved our claim.

    Finally, we consider the exact sequence 
    \begin{equation}\label{Q}
        0\longrightarrow E_{B}^\prime \longrightarrow (E_{B}^\prime)^\vee \longrightarrow Q \longrightarrow 0.
    \end{equation}
    Here, $Q$ is the same as in Definition \ref{def:iota_isotropic}. Hence, we can choose any $\iota$-isotropic subspace $W$ of $Q$ so that the inverse image of $W$ in $(E_{B}^\prime)^\vee$ together with the induced Higgs field $\theta_B$ is a $\overline{\bf{O}}_B$-Higgs bundle. The choice of  $\iota$-isotropic subspaces will be $2^{\beta(\bf{d}_B)-c(\bf{d}_B)}$ by Proposition \ref{description of W}.
\end{proof}

Before giving the third part of the proof of Proposition \ref{L_BC}, we need some preparations. Recall that, in Definition \ref{def:Prym var ramification divisor}, $\sR_a$ denote the ramification divisor of $\bar{\Sigma}_a\rightarrow \bar{\Sigma}_a/\sigma$ which is a reduced divisor and consists of $\sigma$ fixed points. We fix a base point $y_0$ of $\sR_a$ which is not lying over the marked point $x$. As mentioned in Section~\ref{Ap:Prym_var}, for each points $y_i\in \sR_a$, we can define a double cover $\Prym_{y_i,a}$ of $\Prym_a$. See Definition~\ref{def:double cover} for more details.
\begin{definition}\label{def:PrymY}
    Let $Y=\sR_a\setminus  \{\bar{\pi}_a^{-1}(x) \cup \{y_0\} \}$. We denote $\Prym_{Y,a}$ to be the fiber product of $\Prym_{y_i,a}$ over $\Prym_a$ for all $y_i\in Y$.
\end{definition}

According to the second part of the proof of Proposition \ref{L_BC}, for each $\sL \in \Prym_a$, we have a torsion sheaf $Q_{\sL}$ supporting at $x$ as in the exact sequence \ref{Q}.

\begin{lemma}\label{construction of Prym_W}
    We have a finite cover $\Prym_{W,a}$ of $\Prym_a$ of degree $2^{\beta(\bf{d}_B)-c(\bf{d}_B)}$, parametrizing tuples $(\sL, W_{\sL})$ where $\sL\in \Prym_a$ and $W_{\sL}\subset Q_{\sL}$ is an $\iota$-isotropic subspace (see Definition~\ref{def:iota_isotropic}).
\end{lemma}

\begin{proof}
Recall, by Lemma~\ref{structure of special partition}, the partition is given by $[d_{C,1}, \cdots, d_{C, r}]=[^S\bf{T}_1, \cdots, ^S\bf{T}_s]$, which is exactly the ramification index of all points in $\bar{\Sigma}_a$ lying over the puncture point $x$. By corollary \ref{rough decomposition of E}, we may only consider a type (B2) partition. Then for a sequence $^S\bf{T}=[e_1, \cdots , e_{2m}]=[r_1^{m_1}, \cdots , r_k^{m_k}]$ with all parts being even, we use $x^T_i$ to denote the fixed point of $\sigma$ associated with $e_i$. Firstly we consider the projective bundles $\mathbb{P}_j=\mathbb{P}(\oplus_{m_{j-1}+1\leq i \leq m_j}\mathcal{P}_{x^T_i})$ for $1\leq j \leq k$; then take $\prod \mathbb{P}=\prod_{{\Prym_a}}^{1\leq j \leq k} \mathbb{P}_j$ to be the fiber product of all $\mathbb{P}_j$ over $\Prym_a$. Over $\prod \mathbb{P}$ we have the (pulling back of) relative tautological line bundles $\mathcal{O}_{j}(-1):=\mathcal{O}_{\mathbb{P}_j}(-1)$, and hence we have: 
$$
\prod \prod \mathbb{P}= \prod_{\prod \mathbb{P}}\mathbb{P}(\mathcal{O}_{\mathbb{P}_j}(-1)\oplus \mathcal{O}_{\mathbb{P}_{j+1}}(-1))
$$ 
which is the fiber product over $\prod \mathbb{P}$ with $1\leq j \leq k-1$. We denote the (pulling back of) tautological line bundle on $\mathbb{P}(\mathcal{O}_{\mathbb{P}_j}\oplus \mathcal{O}_{\mathbb{P}_{j+1}})$ as $\mathcal{O}_{j, j+1}(-1)$. Then we have a line bundle on $\prod \prod \mathbb{P}$ 
$$
\mathcal{O}(2):=\bigotimes_{j=1}^k\mathcal{O}_j(2)\otimes \bigotimes_{j=1}^{k-1}\mathcal{O}_{j, j+1}(2).
$$ 
We will define ($\sum_{j=1}^km_j-1$) many global sections of $\mathcal{O}(2)$ and then consider the zero locus defined by these sections.

Recall that we have scalars $\Phi$ on $\mathcal{P}_{x^T_i}$, which we defined above Definition \ref{F geq leq i}, then we have several pairings on $\oplus_{m_{j-1}+1\leq i \leq m_j}\mathcal{P}_{x^T_i}$: $$\bigoplus_{m_{j-1}+1\leq i \leq m_j}\mathcal{P}_{x^T_i}\stackrel{\Phi^{-l}}{\longrightarrow} \bigoplus_{m_{j-1}+1\leq i \leq m_j}\mathcal{P}_{x^T_i}\stackrel{\sigma^*}{\longrightarrow}\bigoplus_{m_{j-1}+1\leq i \leq m_j}\mathcal{P}_{x^T_i}^{\vee}.$$ Notice that $\mathcal{O}_j(-1)$ is a sub-line bundle of $\oplus_{m_{j-1}+1\leq i \leq m_j}\mathcal{P}_{x^T_i}$ and so these parings gives morphisms $I_{j, l}: \mathcal{O}_j(-1)\rightarrow \mathcal{O}_j(1)$ and hence global sections of $\mathcal{O}_j(2)$. Here we only take $1\leq l \leq m_j -1$. These would give $\sum_{j=1}^km_j-k$ many sections of $\mathcal{O}(2)$.

To get the remaining sections, we consider $I_{j, \lfloor \frac{m_{j}}{2} \rfloor}: \mathcal{O}_j(-1)\rightarrow \mathcal{O}_j(1)$ and $I_{j+1, 0}: \mathcal{O}_{j+1}(-1)\rightarrow \mathcal{O}_{j+1}(1)$, which would give morphisms $I_{j, \lfloor \frac{m_{j}}{2} \rfloor}\otimes I_{j+1, 0}:\mathcal{O}_{j, j+1}(-1)\rightarrow \mathcal{O}_{j, j+1}(1)$ for $1\leq j \leq k-1$ and hence $k-1$ sections of $\mathcal{O}(2)$.

Let $\Prym_{\bf{T}, a}$ be the zero locus of $\sum_{j=1}^km_j-1$ many sections we defined above. Since $a\in \bf{H}^{\text{KL}}$, it is a degree $2^{\sum_{j=1}^km_j-1}$ finite cover of $\Prym_a$, i.e., $f_\bf{T}:\Prym_{\bf{T}, a}\rightarrow \Prym_a$. See the following Proposition~\ref{fix1} for an alternative description.

In general, consider all the $\bf{T}_i$'s and we define $\Prym_{W,a}$ to be the fiber product of $\Prym_{\bf{T}_i, a}$'s over $\Prym_a$. By Proposition \ref{description of W}, $\Prym_{W,a}$ parametrizes tuples $(\sL, W_{\sL})$ where $\sL\in \Prym_a$ and $W_{\sL}\subset Q_{\sL}$ is an $\iota$-isotropic subspace.
\end{proof}

\begin{proposition}\label{fix1}
    $\Prym_{\bf{T}, a}\cong \prod_{\Prym_a}^{1\leq i \leq 2m-1}\Prym_{x_i^T-x_{2m}^T, a}$.
\end{proposition}

\begin{proof}
    We firstly show that $f_{\bf{T}}^*\sP _{x_i^T-x_{2m}^T}$ is trivial on $\Prym_{\bf{T}, a}$. Consider the variety $\prod \prod \mathbb{P}$ we defined in Lemma \ref{construction of Prym_W}, the pull back of $\sP_{x_i^T-x_{m_j}^T}$ and $\sP_{x_{m_j}^T-x_{m_{j}+1}^T}$ on $\prod \prod \mathbb{P}$ both have a canonical section for $m_{j-1}+1\leq i \leq m_j$, $1\leq j \leq k$ . The zero locus of the canonical section over $\sL\in \Prym_a$ is given by tuple of vectors(up to scalar): $(v_j\in \oplus_{m_{j-1}+1\leq i \leq m_j}\sP_{x_i^T}|_{\sL}, a_jv_j+b_{j+i}v_{j+1})$, $1\leq j \leq k$ such that one of the vector components of $v_j$'s or $a_j$'s or $b_j$'s is zero. By Corollary \ref{vector ponent is not zero}, we see that this zero locus does not intersect with $\Prym_{\bf{T}, a}$, thus $f_{\bf{T}}^*\sP _{x_i^T-x_{2m}^T}$ is trivial on $\Prym_{\bf{T}, a}$ for $1\leq i \leq 2m-1$.

    By Lemma \ref{two kinds of double cover and factor through}, $f_{\bf{T}}$ factor through $\prod_{\Prym_a}^{1\leq i \leq 2m-1}\Prym_{x_i^T-x_{2m}^T, a}$. Notice that they have the same degree over $\Prym_a$, and hence they are isomorphic.
\end{proof}

\begin{proposition}\label{type B JM and different components}
    The fiber product of $\Prym_{W,a}$ and $\Prym_{Y, a}$ over $\Prym_a$ has two connected components.We fix a group structure on this fiber product and denote the identity component as $\Prym_{\overline{\bf{O}}_B, a}$.
\end{proposition}

\begin{proof}
    By Lemma \ref{lem:criterion of connectedness} and Proposition \ref{prop:non-trivial relation}, the fiber product of $\Prym_{W,a}$ and $\Prym_{Y,a}$ over $\Prym_a$ has two connected components.
\end{proof}

Now, we are ready to prove the third part of Proposition \ref{L_BC}.
\begin{proof}[Proof of Proposition \ref{L_BC} (third part: geometry of $h_{\overline{\bf{O}}_B}^{-1}(a)$)]

    We need to describe the geometric structure of $h_{\overline{\bf{O}}_B}^{-1}(a)$. This is parallel to the construction of $\Prym_{\overline{\bf{O}}_B, a}$, and hence we see that $h_{\overline{\bf{O}}_B}^{-1}(a)$ has two connected components and each one is a torsor of $\Prym_{\overline{\bf{O}}_B, a}$.
    
    Finally we need to determine the canonical point $\sE_0$ on $h_{\overline{\bf{O}}_B}^{-1}(a)$. Notice that there is a natural square root $\omega_{\bar{\Sigma}_a}^{1/2}$ of $\omega_{\bar{\Sigma}_a}$ by Proposition~\ref{prop:canonical bundle of normalized} and Lemma~\ref{lem:even coefficients of ramification}. Now we consider $(\bar{\pi}_a)_*\omega_{\bar{\Sigma}_a}^{1/2}$, which is a rank $2n$ bundle on $\Sigma$ with a nondegenerate skew-symmetric pairing taking value in $\omega_{\Sigma}$. By the process above, we can use $(\bar{\pi}_a)_*\omega_{\bar{\Sigma}_a}^{1/2}$ to construct $2^{2n(2g-2)+\beta(\bf{d}_B)-c(\bf{d}_B)-1}$ many $\bar{\bf{O}}_B$-Higgs bundles. By Corollary~\ref{vector ponent is not zero}, Lemma~\ref{two kinds of double cover and factor through} and the process above again, the different choices $\bar{\bf{O}}_B$-Higgs bundles are determined by a sign at each ramification points of $\bar{\Sigma}_a$ (relative to $\bar{\Sigma}_a/\sigma$), but this is given as follows. Notice that $\omega_{\bar{\Sigma}_a}$ is isomorphic to a line bundle pulled back from $\bar{\Sigma}_a/\sigma$ by Lemma~\ref{lem:canonical is a pull back, sum is trivial}, hence the action of $\sigma$ on each fiber of $\omega_{\bar{\Sigma}_a}$ is trivial, which means the action of $\sigma$ on the fiber of $\omega_{\bar{\Sigma}_a}^{1/2}$ at each ramification points is given by $\pm 1$. So we have a natural choice of sign at each ramification point and hence a natural point on $h_{\overline{\bf{O}}_B}^{-1}(a)$.
\end{proof}

We denote the connected component of $h_{\overline{\bf{O}}_B}^{-1}(a)$ containing the canonical point as $h_{\overline{\bf{O}}_B}^{-1}(a)^+$, and the other component as $h_{\overline{\bf{O}}_B}^{-1}(a)^-$.

\begin{corollary}\label{Cor.Not_dual}
    If $\overline{\bf{O}}_B$ and $\overline{\bf{O}}_C$ are special and related by Springer duality, and $c(\bf{d}_B)\neq 0$, then $h_{\overline{\bf{O}}_B}^{-1}(a)^\pm$ and $h_{\overline{\bf{O}}_C}^{-1}(a)$ are torsors of abelian varieties which are \emph{NOT} dual to each other. Consequently, SYZ mirror symmetry does not hold in this case.
\end{corollary}
\begin{proof}
     By Proposition~\ref{prop:two connected components}, the degree of $\Prym_a^\vee$ to $\Prym_a$ is $2^{2n(2g-2)+\beta(\bf{d}_B)-2}$. However, the degree of either $h_{\overline{\bf{O}}_B}^{-1}(a)^+$ or $h_{\overline{\bf{O}}_B}^{-1}(a)^-$ to $\Prym_a$ is $2^{2n(2g-2)+\beta(\bf{d}_B)-c(\bf{d}_B)-2}$. These degrees differ, indicating that the abelian varieties are not dual.
\end{proof}

Next, we construct a (non-canonical) $\overline{\bf{O}}_B$-Higgs bundle in $h_{\overline{\bf{O}}_B}^{-1}(a)^-$. Although this construction is not necessary for subsequent arguments, it is of independent interest.

Fix a point $y \in Y$, let $\sE_0$ denote the canonical point constructed in the third part of the proof of Proposition \ref{L_BC}. Let  $\sE_{1}$ denote the Higgs bundle obtained by changing the sign of $\sE_0$ at $y$. Then

\begin{proposition}\label{prop:const_E1}
    $\sE_1\in h_{\overline{\bf{O}}_B}^{-1}(a)^-$.
\end{proposition}
Before we give the proof, we need the following result of Mumford\cite{Mum71}.
\begin{lemma}\label{different components}
    Let $\mathbf{E}$ be a family of vector bundles together with a symmetric non-degenerate bilinear form taking value in $\omega_{\Sigma}$ on $\Sigma$ parameterized by a connected scheme $S$, then $h^0(\Sigma, \mathbf{E}_s)$ has the same parity for any $s\in S$.
 \end{lemma}
\begin{proof}[Proof of Proposition \ref{prop:const_E1}]
    We firstly tensor $\sE_i$ with $\omega_{\Sigma}^{1/2}$ to be $\tilde{\sE_i}$ and hence the pairings of $\tilde{\sE_i}$ lie in $\omega_{\Sigma}$. By Lemma \ref{different components}, we only need to show that $h^0(\Sigma, \tilde{\sE_i})$ have different parity.

    By the construction of $\tilde{\sE_i}$, we have the following exact sequences: 
     \begin{align*}
         0\longrightarrow \mathrm{H}^0(\Sigma, \tilde{\sE_i})\longrightarrow  \mathrm{H}^0(\Sigma, \bar{\pi}_{*}\omega_{\bar{\Sigma}_a}^{1/2})\oplus \mathrm{H}^0(\Sigma, (\Ker \theta_B)^\vee \otimes \omega_{\Sigma}^{1/2}) \stackrel{(f_i, g_i)}{\longrightarrow} \tilde{Q} 
     \end{align*}

     Hence $h^0(\Sigma, \tilde{\sE_i})$ are determined by images of $f_i$ and $g_i$ as
    $$
         h^0(\Sigma, \tilde{\sE_i})=\dim \Ker f_i +\dim \Ker g_i + \dim \Im f_i\cap \Im g_i.
     $$
     By our construction, we may assume that $f_0=f_1$ and for any $s\in \mathrm{H}^0(\Sigma, (\Ker \theta_B)^\vee \otimes \omega_{\Sigma}^{1/2})$, $g_0(s)=-g_1(s)$ at $y$ and $g_0(s)=g_1(s)$ at other points in $Y$. Then thus $\dim \Ker f_0=\dim \Ker f_1$ and $\dim \Ker g_0=\dim \Ker g_1=0$ since $g_i$ are induced by 
     $$
         0\longrightarrow \omega_{\Sigma}^{-n}((\delta-n)x)\otimes \omega_{\Sigma}^{1/2} \stackrel{\times a_{2n}} {\longrightarrow} (\Ker \theta_B)^\vee \otimes \omega_{\Sigma}^{1/2} \longrightarrow \tilde{Q}.
     $$
     Now we only need to compare $\dim \Im f_0\cap \Im g_0$ and $\dim \Im f_1\cap \Im g_1$. Since $f_0=f_1$, $g_0$ and $g_1$ are different only at $y$. Thus we may only compare the subspaces of $\dim \Im f_0\cap \Im g_0$ and $\dim \Im f_1\cap \Im g_1$ supporting at $y$. In this situation, we may assume that $\Im g_0=\Im f_0=\Im f_1$ by our construction, and hence we only need to compare $\Im g_0$ and $\Im g_0\cap \Im g_1$. Now $\Im g_0$ equals to the space of global sections of $(\Ker \theta_B)^\vee \otimes \omega_{\Sigma}^{1/2}$ and $\Im g_0\cap \Im g_1$ equals to the space of global sections of $(\Ker \theta_B)^\vee \otimes \omega_{\Sigma}^{1/2}$ which vanishes at $y$. Since $\deg (\Ker \theta_B)^\vee \otimes \omega_{\Sigma}^{1/2} >2g(\Sigma)-1 $ we see that $\dim \Im g_0\cap \Im g_1=\dim \Im g_0-1$ and hence $h^0(\Sigma, \tilde{\sE_i})$ have different parity.
\end{proof}

\subsection{Generic fibers in Richardson cases}

When $\bf{O}_{B, R}$ and $\bf{O}_{C, R}$ are Richardson and related by Springer duality, let $P_B$ and $P_C$ denote their respective polarizations. We investigate the generic Hitchin fibers of the following Hitchin systems:
\begin{equation*}
    \begin{tikzcd}
        \bf{Higgs}_{P_B} \ar[rd, "h_{P_B}"'] & & \bf{Higgs}_{P_C} \ar[ld, "h_{P_C}"] \\
        & \textbf{H} &
    \end{tikzcd},
\end{equation*}
focusing on $h_{P_B}^{-1}(a)$ and $h_{P_C}^{-1}(a)$ for $a \in \textbf{H}^{\text{KL}}\subset \textbf{H}$.

Let $\bf{d}_{C, R} = [d_{C,1}, \ldots, d_{C,r}]$, and let $I(P_C)$ be the index set defined in subsection~\ref{s.spal_fiber}. Additionally, let $\bar{\pi}_a^{-1}(x) = \{ x_1, x_2, \ldots \}$ denote the $\sigma$-fixed points over $x$.
\begin{definition}\label{def:Prym_P_C}
    Let $\Prym_{P_C, a}$ be a cover of $\Prym_a$ defined as follows
    \[
        \Prym_{P_C, a} = \prod_{\Prym_a}^{i\in I(P_C)} \Prym_{x_i-x_{i+1}, a}.
    \]
    We fix a group structure on $\Prym_{P_C, a}$.
\end{definition}

\begin{proposition}\label{p.h_P_C}
    The generic fiber $h^{-1}_{P_C}(a)$ is a torsor over $\Prym_{P_C, a}$. 
\end{proposition}
\begin{proof}
    Consider $(E_C, g_C,\Fil^\bullet_{P_{C, JM}}, \theta_C) \in h^{-1}_{\overline{\bf{O}}_{C, R}}(a)$ for $a \in \textbf{H}^{\text{KL}}$, we have $\Res_x(\theta_C) \in \bf{O}_{C, R}$. By the definition of Jacobson--Morozov resolution, we identify
    \[
        h^{-1}_{\overline{\bf{O}}_{C, R}}(a) = \{(E_C, g_C, \theta_C) \mid \mathrm{char}(\theta_C) = a \}.
    \]
    The natural map
    \[
        (E_C, g_C,\Fil^\bullet_{P_C}, \theta_C) \mapsto (E_C, g_C, \theta_C)
    \]
    relates $h^{-1}_{P_C}(a)$ and $ h^{-1}_{\bar{\bf{O}}_C}(a)$. From Proposition~\ref{p.Springer-fiber}, the fibers are governed by $\prod_{j \in I(P_C)} \OG(1, V_{C, j})$.

    More precisely, consider the universal line bundle $\sP$ over $\Prym_a \times \overline{\Sigma}_a$, which induces the vector bundle
    \[
        \sE_{C, a} = (\id \times \bar{\pi}_a)_* \sP
    \]
    over $\Prym_a \times \Sigma$. Restricting to $x$, Corollary \ref{Cor.sub_splitting} gives the splitting
    $$ 
        \sE_{C, a}|_{\Prym_a \times x}:= E_{C, a}\cong \oplus_{i=1}^{r} E_{C, a, i}.
    $$ 
    There is a morphism $\tilde{\theta}_{C,a}: E_{C, a}\rightarrow E_{C, a}$ with restriction $\tilde{\theta}_{C, a,i}: E_{C, a, i}\rightarrow E_{C, a, i}$. The family version of $V_{C, j}$, for $\in I(P_C)$ is given by  
    $$
        \sV_{C,j}=\sV_{d_{C,j}}  \oplus \sV_{d_{C,j+1}}, 
    $$
    where
    $$
        \sV_{d_{C,j}} = \frac{\Ker (\tilde{\theta}_{C, a, j})^{d_{C, j}/2}}{\Ker (\tilde{\theta}_{C, a, j})^{d_{C, j}/2-1}}, 
        \quad 
        \sV_{d_{C,j+1}} = \frac{\Ker (\tilde{\theta}_{C, a, j+1})^{d_{C, j+1}/2}}{\Ker (\tilde{\theta}_{C, a, j+1})^{d_{C, j+1}/2-1}}.  
    $$

    Finally, let $\D_x$ be the formal neighborhood of $x$, then $\bar{\pi}_a^{-1}(\D_x) = \sqcup_{i} \D_{x_i}$. Restricting $\sP$ to $\Prym_a \times \D_{x_i}$, denote by ${\sP}_{\D_{x_i}}$, the local computations yield 
    $$
    \sV_{d_{C,j}} \cong {\sP}_{\D_{x_i}}((-d_{C, j}/2+1)x_i)/ \sP_{\D_{x_i}}(-d_{C, j}/2 \cdot x_i) \cong \sP_{x_i}.
    $$
    Thus, we conclude.
\end{proof}

\begin{corollary}\label{Cor.factor through}
    The morphism $L_{BC}: h_{\overline{\bf{O}}_{B, R}}^{-1}(a)\rightarrow h^{-1}_{\overline{\bf{O}}_{C, R}}(a)$ in Proposition \ref{L_BC} factor though $h_{P_C}^{-1}(a)\rightarrow h^{-1}_{\overline{\bf{O}}_{C, R}}(a)$. Over $h^{-1}_{P_C}(a)$, we have a canonically defined point.
\end{corollary}

\begin{proof}
    This is indicated by Corollary \ref{why factor though}, Proposition \ref{L_BC}, and Proposition~\ref{p.h_P_C}.
\end{proof}

For type B, let $\bf{d}_{B, R}=[d_{B,1}, \ldots, d_{B, r'}]$, for $r' \equiv 1$. Since $\bf{d}_{C, R} = {}^S\bf{d}_{B, R}$, use $I(P_B)$ to index parts of $\bf{d}_{C, R}$.
\begin{definition}
Consider 
\[
     \Prym_{W,a}\times_{\Prym_a} \Prym_{Y,a} \times_{\Prym_a} \prod^{i \in I(P_B)}_{\Prym_a} \Prym_{x_i - x_{i+1}, a},  
\]
where $x_i-x_{i+1}= x_i$ if $x_{i+1}$ does not exist. We fix a group structure on this fiber product and denote the identity component as $\Prym_{P_B, a}$.
\end{definition}

Then we have

\begin{proposition}\label{p.h_P_B}
    The fiber $h^{-1}_{P_B}(a)$ has two connected components, denoted by $h^{-1}_{P_B}(a)^+$ and $h^{-1}_{P_B}(a)^-$. Each component is a torsor over $\Prym_{P_B, a}$. Furthermore, $h^{-1}_{P_B}(a)^+$ contains a canonical point.
\end{proposition}
\begin{proof}
    Similarly, we have a vector bundle $\sE_{B, a}$ over $\Prym_{P_B, a} \times \Sigma$. However, we only have a coarse splitting of $\sE_{B, a}$:
    \[
        \sE_{B, a}|_{\Prym_{P_B, a} \times x} := E_{B, a} \cong \oplus_{i=1}^s E_{B, a, \bf{T}_i},
    \]
    where $\bf{d}_{B, R}=\bf{T}_B:=[\bf{T}_1, \ldots, \bf{T}_s]$ is defined in Lemma~\ref{structure of special partition}. By Proposition \ref{p.Springer-fiber} and similar arguments as Proposition \ref{p.h_P_C}, we conclude.
\end{proof}

Concluding all the above, we have

\begin{theorem}\label{Thm:dual abelian}
     For $a\in\bf{H^{\text{KL}}}$, various Hitchin fibers fit into the following diagrams
     \begin{equation*}  
         \begin{tikzcd}
              h_{P_B}^{-1}(a)^{\pm} \ar[d, "\nu_{P_B}"] & h_{P_C}^{-1}(a)\ar[d,"\nu_{P_C}"] & \Prym_{P_B,a}\ar[d, "\nu_{P_B}"]& \Prym_{P_C,a}\ar[d,"\nu_{P_C}"]\\
              h_{\overline{\bf{O}}_{B, R}}^{-1}(a)^{\pm} \ar[ur]\ar[r,"L_{BC}"]& h_{\overline{\bf{O}}_{C, R}}^{-1}(a) & \Prym_{\overline{\bf{O}}_{B, R},a}\ar[ur]\ar[r]& \Prym_a
         \end{tikzcd}.  
     \end{equation*}
     The right-hand side represents abelian varieties, and the left-hand side fibers are torsors over the corresponding abelian varieties. Additionally:
     \begin{enumerate}
         \item $\deg \nu_{P_B} \cdot \deg \nu_{P_C} = 2^{c(\bf{d}_{B, R})}. $
         \item $\Prym_{P_B,a}$ and $\Prym_{P_C,a}$ are dual abelian varieties.
     \end{enumerate}  
\end{theorem}
\begin{proof}
     By Lemma \ref{lem:criterion of duality}, it suffices to show that the Weil pairing vanishes on $\sA(P_C)\times \sA(P_B)$, as defined in \eqref{A(P_C)} and \eqref{A(P_B)}. Then it follows from Lemma \ref{lem:A(P_B)}.
\end{proof}

\begin{remark}
    The fibers $h_{\overline{\bf{O}}_{C, R}}^{-1}(a)$ and $h^{-1}_{\overline{\bf{O}}_{B, R}}(a)$ depend only on the Richardson orbits $\bf{O}_{C, R}$ (or equivalently $\bf{O}_{B, R}$). However, $h_{P_B}^{-1}(a)$ and $h_{P_C}^{-1}(a)$ depend on the choice of dual polarizations $P_B$ and $P_C$. For instance, let $P'_C$ be another polarization of $\bf{O}_{C,R}$ such that $\deg \mu_{P'_C} > \deg \mu_{P_C}$. From \cite[Lemma 5.2]{FRW24}, we know $I(P'_C) \supset I(P_C)$. Consequently, from the proof of Proposition \ref{p.h_P_C}, $h^{-1}_{P'_C}(a)$ is a degree $\deg \mu_{P'_C}/\deg \mu_{P_C}$ cover of $h^{-1}_{P_C}(a)$.
\end{remark}

The following section sets the stage for analyzing various torsor structures, ultimately leading to a rigorous formulation of Strominger--Yau--Zaslow mirror symmetry.

\section{SYZ mirror symmetry for parabolic Hitchin systems}\label{Sec:SYZ}

In the previous section, we identified the generic fibers of various Hitchin systems as torsors over a family of Prym varieties and their finite covers.

One of the central aspects of SYZ mirror symmetry is the identification of certain torsors smoothly. In our setting, we will demonstrate that the generic fibers of $h_{\overline{\bf{O}}_{C, R}},h_{P_C}$ and the ``+" connected components of generic fibers of $h_{P_B}$ and $h_{\overline{\bf{O}}_{B, R}}$ are trivial torsors over Prym varieties. This identification is achieved by constructing canonical sections for these generic fibers.\footnote{In other words, they are all trivial fibers.} This construction reduces to finding rational points on generic fibers over the function field of the Hitchin base. Therefore, in this section, we allow $\Bbbk$ to be a general field.

Our strategy leverages theta characteristics of the base curve $\Sigma$ and the existence of rational points (over the function field of the Hitchin base) for relative Picard varieties of the families of normalized spectral curves. The existence of such rational points is guaranteed by analyzing successive blow-ups of singular spectral curves.

\subsection{Trivial Torsors: ``$+$'' component}
We begin by recalling the formula for the canonical line bundle of normalized spectral curves, following \cite{SWW22t}. For clarity, we state the result for type A, i.e., partitions $\bf{d}$ with conjugate partitions $\bf{d}^t=[\lambda_1, \lambda_2,\ldots]$.

\begin{proposition}[Proposition 2.3.4 \cite{SWW22t}]\label{prop:canonical bundle of normalized}
   Let $\omega_{\overline{\Sigma}_a}$ denote the canonical line bundle the of normalized spectral curve $\overline{\Sigma}_a$. Then:
     \[
    \omega_{\bar{\Sigma}_a}=\bar{\pi}_a^*(\omega_{\Sigma}^{n}\otimes \sO_{\Sigma}(nD))\otimes \sO_{\bar{\Sigma}_a}(\sum_{i=1}^{d_1}(-\sum_{j=1}^{i}\lambda_j)R_i),
     \]
     where $R_i$ is the divisor with ramification degree $i$ over the marked point $x\in\Sigma$. In particular, all $R_i$ are defined over $\Bbbk$, even if $\Bbbk$ is not algebraically closed.
\end{proposition}

Now assume $\bf{d}=\bf{d}_C$ is a special partition of type C, such that its conjugate partition is also special. Proposition~\ref{prop:canonical bundle of normalized} applies directly to type C. We then deduce the following:

\begin{lemma}\label{lem:even coefficients of ramification}
    If $\bf{d}=\bf{d}_C$ is a special partition of type C, then all the coefficient of $R_i$, i.e., $-\sum_{j=1}^{i}\lambda_j$,  are even.
\end{lemma}
\begin{proof}
    We can rewrite
    \begin{align*}
        \sum_{j=1}^{i}\lambda_j=\sum_{j=1}^{i-1} j(\lambda_j-\lambda_{j+1})+i\lambda_i.
    \end{align*}
    Observe that $\lambda_i-\lambda_{i+1}=\#\{\ell \mid d_{\ell}=i\}$. Then the lemma follows from the combinatorial description of special partitions.
\end{proof}

\begin{proposition}\label{prop:generic fiber is trivial torsor}
    Suppose $\omega_{\Sigma}^{1/2}$, the square root of $\omega_{\Sigma}$, exists over $\Bbbk$. Then, there is a canonical identification between $\Prym(\sR)$ with $\Prym$ over $\mathbf{H}^{\text{KL}}$. In particular, generic fibers of $h_{\overline{\bf{O}}_C}$ and the ``$+$'' components of generic fibers of $h_{\overline{\bf{O}}_B}$ are trivial torsors over corresponding abelian varieties.
\end{proposition}
\begin{proof}
    Proposition \ref{prop:canonical bundle of normalized} holds over general fields, including the function field of the Hitchin base. Let $a$ denote the generic point of the Hitchin base. Consequently, $\sO_{\Sigma}(2nD)$ and $\sO_{\bar{\Sigma}_a}(\sum_{i=1}^{d_1}(-\sum_{j=1}^{i}\lambda_j)R_i)$ admit square roots defined over the function field. Thus, we can choose a square root of the ramification divisor $\sR_a^{1/2}$ (see Definition~\ref{def:Prym var ramification divisor} and Proposition~\ref{Prop:fiber_JM_moduli}) defined over the function field. By Proposition~\ref{Prop:fiber_JM_moduli}, the fiber $h_{\overline{\bf{O}}_C}^{-1}(a)$ is canonically isomorphic to
    \[
    \{\sL\in\Pic^{\frac{\deg \sR_a}{2}}(\bar{\Sigma}_a) \mid \sigma^*\sL\cong\sL^{\vee}(\sR_a)\}.
    \]
    Since $\sR_a^{1/2}$ lies in this set and is a rational point over the function field, the generic fiber of $h_{\overline{\bf{O}}_C}$ is a trivial torsor over the relative Prym variety. In particular, there is a section over $\bf{H^{\text{KL}}}$.

    Since the section $a_{2n}$ is defined over the function field of the Hitchin base, and the description of $W$'s in Proposition \ref{description of W} also works over a general field(see Remark \ref{W is defined over general field}), hence $\Prym_{Y,a}$ and $\Prym_{W,a}$ are constructed over the function field of the Hitchin base. As a consequence, the construction of $h_{\overline{\bf{O}}_B}^{-1}(a)^+$ in Proposition \ref{L_BC} also works over the function field of the Hitchin base. Now, carrying out the construction  from  the section of $h_{\overline{\bf{O}}_C}$ over $\bf{H^{\text{KL}}}$, we obtain sections of $h_{\overline{\bf{O}}_B}, h_{P_B}$ over $\bf{H^{\text{KL}}}$ which lie in the ``$+$'' component. In particular, $h_{\overline{\bf{O}}_B}^{-1}(a)^+$ and $h_{P_B}^{-1}(a)^+$ are trivial torsors over $\bf{H^{\text{KL}}}$.
\end{proof}

\subsection{Nontrivial Torsors: ``$-$'' component}\label{sec:-component}

As in Section \ref{Ap:Prym_var}, fix a point $y_0\in \sR_a$. We can construct a fiber product
\[
\prod_{y\in \sR_a\setminus \{y_0\}}\Prym_{y,a},
\]
where each $\Prym_{y,a}$ is a double cover of $\Prym_a$ associated with the 2-torsion line bundle $\sO_{\bar{\Sigma}_a}(y-y_0)\in \Jac(\bar{\Sigma}_a)$. By Proposition \ref{prop:two connected components}, this fiber product has two connected components: a ``$+$" component and a ``$-$" component. By Proposition \ref{prop:const_E1}, we know that $\Prym_{a}^{\vee}$ corresponds to the ``$+$" component. 

To analyze this further, consider the following commutative diagram:
\[
\begin{tikzcd}
    &\Pic^0(\bar{\Sigma}_{a})\ar[rd,"\phi^0"]\\
    &\Pic^1(\bar{\Sigma}_{a})\ar[r,"\phi^1"]\ar[u,"\otimes\sO(-y_0)"]& \Prym_{a} \subset \Pic^0(\bar{\Sigma}_a)
\end{tikzcd}.
\]
Both the maps $\phi^0,\phi^1$ from $\Pic^0,\Pic^1$ to $\Pic^0$ are given by
\[
\sL\mapsto \sigma^*\sL^{\vee}\otimes\sL,
\]
whose images lie in $\Prym_a$. Restricting to $\Prym_{a}$, this induces
\[
\Prym_a \xrightarrow{[2]} \Prym_a
\]
as in \eqref{eq:mult 2 map}, factoring through $\Prym_{a}^{\vee}$. Actually, $\phi^0$ and $\phi^1$ factor through $\Prym_{y,a}$, but the factorization is not unique, which depends on the trivialization of the pullback of those 2-torsion line bundle on $\Prym_a$.

For all $z\in\sR_a\setminus\{y_0\}$, the following diagram commutes
\[
\begin{tikzcd}
   &\Pic^0(\bar{\Sigma}_{a})\ar[rrd, "\phi^0" near start]\ar[rr,dotted,"\phi^0_{y}" description]& &\Prym_{y,a}\ar[d]\ar[loop right,"\iota_{y}"]\\
    &\Pic^1(\bar{\Sigma}_{a})\ar[rr,"\phi^1"]\ar[u,"\otimes\sO(-z)"]\ar[urr,dotted]& &\Prym_a\subset \Pic^0(\bar{\Sigma}_a) 
\end{tikzcd}.
\]
Here, $\iota_{y}$ is the natural involution on the double cover $\Prym_{y-y_0,a}$, which exchanges its two sheets. 

\begin{proposition}
    The map
    \[
    \iota_{z}\circ \prod\phi^0_{y}\circ (\otimes\sO(-z)): \Pic^1(\bar{\Sigma}_{a})\rightarrow\prod\Prym_{y,a}
    \]
    is independent of the choice of $z\in\sR_a$. Furthermore, the image lies in a different connected component from the image of $\Pic^0(\bar{\Sigma}_{a})$.
\end{proposition}
\begin{proof}
    The first argument follows from
    \[
    \phi_{y}^{0}(\sO(z-y_0))\ne 0, \forall z\ne y.
    \]
    The second argument follows from Lemma \ref{lem:permute conn comp}.
\end{proof}

\begin{proposition}\label{prop:deg 1 and non-spin}
    The image of $\Pic^1(\bar{\Sigma}_a)$ in $\prod_{y\in \sR_a\setminus \{y_0\}}\Prym_{y,a}$ under the map $\prod \phi^1_{y}$ corresponds to the ``$-$" component.
\end{proposition}
\begin{proof}
    Let $\sE_0$ denote the canonical $\SO_{2n+1}$-Higgs bundle constructed in Proposition~\ref{L_BC}. This bundle is the image of $0 \in \Pic^0(\bar{\Sigma}a)$, serving as the identity element in the fiber product $\prod_{y\in\sR_a\setminus \{y_0\}}\Prym_{y,a}$. It then determine all $\{\phi^0_{y}\}_{y\in\sR_a}$.

    First, note that $\phi^0(\sO(y-y_0))=0$ for all $y\in\sR_a\setminus\{y_0\}$. By Proposition \ref{prop:weil pairing}, we deduce
    \[
    \phi^0_{y}(y)=0,\quad \phi^0_{y}(z)\ne 0 \;\quad \forall z\in\sR_a\setminus\{y_0,y\}.
    \]
    By combining this observation with the construction in Proposition~\ref{type B JM and different components} and the choice of $-1$ at $y$ for some $y \in \sR_a$, we conclude that the ``$-$'' component corresponds to the image of $\Pic^1(\bar{\Sigma}a)$.
\end{proof}

\subsection{ Strominger--Yau--Zaslow mirror symmetry}

Following the strategies in \cite{DP12} and \cite{GWZ20}, we first introduce $\mu_2$-gerbes on the various Higgs moduli spaces. Since $\SO_{2n+1}$ is an adjoint group, following \cite{DP12}, we choose the trivial $\mu_2$-gerbe $\alpha_B$ on $\bf{Higgs}_{P_B}$\footnote{It is a simplified statement of "trivial equivariant $\mu_2$-gerbes" on generic Galois cover moduli of parabolic (twisted) Spin higgs bundles.}. On the other hand, we define the $\mu_2$-gerbe $\alpha_C$ as the lifting gerbe of the universal (parabolic) $\PSp_{2n}$-Higgs bundle over $\bf{Higgs}_{P_C}$. 

\begin{lemma}
    The $\mu_2$-gerbe $\alpha_C$ is an arithmetic gerbe,\footnote{See \cite[\S 6.1]{GWZ20} for the definition of arithmetic gerbes.} i.e., it splits over a finite cover of $\bf{H}^{\text{KL}}$.
\end{lemma}
\begin{proof}
    Over $\bf{H^{\text{KL}}}$, we have a natural finite map
    \[    \bf{Higgs}_{P_C}|_{\bf{H}^{\text{KL}}}\rightarrow\bf{Higgs}_{\overline{\bf{O}}_C}|_{\bf{H}^{\text{KL}}},
    \]
    which originates from
    \[
    T^*(\Sp_{2n}/P_C)|_{\bf{O}_C}\rightarrow \bf{O}_{C}.
    \]
    Thus, over $\bf{H}^{\text{KL}}$, the lifting gerbe $\alpha_C$ is the pullback of a lifting gerbe on $\bf{Higgs}_{\overline{\bf{O}}_C}$, which we also denote by $\alpha_C$ (to avoid ambiguity). By Theorem~\ref{family version of filtration}, over $\bf{H^{\text{KL}}}$, the moduli space $\bf{Higgs}_{\overline{\bf{O}}_C}$ can be identified with some line bundles over normalized spectral curves. Therefore, over some \'etale cover of $\bf{H}^{\text{KL}}$ (if necessary), the push forward of these line bundles provides the lifting of the universal $\PSp_{2n}$-bundles. As a result, $\alpha_{C}$ splits, i.e., it's a pullback of a gerbe from the base. 
\end{proof}

The relative splitting of $\alpha_{C}$ with respect to $h_{P_C}$ over $\bf{H}^{\text{KL}}$ defines a $(\Prym_{P_C})^{\vee}[2]$-torsor, denoted by $\Split(\bf{Higgs}_{P_C},\alpha_C)$. Following \cite[Definition 6.4]{GWZ20}, we have a $(\Prym_{P_C})^{\vee}$-torsor:
\[
\Split'(\bf{Higgs}_{P_C}|_{\bf{H}^{\text{KL}}},\alpha_C|_{\bf{H}^{\text{KL}}}):=\Split(\bf{Higgs}_{P_C},\alpha_C)\times_{(\Prym_{P_C})^{\vee}[2]} (\Prym_{P_C})^{\vee}
\] 
By Theorem \ref{Thm:dual abelian}, it is a torsor over $\Prym_{P_B}$.

\begin{theorem}\label{thm:syz mirror}
    The SYZ mirror symmetry holds for
    \[
    \begin{tikzcd}
        (\bf{Higgs}_{P_C},\alpha_C)\ar[rd]&&(\bf{Higgs}_{P_B},\alpha_B)\ar[ld]\\
        &\mathbf{H}
    \end{tikzcd}
    \]
    More specifically:
    \begin{enumerate}
        \item $\bf{Higgs}_{P_C}|_{\bf{H}^{\text{KL}}}$ and $\bf{Higgs}_{P_B}^{+}|_{\bf{H}^{\text{KL}}}$ are trivial torsors over their corresponding abelian schemes, which are family (over $\bf{H}^{\text{KL}}$) of the abelian varieties.
        \item $ \Split'(\bf{Higgs}_{P_C}|_{\bf{H}^{\text{KL}}},\alpha_C|_{\bf{H}^{\text{KL}}})\cong \bf{Higgs}_{P_B}^{-}|_{\bf{H}^{\text{KL}}}$.
        \item $\Split'(\bf{Higgs}^{\pm}_{P_B}|_{\bf{H}^{\text{KL}}},\alpha_B|_{\bf{H}^{\text{KL}}})\cong \bf{Higgs}_{P_C}|_{\bf{H}^{\text{KL}}}$
    \end{enumerate}
\end{theorem}
\begin{proof}
    By Proposition \ref{prop:generic fiber is trivial torsor}, the generic fibers of $h_{\overline{\bf{O}}_C}$ are trivial torsors over the relative Prym variety. Therefore, the generic fibers of $h_{P_C}$ are also trivial torsors over $\Prym_{P_C}$. In particular, there is a canonical isomorphism
    \[   \Split'(\bf{Higgs}_{P_B}|_{\bf{H}^{\text{KL}}},\alpha_B|_{\bf{H}^{\text{KL}}})\cong \bf{Higgs}_{P_C}|_{\bf{H}^{\text{KL}}}.
    \]
    Next, consider the fibers of $h_{P_B}^{-1}(a)$, described as
    \[
        \Prym_{W,a}\times_{\Prym_a} \Prym_{Y,a} \times_{\Prym_a} \prod^{i \in I(P_B)}_{\Prym_a} \Prym_{x_i - x_{i+1}}.
    \]
    where $\prod_{i=1}^{2N-1} \Prym_{x_i - x_0}$ surjects onto $h_{P_B}^{-1}(a)$, as shown in Lemma~\ref{lem:depends only on subspace}. This map induces a bijection between the connected components of the fibers.

    Since $\alpha_C$ is the lifting gerbe of projective symplectic universal family, by a similar argument to \cite[Proposition 3.2]{HT03}, combined with  Proposition~\ref{prop:deg 1 and non-spin}, which relates the degree-1 Picard variety to the ``$-$" component, we obtain the canonical isomorphism
    \[
    \Split'(\bf{Higgs}_{P_C},\alpha_C)\cong \text{Image}(\Pic^1(\bar{\Sigma}_a)\mapsto \prod^{2N-1}_{i=1}\Prym_{x_i-x_0}).
    \]
    Thus, $\Split'(\bf{Higgs}_{P_C},\alpha_C)$ corresponds to the ``$-$'' component.

    Since $\alpha_B$ is a trivial gerbe, and over the complex field, we can always find a square root $\omega_{\Sigma}^{1/2}$, hence the last statement follows from Proposition \ref{prop:generic fiber is trivial torsor}.
\end{proof}

\section{Topological mirror symmetry for parabolic Hitchin systems} \label{Sec:TMS}

In Theorem \ref{thm:syz mirror}, we establish the SYZ mirror symmetry between $(\bf{Higgs}_{P_C},\alpha_C)$ and $(\bf{Higgs}_{P_B},\alpha_B)$, where $\alpha_B$ is trivial. In this section, we demonstrate that topological mirror symmetry also holds by leveraging the $p$-adic integration technique described in \cite[\S 6]{GWZ20}. To apply $p$-adic integration, we first verify the smoothness of the relevant moduli spaces (or stacks) and the properness of the associated Hitchin maps.

Since $\SO_{2n+1}$ is an adjoint group, the moduli spaces on the type B side should be treated as orbifolds. These have finite schematic covers by moduli spaces corresponding to the simply connected group $\Spin_{2n+1}$. Specifically, the diagram below illustrates this relationship:
\[\small
\begin{tikzcd}[column sep=-0.15em]
    \bf{Higgs}_{\Spin_{2n+1},P_B'}\ar[loop right]\ar[rdd] & &\Pic^0(\Sigma)[2]&&\bf{Higgs}_{\Spin_{2n+1},\overline{\bf{O}}_B}\ar[ldd]\ar[loop left]\\
   \\
    &\bf{Higgs}_{\SO_{2n+1}, P_B} && \bf{Higgs}_{\SO_{2n+1},\overline{\bf{O}}_B}
\end{tikzcd}.
\]
Here, $P_B'$ denotes the preimage of $P_B$ in $\Spin_{2n+1}$. Thus, we need to ensure that $\bf{Higgs}_{P_C}$ and $\bf{Higgs}_{\Spin_{2n+1},P'_B}$ are smooth varieties and that the corresponding Hitchin maps are proper. This can be achieved by finding generic weights such that the semistability condition coincides with the stability condition.

Let $P_C$ be a parabolic subgroup of $\Sp_{2n}$ with Levi type $(p_1, \ldots, p_k; q)$. If $P_B$ and $P_C$ are dual, then the Levi type of $P_B$ is $(p'_1, \ldots, p'_k; q+1)$, where $(p'_1, \ldots, p'_k)$ is a permutation of $(p_1, \ldots, p_k)$. Motivated by \cite[Lemma 4.3]{She24}, we assume the following condition:
\begin{condition}\label{cond:coprime}
    $\gcd(p_1,\ldots,p_{k}, 2q)=\gcd(p_1,\ldots,p_{k}, 2q+1)=1$.
\end{condition}
This condition ensures that semistability and stability coincide for generic weights. Consequently, the moduli spaces $\bf{Higgs}_{P_C}$ and $\bf{Higgs}_{\Spin_{2n+1}, P_B'}$ of stable objects are smooth varieties, and their corresponding Hitchin maps are proper.
\subsection{Self-dual Isogeny}\label{sec:self dual}

Recall the following commutative diagrams over the $\bf{H}^{\KL}$:
  \begin{equation*}  
         \begin{tikzcd}
              \bf{Higgs}^{\KL,\pm}_{P_B}\ar[r,"p_{BC}"] \ar[d, "\nu_{P_B}"'] &  \bf{Higgs}^{\KL}_{P_C}\ar[d,"\nu_{P_C}"] & \Prym_{P_B}\ar[d]& \Prym_{P_C}\ar[d]\\         \bf{Higgs}^{\KL,\pm}_{\overline{\bf{O}}_{B, R}} \ar[ur]\ar[r,"L_{BC}"]&  \bf{Higgs}^{\KL}_{\overline{\bf{O}}_{C, R}} & \Prym_{\overline{\bf{O}}_{B, R}}\ar[ur]\ar[r]& \Prym
         \end{tikzcd}.  
     \end{equation*}
   We use $\bullet$ to denote $B$ or $C$. Here $\bf{Higgs}_{P_{\bullet}}^{\KL}$ is the preimage of $h_{P_{\bullet}}$ over $\bf{H}^{\KL}$. Hence $\bf{Higgs}_{P_{\bullet}}^{\KL}\rightarrow \bf{H}^{\KL}$ is a fibration of abelian torsors. Notice that for any $a\in \bf{H}^{\KL}$, the fiber $h_{P_{\bullet}}^{-1}(a)$ is isogeny to a torsor of the Prym variety of the normalized spectral curve $\bar{\Sigma}_a$.
   Here $p_{BC}$ is a morphism between torsors induced by isogeny of the family of Prym varieties.

   Notice that the isogeny between $\Prym_{P_{B}}\rightarrow \Prym_{P_{C}}$ is induced from the isogeny $\Prym^{\vee}\rightarrow\Prym$ where the polarization are restriction of the natural polarizations of relative Jacobians.

   Hence we can conclude that the isogeny is self-dual as in Condition (b) \cite[Definition 6.]{GWZ20}

\subsection{Rational Points and Splitting of Gerbes}
 We first choose a square root $\omega_{\Sigma}^{1/2}$ of $\omega_{\Sigma}$ over the complex field $\mathbb{C}$. Then we can choose a finitely generated $\mathbb{Z}$-algebra contained in $\mathbb{C}$ such that $\omega_{\Sigma}^{1/2}$ can be defined over $R$. 
 
\begin{lemma}\label{lem:trivial torsor for para B}
    Under Condition \ref{cond:coprime},  for any ring homomorphism $R\rightarrow \sO_F$, $\bf{Higgs}_{P_B}^{-}$ is a trivial torsor over $\Prym_{P_B}$.
\end{lemma}
\begin{proof}
    By Condition \ref{cond:coprime}, we can construct a $F$-rational point on $\Pic^{1}(\overline{\Sigma}_a)$ for the function field of $\bf{H}^{\KL}$, see \cite{SWW22t}. The lemma follows from Proposition \ref{prop:deg 1 and non-spin}.   
\end{proof}
\begin{corollary}
     For any ring homomorphism $R\rightarrow \sO_{F}$, the gerbe $\alpha_C|_{h^{-1}_{P_C}(a)}$ splits over $F$ for $a\in \bf{H}(\sO_F)\cap \bf{H}^{\KL}(F)$.
\end{corollary}
\begin{proof}
    By the proof of Theorem \ref{thm:syz mirror}, we have:
     \[
    \Split'(\bf{Higgs}_{P_C},\alpha_C)\cong \text{Image}\left(\Pic^1(\bar{\Sigma}_a)\mapsto \prod^{2N-1}_{i=1}\Prym_{x_i-x_0} \right).
    \]
    Hence the gerbe $\alpha_C$ is trivial along the fiber over $F$. Since $a\in \bf{H}(\sO_F)\cap \bf{H}^{\KL}(F)$, and $\Br(\sO_F)$ is trivial. Hence $\alpha_C|_{h_{P_C}^{-1}(a)}$ splits over $F$.
\end{proof}
\begin{proposition}
      For any ring homomorphism $R\rightarrow \sO_F$, $\alpha_B,\alpha_C$ both split and Hitchin fibers $h_{\bullet}^{-1 }(a)^{\pm}(F)\ne \emptyset$, for $a\in \bf{H}(\sO_F)\cap \bf{H}^{\KL}(F)$.
\end{proposition}
\begin{proof}

    By Lemma \ref{lem:trivial torsor for para B}, we only need to show that for any ring homomorphism $R\rightarrow \sO_F$, Hitchin fibers $h_{\bullet}^{-1}(a)(F)\ne \emptyset$. Since we choose $R$ such that  $\omega_{\Sigma}^{1/2}$ can be defined over $R$, then by Proposition \ref{prop:generic fiber is trivial torsor}, we have that $h_{P_B}^{-1}(a)^{+}(F)$ is nonempty. Since we have natural  morphism $h_{P_B}^{-1}(a)^{+}\rightarrow h^{-1}_{P_C}(a)$, then $h_{P_C}^{-1}(a)(F)\ne\emptyset$.
\end{proof}

\subsection{Gauge Forms and Orbifold Measure}
Notice that under Condition \ref{cond:coprime}, semistability and stability conditions coincide. In particular, the moduli stack of semistable parabolic $\SO_{2n+1}$-Higgs bundles is an admissible finite quotient stack in the sense of \cite[Definition 4.10]{GWZ20}.  In this subsection, we first construct natural symplectic forms on both sides. On B-side, we should understand this as an equivariant symplectic forms on  moduli spaces of corresponding parabolic (twsited) $\Spin_{2n+1}$ Higgs bundles at the beginning of the section. 

By \cite[Remark 4.13]{GWZ20}, see also \cite[Lemma 2.5]{She24}, determinant of these symplectic forms define gauge forms on $\bf{Higgs}_{P_C}, \bf{Higgs}_{P_B}$\footnote{On B side, due to the orbifold nature, it is actually a power of the determinant. Here for the ease of notation and expression, we omit it.  This will not affect the p-adic integration which only integrate over $\bf{H}^{\KL}$}. And when restricted to fibers over $\bf{H^{\text{KL}}}$, they are translation-invariant volume forms on those torsors. Then, we compare the gauge forms on both sides.  The main difficulty lies in that $\codim (\bf{H}\backslash\bf{H}^{\KL})=1$.  The strategy is to construct gauge forms over $\bf{H}^{\KL}$ by Serre duality between the tangent bundle of $\bf{H}^{\KL}$ and relative tangent bundle of the Hitchin map. And then relate the gauge forms constructed from the symplectic structure on moduli spaces. In conclusion, Langlands dual parabolic Hitchin systems considered here fit into the ``weak abstract dual Hitchin systems" introduced by Shen\cite{She24}. For Langlands dual parabolic $\SL_n/\PGL_n$-Hitchin systems, Shen \cite{She24} constructed natural gauge forms for general line bundles as coefficients of Higgs fields rather than $\omega_{\Sigma}(D)$ used in this paper which makes the construction of global gauge forms much harder.

As before, we assume $D=x$ for simplicity. For notation ease, in this subsection, we sometimes use $\bullet$ to denote $B$ or $C$. First recall the tangent complex of parabolic Higgs bundles, for more details see \cite{BKV18}, or more general setting \cite{BBAMY}.
The tangent complex at a parabolic Higgs bundle $\bf{E}:=(E,\theta,\ldots) \in \bf{Higgs}_{P_{\bullet}}$ is given by:
\[
    \mathscr{F}^{\bullet}_{\bf{E}}=\left[ \ad^{par}(E)\xrightarrow{\ad_{\theta}}\ad^{spar}(E)\otimes \omega_{\Sigma}(x) \right].
\]
Notice that here we choose a Killing form on Lie algebras and we have $\ad^{par}(E)^\vee \cong \ad^{spar}(E)\otimes \sO_{\Sigma}(x)$.

Due to the Condition \ref{cond:coprime}, semistablity coincide with stability. Hence the tangent space of $\bf{Higgs}_{P_{\bullet}}$ at $(E, \theta)$ is isomorphic to $\mathbb{H}^1(\mathscr{F}^{\bullet}_{\bf{E}})$, which fits into the following exact sequence:
\[
\begin{tikzcd}
    & \operatorname{H}^{0}(\Sigma,\ad^{spar}(E)) \rar & \operatorname{H}^{0}(\Sigma,\ad^{spar}(E) \otimes \omega_{\Sigma}(x)) \rar
             \ar[draw=none]{d}[name=X, anchor=center]{}
    & \mathbb{H}^1(\mathscr{F}^{\bullet}_{\bf{E}}) \ar[rounded corners,
            to path={ -- ([xshift=2ex]\tikztostart.east)
                      |- (X.center) \tikztonodes
                      -| ([xshift=-2ex]\tikztotarget.west)
                      -- (\tikztotarget)}]{dll}[at end]{\ } \\      
    & \operatorname{H}^1(\Sigma,\ad^{par}(E)) \rar &\operatorname{H}^{1}(\Sigma,\ad^{spar}(E)\otimes\omega_{\Sigma}(x)) 
\end{tikzcd}
\]

Since we have $\ad^{par}(E)^\vee \cong \ad^{spar}(E)\otimes \sO_{\Sigma}(x)$, Serre duality induces a skew-symmetric, non-degenerate pairing on $\mathbb{H}^1(\mathscr{F}^{\bullet}_{\bf{E}})$. Determinant of this pairing defines a natural gauge form on the moduli space $\mathbf{Higgs}_{P_{\bullet}}$, which we denote by $\omega_{\bullet}$.

We now restrict our attention to the “$\KL$” locus of the moduli space. Consider the following commutative diagram:
\[
\begin{tikzcd}
   \bf{Higgs}^{\KL}_{P_B}\ar[rd,"h_{P_B}"']\ar[rr, "p_{BC}"]&&\bf{Higgs}^{\KL}_{P_C}\ar[ld,"h_{P_C}"]\\
   &\bf{H}^{\KL}
\end{tikzcd}.
\]

The fiber $h_{P_{\bullet}}^{-1}(a)$ is a torsor of a finite cover of the Prym variety. Thus, for each $a\in\bf{H}^{\KL}$, we denote elements in $h_{P_{\bullet}}^{-1}(a)$ by a pair $(\sL, \bullet_{\sL})$, where $\sL \in \Prym(\sR_a)$ and $\bullet_{\sL}$ records a point of the fiber of $h_{P_{\bullet}}^{-1}(a)\rightarrow \Prym(\sR_a)$ at $\sL$.

Our goal is to describe the tangent space to $h_{P_{\bullet}}^{-1}(a)$ at the point $(\sL, \bullet_{\sL})$.

\begin{lemma}\label{structure sheaf of the normalized spectral curve}
    There exists a subsheaf $\overline{\sF} \subset \overline{\pi}_{a*}\sO_{\overline{\Sigma}_a}$ such that the tangent space to $h_{P_{\bullet}}^{-1}(a)$ at $(\mathcal{L}, \bullet_{\sL})$ is isomorphic to $\operatorname{H}^1(\Sigma, \overline{\sF})$, and moreover:
    \[
    \overline{\sF}^\vee \cong \bigoplus_{i=1}^n \omega_{\Sigma}^{\otimes 2i-1}((2i - \eta_{2i})x).
    \]
\end{lemma}
\begin{proof}
    Notice that we can identify tangent spaces of Hitchin fibers with that of $\Prym_{a}$ which is defined as:
    \[
    1\rightarrow \Prym_a\rightarrow \Jac(\overline{\Sigma}_a)\xrightarrow{\Nm}\Jac(\overline{\Sigma}_a/\sigma)\rightarrow 1
    \]
    where $\sigma$ is the natural involution on the normalized spectral curve $\overline{\Sigma}_a$. We first put the following commutative diagram:
    \[
    \begin{tikzcd}
    \overline{\Sigma}_a \arrow[rr,"\overline{p}"] \arrow[d] \arrow[rdd, dashed, "\bar{\pi}_a" near start] &    & \overline{\Sigma}_a/\sigma\arrow[d,"N"] \arrow[ldd, dashed,"\bar{\pi}'_a"' near start] \\
    \Sigma_{a} \arrow[rd,"\pi_a",swap] \arrow[rr,"p"]                    &    & \Sigma_a/\sigma \arrow[ld,"\pi'_a"]                    \\
    & \Sigma &      
    \end{tikzcd}
    \]

    We define $\overline{\sF}$ fitting into the following exact sequence:
    \[
    \begin{tikzcd}       0\ar[r]&\overline{\sF}\ar[r]&\overline{\pi}_{a*}\sO_{\overline{\Sigma}_a}\ar[r]&\bar{\pi}'_{a*}\sO_{\overline{\Sigma}_a/\sigma}\ar[r]&0\\
    0\ar[r]&\sF\ar[r]\ar[u,hook]&\pi_{a*}\sO_{\Sigma_a}\ar[r]\ar[u,hook]&\pi'_{a*}\sO_{\Sigma_a/\sigma}\ar[u,hook]\ar[r]&0\\
    \end{tikzcd}
    \]
    In particular, we have:
    \[
    \operatorname{H}^1(\Sigma,\overline{\sF})\cong T_e\Prym_a.
    \]
    Applying the functor $\sH om(-, \omega_\Sigma)$ and using Grothendieck–-Serre duality, we obtain:

    \[
     \begin{tikzcd}       0\ar[r]&\bar{\pi}'_{a*}\omega_{\overline{\Sigma}_a/\sigma}\ar[r]\ar[d,hook]&\overline{\pi}_{a*}\omega_{\overline{\Sigma}_a}\ar[r]\ar[d,hook]&\overline{\sF}^{\vee}\otimes\omega_{\Sigma}\ar[d,hook]\ar[r]&0\\
    0\ar[r]&\oplus_{i=0}^{n-1}\omega_{\Sigma}^{2i+1}(2ix)\ar[r]&\oplus_{i=1}^{2n}\omega_{\Sigma}^{i}((i-1)x)\ar[r]&\sF^{\vee}\otimes\omega_{\Sigma}\ar[r]&0\\
    \end{tikzcd}
    \]

    Our goal is to show that:
    \[
    \overline{\sF}^{\vee}\otimes\omega_{\Sigma}=\oplus_{i=1}^{n}\omega_{\Sigma}^{2i}((2i-\eta_{2i})x).
    \]
     Notice that we have the natural perfect pairing as follows:
    \begin{equation}\label{eq:local GS duality}
    \begin{tikzcd}       \overline{\pi}_{a*}\omega_{\overline{\Sigma}_a}\ar[d,hook]&\times& \overline{\pi}_{a*}\sO_{\overline{\Sigma}_a}\\       \oplus_{i=1}^{2n}\omega_{\Sigma}^{i}((i-1)x)&\times& \oplus_{i=0}^{2n-1}\omega_{\Sigma}^{-i}(-ix)\ar[r,"\Tr"]\ar[u,hook]&\omega_{\Sigma}
    \end{tikzcd}
    \end{equation}
    As a result, it suffices to analyze locally around the marked point $x\in\Sigma$. We claim that the image of $\overline{\pi}_{a*}\omega_{\overline{\Sigma}_a}$ is contained in $\oplus_{i=1}^{2n}\omega_{\Sigma}^{i}((i-\eta_i)x)$. Then we have  $\overline{\sF}^{\vee}\otimes\omega_{\Sigma}\hookrightarrow\oplus_{i=1}^{n}\omega_{\Sigma}^{2i}((2i-\eta_{2i})x)$. By Corollary \ref{half dimension}, the inclusion is an isomorphism. By the perfectness of horizaontal pairings, it suffices to show that the embedding $ \pi_{a*}\sO_{\Sigma_a}=\oplus_{i=0}^{2n-1}\omega^{-i}_{\Sigma}(-ix)\hookrightarrow\overline{\pi}_{a*}\sO_{\overline{\Sigma}_a}$ factor through $ \oplus_{i=0}^{2n-1}\omega^{-i}_{\Sigma}(-(i+\eta_{i+1}-1)x)\hookrightarrow\overline{\pi}_{a*}\sO_{\overline{\Sigma}_a}$

    If we fix a choice of local generators of $\omega_{\Sigma}(x)$, the embedding $ \pi_{a*}\sO_{\Sigma_a}\hookrightarrow\overline{\pi}_{a*}\sO_{\overline{\Sigma}_a}$ is
    \[
    \sO[\lambda]/f(\lambda)\hookrightarrow \oplus_{i=1}^{|\bf{d}|} \sO[\lambda_i]/f_i(\lambda_i),\quad \lambda\mapsto (\lambda_1,\ldots,\lambda_{|\bf{d}|}).
    \]
    Here $\lambda$ (reps. $\lambda_i, 1\le i\le |\bf{d}|$) is treated as an $\sO$-linear map on the free module $\sO[\lambda]/f(\lambda)$ (resp $\sO[\lambda_i]/f_i(\lambda_i)$).
    A direct calculation shows that $\frac{\lambda^{j}}{t^{\eta_{j+1}-1}}$ is a well defined $\sO$-linear map on free modules $\sO[\lambda_i]/f_i(\lambda_i),1\le i\le |\bf{d}|$. 

    Hence the image of $\overline{\pi}_{a*}\omega_{\overline{\Sigma}_a}$ is contained in $\oplus_{i=1}^{2n}\omega_{\Sigma}^{i}((i-\eta_i)x)$. And our conclusion follows.
\end{proof}

By the smoothness of Hitchin maps on $\bf{H}^{\KL}$, we have:
\[
0\rightarrow h_{P_B}^*\Omega^1_{\bf{H}^{\KL}}\rightarrow\Omega^1_{\bf{Higgs}^{\KL}_{P_B}}\rightarrow \Omega^1_{\bf{Higgs}^{\KL}_{P_B}/\bf{H}^{\KL}}\rightarrow 0
\]
and similarly on C-side. 

\begin{proposition}\label{prop:comp w KL}
    There exist translation-invariant symplectic forms $\omega^{\KL}_{\bullet}$ on $\bf{Higgs}_{P_\bullet}^{\KL}$. Moreover, these satisfy the compatibility relation $p_{BC}^*\omega_C^{\KL}=\omega_B^{\KL}$.
\end{proposition}
\begin{proof}
    The existence follows from a relative verison of Lemma \ref{structure sheaf of the normalized spectral curve} over $\bf{H}^{\KL}$ and Serre duality over $\Sigma$. To be more precise,  we have the isomorphism:
    \[                          h_{P_{\bullet}}^*h_{P_{\bullet}*}\Omega^1_{\bf{Higgs}^{\KL}_{P_{\bullet}}/\bf{H}^{\KL}}\cong \Omega^1_{\bf{Higgs}^{\KL}_{P_{\bullet}}/\bf{H}^{\KL}}.
    \]
    As in Proposition \ref{prop:generic fiber is trivial torsor} and Lemma \ref{lem:trivial torsor for para B}, we show that both are trivial torsors.  By Lemma \ref{structure sheaf of the normalized spectral curve}, there is a natural section of $\Omega^1_{\bf{H}}\wedge h_{P_{\bullet}}^*h_{P_{\bullet}*}\Omega^1_{\bf{Higgs}^{\KL}_{P_{\bullet}}/\bf{H}^{\KL}}$ up to constant. 
    
    We denote the resulting 2-forms by $\omega_B^{\KL}$ and $\omega_C^{\KL}$, respectively. These forms are translation-invariant by construction, and the compatibility under $p_{BC}^*$ follows immediately.
\end{proof}   

We now compare $\omega^{\KL}_\bullet$ with the restriction of the symplectic form $\omega_\bullet$ to $\mathbf{Higgs}^{\KL}_{P_\bullet}$.

Fix a point $a\in \bf{H}^{\KL}$, and let $\bf{E}_0=(E_0, \theta_0, \cdots)\cong \bar{\pi}_{a*}((\sL_0, \bullet_{\sL_0}))\in h_{P_\bullet}^{-1}(a)$. As discussed previously, the tangent space to $\bf{Higgs}_{P_{\bullet}}$ at $\bf{E}_0$ is given by $\mathbb{H}^1(\mathscr{F}^{\bullet}_{\bf{E}_0})$. Meanwhile, viewing $\bf{E}_0\cong \bar{\pi}_{a*}((\sL, \bullet_{\sL}))$ as a point in the abelian torsor fibration $\bf{Higgs}_{P_{\bullet}}^{\KL}\rightarrow \bf{H}^{\KL}$, then we have the following exact sequence: 
\[
 \begin{tikzcd}
     & \operatorname{H}^1(\Sigma, \overline{\sF}) \ar[d]\ar[rd, "\Psi"] & \\
     \operatorname{H}^0(\Sigma,\ad^{spar}(E_0)\otimes\omega_{\Sigma}(x)) \ar[r] \ar[rd, "\Phi"] &\mathbb{H}^1(\mathscr{F}^{\bullet}_{\bf{E}_0}) \ar[r] \ar[d] & \operatorname{H}^{1}(\Sigma, \ad^{par}(E_0))\\
      & \bf{H}& 
 \end{tikzcd}.
\] 

We identify $\mathbf{H}$ with the tangent space $T_a \mathbf{H}$ at the point $a \in \mathbf{H}^{\KL}$. As discussed earlier, the horizontal exact sequence defines a nondegenerate pairing on the hypercohomology group $\mathbb{H}^1(\mathscr{F}^{\bullet}_{\bf{E}_0})$.

On the other hand, by Lemma~\ref{structure sheaf of the normalized spectral curve}, we have an isomorphism $\operatorname{H}^1(\Sigma, \overline{\sF})^\vee \cong \bf{H}$, so the vertical exact sequence also induces a nondegenerate pairing on $\mathbb{H}^1(\mathscr{F}^{\bullet}_{\bf{E}_0})$.

Our goal is to compare these two pairings—one arising from the horizontal sequence and the other from the vertical sequence (via the spectral construction)—and to show that they coincide.

\begin{proposition}\label{commutative diagram for gauge forms}
    There exists a linear isomorphism $\Delta: \mathbf{H} \rightarrow \mathbf{H}$, independent of the choice of $a \in \mathbf{H}^{\KL}$, such that the following diagram commutes:
    \[
    \begin{tikzcd}
        \operatorname{H}^0(\Sigma, \ad^{spar}(E_0) \otimes \omega_{\Sigma}(x)) \ar[r, "\Delta \circ \Phi_{\bullet}"] \ar[d] & \mathbf{H} \ar[d] \\
        \operatorname{H}^1(\Sigma, \ad^{par}(E_0))^\vee \ar[r, "\Psi_{\bullet}^\vee"] & \operatorname{H}^1(\Sigma, \overline{\mathcal{F}})^\vee
    \end{tikzcd}
    \]
    where the vertical morphisms are given by Serre duality.
\end{proposition}

\begin{proof}
    We first consider the morphism $\Psi_{\bullet}: \operatorname{H}^{1}(\Sigma, \overline{\sF})\rightarrow \operatorname{H}^{1}(\Sigma, \ad^{par}(E_0))$. Since $(E_0, \theta_0)\cong \bar{\pi}_{a*}((\sL_0, \bullet_{\sL_0}))$, this map is given by sending an infinitesimal deformation of $(\sL_0, \bullet_{\sL_0})$ to one of $(E_0, \theta_0)$. We claim that $\Psi$ can be induced by a sheaf homomorphism $\tilde{\Psi}_{\bullet}: \overline{\sF}\rightarrow \ad^{par}(E_0)$. 
    
    For the case $\bullet=C$, we have $\bar{\pi}_{a*}\sL_0=E_0$, hence $E_0$ admits a natural $\bar{\pi}_{a*}\sO_{\overline{\Sigma}_a}$ module structure, this gives the morphism $\tilde{\Psi}_C: \overline{\sF}\rightarrow \ad^{par}(E_0)$. For the case $\bullet=B$, we have $\bar{\pi}_{a*}\sL_0\neq E_0$. As in the second part of the proof of Proposition \ref{L_BC}, we have a line bundle $\ker \theta_0\cong \omega_{\Sigma}^{-n}((\delta-n)x)$, so that $\bar{\pi}_{a*}\sL_0\oplus \ker \theta_0$ is a subsheaf of $E_0$ and the quotient of $E_0$ by $\bar{\pi}_{a*}\sL_0\oplus \ker \theta_0$ supports only on $x$. Notice that we take $(\sL_0, B_{\sL_0})\in (h_{P_B}^{-1}(a))$, the choice of $B_{\sL_0}$ ensures that we can lift the $\bar{\pi}_{a*}\sO_{\overline{\Sigma}_a}$ module structure on $\bar{\pi}_{a*}\sL_0$ to $E_0$, which is also pass to a sheaf homomorphism $\tilde{\Psi}_B: \overline{\sF}\rightarrow \ad^{par}(E_0)$ since sections in $\ad^{par}(E_0)$ are traceless. In both cases, if we use $\lambda$ to denote a local generator of $\overline{\sF}$, then the morphism $\tilde{\Psi}_{\bullet}$ sends $\lambda^{2i-1}$ to $\theta_0^{2i-1}$. 

    On the other hand, $\Phi_{\bullet}$ is the tangent map of $\operatorname{H}^{0}(\Sigma, \ad^{spar}(E_0)\otimes \omega_{\Sigma}(x))\rightarrow \bf{H}$ that sends $\theta$ to $(\operatorname{Tr}(\wedge^{2i}\theta))_{1\leq i \leq n}$ at $\theta_0$. We consider a homomorphism $$N_{\bullet}:\ad^{spar}(E_0)\otimes \omega_{\Sigma}(x)\longrightarrow \oplus_{i=1}^n\omega_{\Sigma}^{\otimes 2i}((2i-\eta_{2i}))x)$$ sending $\theta$ to $(2i\operatorname{Tr}(\theta_0^{2i-1}\theta))_{1\leq i \leq n}$.Then $\operatorname{H}^{0}(N_{\bullet})$ is the tangent map of $\operatorname{H}^{0}(\Sigma, \ad^{spar}(E_0)\otimes \omega_{\Sigma}(x))\rightarrow \bf{H}$ which sending $\theta$ to $(\operatorname{Tr}(\theta^{2i}))_{1\leq i \leq n}$ at $\theta_0$. By the relations between $(\operatorname{Tr}(\theta^{2i}))_{1\leq i \leq n}$ and $(\operatorname{Tr}(\bigwedge^{2i}\theta))_{1\leq i \leq n}$, there is an automorphism of $\bf{H}$, see \cite{Lew94} or \cite{KK92}. Consider the differential map of this automorphism at $a$, we have a linear isomorphism $\Delta_{1}:\bf{H}\rightarrow \bf{H}$ so that $\Delta_{1} \circ\operatorname{H}^{0}(N_{\bullet})=\Phi_{\bullet}$. 

    Now we take dual of $N_{\bullet}$ and then tensor with $\omega_{\Sigma}$, then use the pairing between $\ad^{par}(E_0)$ and $\ad^{spar}(E_0)\otimes \sO_{\Sigma}(x)$ we get $$N_{\bullet}^{\vee}: \oplus_{i=1}^n\omega_{\Sigma}^{\otimes 1-2i}((\eta_{2i}-2i)x)\longrightarrow \ad^{par}(E_0).$$  Now by Lemma \ref{structure sheaf of the normalized spectral curve}, we see that we have a morphism $$(\times2i)_{1\leq i \leq n}:\bigoplus_{i=1}^n\omega_{\Sigma}^{\otimes 1-2i}((\eta_{2i}-2i)x)\longrightarrow \bigoplus_{i=1}^n\omega_{\Sigma}^{\otimes 1-2i}((\eta_{2i}-2i)x)$$ so that  $N_{\bullet}^{\vee}\circ (\times2i)_{1\leq i \leq n}=\tilde{\Psi}_{\bullet}$. We denote $\Delta=\operatorname{H}^{0}((\times2i)_{1\leq i \leq n})^{-1}\circ \Delta_{1}^{-1}$, then if we apply Serre duality to $\overline{\sF}$ and $\ad^{par}(E_0)$, by the functoriality of the Serre duality, we have the commutative diagram.   
\end{proof}

Together with Propositions~\ref{prop:comp w KL} and~\ref{commutative diagram for gauge forms}, we obtain the following:

\begin{corollary}\label{cor:caomparison of symp forms}
    The restriction of the gauge-theoretic symplectic form $\omega_B$ to the $\KL$-locus satisfies
    \[
    \omega_{B}\big|_{\mathbf{Higgs}_{P_{B}}^{\KL}} = p_{BC}^*\left( \omega_{C} \big|_{\mathbf{Higgs}_{P_{C}}^{\KL}} \right).
    \]
\end{corollary}

\subsection{Relative Setting}

To apply the $p$-adic machinery developed by Groechenig--Wyss--Ziegler \cite{GWZ20}, we need to work over a finitely generated $\mathbb{Z}$-algebra $R \subset \mathbb{C}$. 

All isogenies and the constructions of gauge forms are defined algebraically—for instance, the resolution of planar singularities via Kazhdan--Lusztig maps, the construction of $\theta$-direct summands, the polarization of relative Jacobians, and Serre duality. 

In addition, to obtain splittings and rational points, we require the existence of a square root of the canonical bundle $\omega_\Sigma$ defined over $R$. This condition depends only on the base curve $\Sigma$ and can be satisfied within the chosen framework.

Therefore, all relevant structures and constructions can be organized into a family defined over a finitely generated $\mathbb{Z}$-algebra $R \subset \mathbb{C}$.

\subsection{Proof of the TMS}Using the $p$-adic integration technique from \cite[\S 6]{GWZ20}, we deduce the following equalities of stringy E-polynomials twisted by gerbes. See \cite[\S 2]{GWZ20} for details about their definitions and derivations.

\begin{theorem}\label{Thm:TMS}
    Under Condition~\ref{cond:coprime}, topological mirror symmetry holds for the Langlands dual parabolic Hitchin systems:
	\begin{equation*}
		E^{\alpha_C}(\bf{Higgs}_{P_C};u,v)=E(\bf{Higgs}_{P_C};u,v)=E_{\text{st}}(\bf{Higgs}_{P_B}^{-};u,v)=E_{\text{st}}(\bf{Higgs}_{P_B}^{+};u,v).       
	\end{equation*}  
\end{theorem}
\begin{proof}
We verify that the Hitchin systems satisfy the conditions of ``weak abstract dual Hitchin systems'' by Shen\cite[\S 3.1]{She24} which is modified for Langlands dual parabolic Hitchin systems from the ``abstract dual Hitchin systems" by Groechenig--Wyss--Ziegler \cite[\S 6]{GWZ20}.
\begin{enumerate}
	\item A pair of Hitchin systems:
	\[\small
	\begin{tikzcd}[column sep= 0.0 em]
		\bf{Higgs}_{P_C}\ar[rd,"h_{P_C}", swap]\supset \bf{Higgs}_{P_C}|_{\bf{H^{\text{KL}}}}&& \bf{Higgs}_{P_B}\ar[ld,"h_{P_B}"]\supset \bf{Higgs}_{P_B}|_{\bf{H^{\text{KL}}}}\\
		&\bf{H}\supset\bf{H^{\text{KL}}}
	\end{tikzcd}.
	\]
	\item Arithmetic duality, see Theorem \ref{thm:syz mirror}. And the self-dual isogeny between generic fibers has been explained in Section \ref{sec:self dual}.
	\item The relation between ``the existence of rational points over p-adic fields'' and ``trivialization of gerbes'' as in \cite{GWZ20} follows from Proposition \ref{prop:generic fiber is trivial torsor}.
	\item Via the symplectic forms $\omega_B,\omega_C$, we have $\det\omega_B,\det\omega_C$ as gauge forms. 
    As our gerbes $\alpha_C,\alpha_B$ always split by the existence of rational points over $F$, we only need to check the equality:
    \[
    \int_{\bf{Higgs}_{P_B}(\sO_F)\cap \bf{Higgs}_{P_B}^{\KL}(F)}g_B= \int_{\bf{Higgs}_{P_C}(\sO_F)\cap \bf{Higgs}_{P_C}^{\KL}(F)}g_C
    \]
    This follows from $\omega_{B}|_{\bf{Higgs}_{P_B}^{\KL}}=p_{BC}^*\omega_{C}|_{\bf{Higgs}_{P_B}^{\KL}}$ in Corollary \ref{cor:caomparison of symp forms} since gauge forms $g_B,g_C$ are determinant of the chosen symplectic forms $\omega_B,\omega_C$.
\end{enumerate} 
\end{proof}

\bibliographystyle{alpha}
\bibliography{ref}

\begin{thebibliography}{BBAMY25}

\bibitem[BBAMY25]{BBAMY}
Roman Bezrukavnikov, Pablo Boixeda~Alvarez, Michael McBreen, and Zhiwei Yun.
\newblock Non-abelian {H}odge moduli spaces and homogeneous affine {S}pringer fibers.
\newblock {\em Pure Appl. Math. Q.}, 21(1):61--130, 2025.

\bibitem[BK18]{BK18}
David Baraglia and Masoud Kamgarpour.
\newblock On the image of the parabolic {H}itchin map.
\newblock {\em Quarterly Journal of Mathematics}, 69(2):681--708, 2018.

\bibitem[BKV18]{BKV18}
David Baraglia, Masoud Kamgarpour, and Rohith Varma.
\newblock Complete integrability of the parahoric hitchin system.
\newblock {\em International Mathematics Research Notices}, 2019(21):6499--6528, 01 2018.

\bibitem[CM93]{CM93}
David~H. Collingwood and William~M. McGovern.
\newblock {\em Nilpotent orbits in semisimple {L}ie algebras}.
\newblock Van Nostrand Reinhold Mathematics Series. Van Nostrand Reinhold Co., New York, 1993.

\bibitem[CR04]{CR04}
Weimin Chen and Yongbin Ruan.
\newblock A new cohomology theory of orbifold.
\newblock {\em Communications in Mathematical Physics}, 248(1):1--31, 2004.

\bibitem[dCMS22a]{CMS1}
Mark~Andrea de~Cataldo, Davesh Maulik, and Junliang Shen.
\newblock Hitchin fibrations, abelian surfaces, and the {$P=W$} conjecture.
\newblock {\em J. Amer. Math. Soc.}, 35(3):911--953, 2022.

\bibitem[dCMS22b]{CMS2}
Mark~Andrea de~Cataldo, Davesh Maulik, and Junliang Shen.
\newblock On the {$P=W$} conjecture for {${\rm SL}_n$}.
\newblock {\em Selecta Math. (N.S.)}, 28(5):Paper No. 90, 21, 2022.

\bibitem[DP09]{DP08}
R.~Donagi and T.~Pantev.
\newblock Geometric {L}anglands and non-abelian {H}odge theory.
\newblock In {\em Surveys in differential geometry. {V}ol. {XIII}. {G}eometry, analysis, and algebraic geometry: forty years of the {J}ournal of {D}ifferential {G}eometry}, volume~13 of {\em Surv. Differ. Geom.}, pages 85--116. Int. Press, Somerville, MA, 2009.

\bibitem[DP12]{DP12}
R.~Donagi and T.~Pantev.
\newblock Langlands duality for {H}itchin systems.
\newblock {\em Invent. Math.}, 189(3):653--735, 2012.

\bibitem[FRW24]{FRW24}
Baohua Fu, Yongbin Ruan, and Yaoxiong Wen.
\newblock Mirror symmetry for special nilpotent orbit closures.
\newblock {\em Sci China Math}, 67, 2024.

\bibitem[GW08]{GW08}
Sergei Gukov and Edward Witten.
\newblock Gauge theory, ramification, and the geometric {L}anglands program.
\newblock {\em Current Developments in Mathematics}, 2006:146, 02 2008.

\bibitem[GW10]{GW10}
Sergei Gukov and Edward Witten.
\newblock {Rigid surface operators}.
\newblock {\em Adv. Theor. Math. Phys.}, 14(1):87--178, 2010.

\bibitem[GWZ20]{GWZ20}
Michael Groechenig, Dimitri Wyss, and Paul Ziegler.
\newblock Mirror symmetry for moduli spaces of {H}iggs bundles via p-adic integration.
\newblock {\em Invent. Math.}, 221(2):505--596, 2020.

\bibitem[Hes78]{He78}
Wim~H Hesselink.
\newblock Polarizations in the classical groups.
\newblock {\em Mathematische Zeitschrift}, 160(3):217--234, 1978.

\bibitem[{Hit}07]{Hit07}
Nigel {Hitchin}.
\newblock {Langlands duality and {$G_{2}$} spectral curves}.
\newblock {\em {Q. J. Math.}}, 58(3):319--344, 2007.

\bibitem[HMMS22]{HMMS}
Tamas {Hausel}, Anton {Mellit}, Alexandre {Minets}, and Olivier {Schiffmann}.
\newblock {$P=W$ via $H_2$}.
\newblock {\em arXiv e-prints}, page arXiv:2209.05429, September 2022.

\bibitem[How96]{Howe96}
Everett~W. Howe.
\newblock The {W}eil pairing and the {H}ilbert symbol.
\newblock {\em Math. Ann.}, 305(2):387--392, 1996.

\bibitem[HT03]{HT03}
Tam\'{a}s Hausel and Michael Thaddeus.
\newblock Mirror symmetry, {L}anglands duality, and the {H}itchin system.
\newblock {\em Invent. Math.}, 153(1):197--229, 2003.

\bibitem[KK92]{KK92}
LA~Kondratyuk and MI~Krivoruchenko.
\newblock Superconducting quark matter in su (2) colour group.
\newblock {\em Zeitschrift f{\"u}r Physik A Hadrons and Nuclei}, 344:99--115, 1992.

\bibitem[KL88]{KL88}
D.~Kazhdan and G.~Lusztig.
\newblock Fixed point varieties on affine flag manifolds.
\newblock {\em Israel J. Math.}, 62(2):129--168, 1988.

\bibitem[Kon95]{Ko95}
Maxim Kontsevich.
\newblock Homological algebra of mirror symmetry.
\newblock In {\em Proceedings of the {I}nternational {C}ongress of {M}athematicians, {V}ol. 1, 2 ({Z}\"{u}rich, 1994)}, pages 120--139. Birkh\"{a}user, Basel, 1995.

\bibitem[KP89]{KP89}
Hanspeter Kraft and Claudio Procesi.
\newblock A special decomposition of the nilpotent cone of a classical {L}ie algebra.
\newblock {\em Ast{\'e}risque}, 173(174):10, 1989.

\bibitem[KW07]{KW07}
Anton Kapustin and Edward Witten.
\newblock {Electric-Magnetic duality and the geometric Langlands program}.
\newblock {\em Commun. Num. Theor. Phys.}, 1:1--236, 2007.

\bibitem[Lew94]{Lew94}
Mordechai Lewin.
\newblock On the coefficients of the characteristic polynomial of a matrix.
\newblock {\em Discrete Mathematics}, 125(1-3):255--262, 1994.

\bibitem[Lus79]{Lus79}
G.~Lusztig.
\newblock A class of irreducible representations of a {W}eyl group.
\newblock {\em Indagationes Mathematicae (Proceedings)}, 82(3):323--335, 1979.

\bibitem[MS21]{MS21}
Davesh Maulik and Junliang Shen.
\newblock Endoscopic decompositions and the {H}ausel-{T}haddeus conjecture.
\newblock {\em Forum Math. Pi}, 9:Paper No. e8, 49, 2021.

\bibitem[MS24]{MS22}
Davesh Maulik and Junliang Shen.
\newblock The {$P=W$} conjecture for {${\rm GL}_n$}.
\newblock {\em Ann. of Math. (2)}, 200(2):529--556, 2024.

\bibitem[MSY25]{MSY23}
Davesh Maulik, Junliang Shen, and Qizheng Yin.
\newblock Perverse filtrations and {F}ourier transforms.
\newblock {\em Acta Math.}, 234(1):1--69, 2025.

\bibitem[Mum71]{Mum71}
David Mumford.
\newblock Theta characteristics of an algebraic curve.
\newblock {\em Annales scientifiques de l'École Normale Supérieure}, 4(2):181--192, 1971.

\bibitem[Ng{\^o}10]{Ngo10}
Bao~Ch{\^a}u Ng{\^o}.
\newblock Le lemme fondamental pour les algebres de {L}ie.
\newblock {\em Publ. Math. Inst. Hautes Etudes Sci.}, (111):1--169, 2010.

\bibitem[Rua03]{Ru03}
Yongbin Ruan.
\newblock Discrete torsion and twisted orbifold cohomology.
\newblock {\em Journal of Symplectic Geometry}, 2:1--24, 2003.

\bibitem[She24a]{She24}
Shiyu Shen.
\newblock Mirror symmetry for parabolic {H}iggs bundles via {$p$}-adic integration.
\newblock {\em Adv. Math.}, 443:Paper No. 109616, 40, 2024.

\bibitem[She24b]{She18}
Shiyu Shen.
\newblock Tamely ramified geometric langlands correspondence in positive characteristic.
\newblock {\em International Mathematics Research Notices}, 2024(7):6176--6208, 2024.

\bibitem[Som01]{so01}
Eric Sommers.
\newblock Lusztig's canonical quotient and generalized duality.
\newblock {\em Journal of Algebra}, 243(2):790--812, 2001.

\bibitem[Spa88]{Spa88}
N.~Spaltenstein.
\newblock {Polynomials over local fields, nilpotent orbits and conjugacy classes in Weyl groups}.
\newblock {\em Astérisque}, 168:191--217, 1988.

\bibitem[Spa90]{Spa90}
N~Spaltenstein.
\newblock On the {Kazhdan-Lusztig} map for exceptional {Lie} algebras.
\newblock {\em Advances in Mathematics}, 83(1):48--74, 1990.

\bibitem[Spa06]{Spa06}
Nicolas Spaltenstein.
\newblock {\em Classes unipotentes et sous-groupes de Borel}, volume 946.
\newblock Springer, 2006.

\bibitem[Spr76]{Spr}
T.~A. Springer.
\newblock Trigonometric sums, {G}reen functions of finite groups and representations of {W}eyl groups.
\newblock {\em Invent. Math.}, (36):173--207, 1976.

\bibitem[SS95]{SS95}
M~Schottenloher and P~Scheinost.
\newblock Metaplectic quantization of the moduli spaces of flat and parabolic bundles.
\newblock 1995.

\bibitem[SWW22a]{SWW22}
Xiaoyu Su, Bin Wang, and Xueqing Wen.
\newblock Parabolic {H}itchin maps and their generic fibers.
\newblock {\em Math. Z.}, 301(1):343--372, 2022.

\bibitem[SWW22b]{SWW22t}
Xiaoyu Su, Bin Wang, and Xueqing Wen.
\newblock Topological mirror symmetry of parabolic {H}itchin systems.
\newblock {\em arXiv preprint arXiv:2206.02527}, 2022.

\bibitem[SXY23]{SXY23}
Peng Shan, Dan Xie, and Wenbin Yan.
\newblock Mirror symmetry for circle compactified 4d $\mathcal{N}= 2$ {SCFTs}.
\newblock {\em arXiv preprint arXiv:2306.15214}, 2023.

\bibitem[Wan23]{Wang23}
Bin Wang.
\newblock On parahoric {H}itchin systems over curves.
\newblock {\em Internat. J. Math.}, 34(13):Paper No. 2350081, 22, 2023.

\bibitem[Yun21]{Yun21}
Zhiwei Yun.
\newblock Minimal reduction type and the {K}azhdan-{L}usztig map.
\newblock {\em Indag. Math. (N.S.)}, 32(6):1240--1274, 2021.

\end{thebibliography}
\end{document}